\documentclass[12pt]{article}
\usepackage{amsfonts,appendix}
\usepackage{titling}
\usepackage{amsthm,amsmath,amsfonts,amssymb,amscd,latexsym,stmaryrd}
\usepackage[vcentermath,enableskew]{youngtab}
\usepackage[mathscr]{eucal}
\usepackage{lscape}
\usepackage{array}
\usepackage{makeidx}
\usepackage{amssymb,multicol}
\usepackage{fancyhdr}
\usepackage{fancybox}
\usepackage{eurosym}
\theoremstyle{definition}
\usepackage{xcolor}
\usepackage{colortbl}
\usepackage[draft]{pdfpages}
\usepackage{textcomp}
\usepackage{graphicx}
\usepackage{tikz}
\usepackage{ytableau}
\usepackage{enumerate}

\usepackage{mathtools}
\usepackage{amsmath,latexsym,amssymb,amsthm,enumerate,amsmath,amscd}
\usepackage[mathscr]{eucal}
\usepackage[utf8]{inputenc}
\usepackage{color}
\usepackage{xy}
\usepackage{comment}
\usepackage{array}
\usepackage{tikz}
\usepackage{subfig}
\usetikzlibrary{positioning}
\usetikzlibrary{calc,shadings,patterns,bending,arrows.meta}
\def\ns{\footnotesize \it}

\theoremstyle{plain}
\newtheorem{theorem}{Theorem}[section]
\newtheorem{proposition}[theorem]{Proposition}

\newtheorem{lemma}[theorem]{Lemma}
\newtheorem{rmk}[theorem]{Remark}
\theoremstyle{definition}
\newtheorem{definition}[theorem]{Definition}

\newtheorem{example}[theorem]{Example}

\newtheorem{thm}{Theorem}

\newtheorem{transl}[theorem]{Note of Translators}

\newcommand{\subtitle}[1]{%
  \posttitle{%
    \par\end{center}
    \begin{center}\large#1\end{center}
    \vskip0.5em}%
    }

\newcommand\blfootnote[1]{%
  \begingroup
  \renewcommand\thefootnote{}\footnote{#1}%
  \addtocounter{footnote}{-1}%
  \endgroup
}

\def\Pgl{\mathrm{PGL}}
\def\Gl{\mathrm{GL}}
\def\Grass{\mathrm{Grass}}
\def\m{\mathfrak{m}}
\def\Hilb{\mathrm{Hilb}}
\def\Hom{\mathrm{Hom}}
\def\Jac{\mathrm{Jac}}
\def\Gor{\mathrm{Gor}}
\def\GL{\mathrm{GL}}
\def\PGL{\mathrm{PGL}}
\def\Stab{\mathrm{Stab}}
\def\SL{\mathrm{SL}}
\def\mod{\mathrm{mod}}

\newcommand{\V}{\mathcal{V}}
\title{Nets of Conics and associated Artinian algebras of length 7}
\subtitle{
Translation and update of the 1977 version by J. Emsalem and A.~Iarrobino\footnote{R\'{e}seaux de Coniques et alg\`{e}bres de longueur sept associ\'{e}es (1977) Univ. Paris VII.}}\normalsize
\author{
	Nancy Abdallah\\[0.05in]{\ns University of Bor\r{a}s, Bor\r{a}s, Sweden.}\\[0.2in]
	Jacques Emsalem\\[0.05in]{\ns Paris, France.}\\[0.2in]
Anthony Iarrobino\\[.05in]
{\ns Department of Mathematics, Northeastern University, Boston, MA 02115,
USA.
}}
\begin{document}
\date{November 22, 2021; revised August 8, 2022}
\maketitle
\par
\newpage
\begin{abstract}
We classify the orbits of nets of conics under the action of the projective linear group and we determine the specializations of these orbits, using geometric and algebraic methods. We study related geometric questions, as the parametrization of planar cubics. We show that Artinian algebras of Hilbert function $ H=(1,3,3,0)$ determined by nets, can be smoothed - deformed to a direct sum of fields; and that algebras of Hilbert function $H=(1,r,2,0)$, determined by pencils of quadrics, can also be smoothed. This portion is a translation and update of a 1977 version, a typescript  by the second two authors that was distributed as a preprint of University of Paris VII.  In a new Historical Appendix A we describe related work prior to 1977. In an Update Appendix B we survey some developments since 1977 concerning nets of conics, related geometry, and deformations of Artinian algebras of small length.
\end{abstract}
 \tableofcontents\noindent
\newpage
\addcontentsline{toc}{subsection}{Translators' Note and 
Acknowledgment}

\subsection*{Translators' Note and Acknowlegment} The original ``R\'{e}seaux de Coniques et alg\`{e}bres de longueur sept associ\'{e}es'', was typed with hand-drawn diagrams, and distributed in 1977 as a preprint of Universit\'{e} de Paris VII, where Jacques Emsalem was Assistant Titulaire. It was drafted in 1975-1976 during a one year visit to \'{E}cole Poytechnique, then in Paris, by Tony Iarrobino under an exchange program of Centre Nationale de Recherche Scientifique (CNRS) and the National Science Foundation (NSF). They were grateful to L\^{e} D\~{u}ng Trang, Monique Lejeune-Jalabert, and Bernard Teissier, who sponsored this visit. In particular L\^{e} realized that Jacques and Tony were working on the same problem from different viewpoints, and this led to their collaboration. \par
The original goal of this classification of nets of planar conics was to establish the classification of the related graded commutative Artinian algebras $A$ of length~7, having Hilbert function $H(A)=(1,3,3)$. Jacques Emsalem had independently determined the classification up to isomorphism of those having length at most 6, and understood that $H(A)=(1,3,3)$ or $(1,4,2)$ were the smallest  with moduli. See the historical note for further comment (Appendix~A).
\par
We should like to thank Ivan Cheltsov for the encouragement to write this translation. He was interested in a portion of the classification, relevant to the divisor $X$ in $\mathbb{P}^2\times \mathbb{P}^2$ of degree $ (1,2)$ (which is the same thing as a net of conics),  and showed independently that if $X$ is smooth, then $X $ is in one of the three classes in the first line of Table \ref{table1}. Here $X$ is a Fano threefold No. 2-24 in \cite{Fan} and the result helps determine which $X$ in the family are Kahler-Einstein \cite{Ch1,AC-V}. We should like to thank also Sandro Verra, who had a copy of the preprint, and informed Ivan of it - and we used the pdf file Sandro sent!  We thank Aldo Conca, Steve Kleiman, Hal Schenck, Igor Dolgachev, and Joachim Jelisiejew for helpful comments. We appreciate the detailed comments of a referee.\par
Nancy and Tony made the first draft of the translation, then corrected with Jacques. In making the translation, we have in places updated or clarified the wording; we have introduced main results as Theorems in the Introduction. We have corrected some mistakes, mostly typos in the original (in particular the top right entry of Table \ref{table1}).  We have added a few translator notes when we felt it would help the reader. The many figures in the original were hand-drawn by Jacques; we have used tikz and made new figures. 
\par
We have added a Historical Appendix A to give some further references to work earlier than 1977 on nets of planar conics and Artinian algebras having small length.\par
We have added an Appendix B where we discuss selectively some developments since 1977 related to nets of conics, small-length Artinian algebras, and possible directions for future research.\par
We have retained the original seven references, but added a separate Bibliography, 2021, to include further references, in particular for this note and the Appendices.\par \qquad\qquad {\it Nancy Abdallah, Jacques Emsalem\footnote{{Jacques Emsalem, friend and colleague died on July 12, 2022, after submission of the completed MS. Nancy and Tony made final corrections in response to referee comments. We have greatly appreciated the opportunity to work together with Jacques -N.A and A.I.}}, and Tony Iarrobino}.

\newpage
\addcontentsline{toc}{subsection}{Notation.}
\subsection*{Notation}
Here ${\sf k}$ is an algebraically closed field of characteristic zero.\footnote{{\it Note of Translators:} The classification results for pencils of ternary conics, nets of ternary conics, ternary cubics, and length 7 algebras, except those involving Hessians,  extend to algebraically closed fields of characteristic $p\not=2$ or $3$. For duality results one needs to replace differentiation by the contraction pairing from $S={\sf k}[x,y,z]$ to the divided power ring related to $R$ \cite{Em}. See page \pageref{A1thm}.} We let $R={\sf k}[[X,Y,Z]]$ the formal power series over $\sf k$. Here $R_i$ denotes the vector subspace of $R$ constituted by $0$ and the homogeneous polynomials of degree $i$.\par
One denotes by $P$ the projective plane whose generic point admits $(X,Y,Z)$ as system of homogeneous coordinates. If $V$ is the dual vector space to $R_1$ it amounts to the same to say that $P$ is the projective space associated to $V$.\par
As $V$ is a vector space of dimension 3, we denote by $\Gl(3)$ the group of linear automorphisms of $V$. Thus, $\Gl(3)$ has a canonical action on each $R_i$. We denote by $Pgl(3)$ the quotient $\Gl(3)/{\sf k}^\ast$ of $\Gl(3)$ by the group of non-null homotheties of $V$ with center the origin (scalings). 
One denotes by $P^\ast$ the projective space associated to the vector space $R_1$. It is what we classically call the dual of $P$ and it can be canonically identified with the set of lines of $P$.\par
More generally, one may consider for each $i$ the projective space associated to $R_i$, which is canonically identified with the set of projective curves of $\mathbb P^2$ having degree $i$. Here, we consider primarily the projective space associated to $R_2$, which we will denote $C_5 $ ($C$ for conic, the index $5$ indicating the dimension of the space).  If $F\in R_i$ we will denote by $(F)$ the corresponding curve, which we will term indifferently the curve of equation $F=0$ or the curve $(F)$.\par
The group $Pgl(3)$ has a natural action on $P,P^\ast,C_5$, and, more generally on the projective space associated to $R_i$ for each $i$, and, more generally still, on the $k$-th Grassmannian $\Grass(k,R_i)$ of $R_i$. \par
One terms \emph{pencil of conics}, respectively \emph{net of conics} a line (respectively, a plane) of $C_5$. More generally, one terms \emph{pencil of hypersurfaces of degree $n$} a line in the projective space associated to $R_n$.
\newpage
\section{Introduction}\label{introsec}
We let ${\sf k}$ be an algebraically closed field of characteristic zero, recall $R={\sf k}[[X,Y,Z]]$ and $R_i$ the vector space of degree-$i$ forms in $R$. The group $\Gl(3)$ acts on $R_i$ and its quotient $Pgl(3)=\Gl(3)/{\sf k}^\ast$ acts on the set of vector subspaces of  $R_i$: in particular if $V=\langle v_1,v_2,v_3\rangle$ is a subspace of $R_2$ of dimension 3 (the case that interests us) and $\sigma\in\Gl(3)$ is a representative of $\sigma'\in Pgl(3)$ then $ \sigma'(V)=\langle\sigma(v_1),\sigma(v_2),\sigma(v_3)\rangle$. Here, we classify the orbits of the dimension-three subspaces under this natural action of $Pgl(3)$. It is the same as classifying up to isomorphism the ``nets of projective plane conics", or, also, to classify up to isomorphism the commutative local ${\sf k}$-algebras $A$ of Hilbert function\footnote{The \emph{Hilbert function} of a local algebra is the sequence $(n_i)$, where $n_i=\dim \m^i/\m^{i+1}, \m$ being the maximal ideal.  {\it Note of translators}: the original article used ``type'' meaning Hilbert function; we no longer do so, as we reserve ``type'' for the sense in Definition \ref{typedef}).} $(1,3,3),$ so of the form $A={\sf k}[X,Y,Z]/(V,(X,Y,Z)^3)$.
We describe also the closures of the orbits, as subvarieties of the Grassmannian $\Grass(3,R_2)$, parametrizing the vector spaces in $R_2$ having dimension three. This is the problem of ``specialization of orbits".\par
The Artinian algebras that we have just considered have dimension $7$ as vector spaces over ${\sf k}$. This is the smallest dimension for which there appear continuous families of orbits.\footnote{For the algebras of dimension at most $6$, their classification has been obtained by different persons. The most interesting case is treated implicitly here in \S 3. The other cases are simple enough. [{\it Note of translators}: see Appendix A].}\label{smalllength} That such a family occurs in the case of Hilbert function $(1,3,3)$ Artinian algebras can be readily seen: the variety $\Grass(3,R_2)$ has dimension $3(6-3)=9$, but $Pgl(3)$ has dimension $3^2-1=8$. One expects, then, to find a one-parameter family of orbits, that we will describe. An analogous calculation (section 8) shows that there is also a one-parameter family of isomorphism classes of algebras of Hilbert function $(1,4,2)$. We have here the two Hilbert functions for length-$7$ local algebras that are isomorphic to their associated graded algebras, and that have a continuous family of orbits.\par
Returning to our main goal, the algebras of Hilbert function $(1,3,3)$, the one-parameter family is obtained as follows: let $\phi$ be a general enough cubic form in $R$; then the vector space $J_\phi$
spanned by the first-order partial derivatives of $\phi$, has dimension $3$, and we will see that ``$\phi\cong \phi'$ (mod $Pgl(3)$)'' is equivalent to $J_\phi=J_\phi'$ (mod $Pgl(3))$". Thus, there is a family of vector spaces $V=J_\phi$ that, like the cubics $\phi$ are parametrized up to isomorphism by the ${\sf j}$-invariant of $\phi$.
We denote by $\Gamma(V)$ the cubic parametrizing the decomposable elements of the vector space $V$: those elements $v$ that can be written as a product of linear forms. We will also denote by $\pi$ the plane of $V$, and by $\Gamma_\pi=\Gamma(V)$ the cubic.\par
Our first main result concerns the nets of conics corresponding to these smooth cubics $\Gamma(V)$ and their
$\sf j$-constant specializations, and as well the nets for which $\Gamma(V)$ is a cubic with a node having distinct tangents. An \emph{elementary} specialization will be from $A$ to $B$ where there is no intermediate specialization $C$ not equal to $A$ or $B$. We will refer to the orbits in Table \ref{table1} by their dimension and
their position from the left, thus \#7b refers to the second orbit in the dimension seven row.
\begin{thm}(Theorem \ref{mainthm}A, \S 6.5 (2), Proposition \ref{elemprop}) The nets of conics corresponding to smooth $\Gamma(V)$ are Jacobian nets $V=J_\phi$ of smooth cubics $\phi$, whose isomorphism classes are determined by the $\sf j$ invariant of $\phi$.  There are elementary $\sf j$-constant specializations of the net $J_\phi$ for each permissible $\sf j$ to the net \#7b $\langle X^2+YZ,XY,Z^2\rangle$ corresponding to a cuspidal cubic $\Gamma(V)$. There are likewise $\sf j$-constant specializations for permissible $\sf j$ to each of the specializations \#6b,6c, 5a,5b,4,2a,2b of \#7b. The specializations of the two nets \#8a,8c whose $\Gamma(V)$ is a cubic having a node with two distinct tangents are as shown in Table \ref{table1}.
\end{thm}
We will use two methods to study other orbits, not corresponding to smooth cubics: the geometric method is more useful. Consider for each vector space $V=\langle v_1,v_2,v_3\rangle$ in $R_2$, the isomorphism class of the cubic $\Gamma(V)$ with equation $f(a_1,a_2,a_3)=0$ which characterizes the singular conics: 	that is, $\Gamma(V)$ parametrizes the elements ${v=a_1v_1+a_2v_2+a_3v_3\in V}$ that are decomposable: $\exists \ell,b\in R_1 \mid v=\ell \cdot b$. That is, if $P(V)$ is the projective plane associated to $V$, we let $\Gamma(V)=P(V)\cap S_4$ where $S_4$ is the cubic hypersurface of degenerate conics (that can be factored).
The geometric method takes into account also the set $D_2$ of ``double lines" (singular conics whose equation is the square of a linear form). We will show, for those orbits for which $\Gamma(V)$ is singular,
\begin{thm}\label{1thm} \begin{enumerate}[(1)] 
\item (Theorem \ref{mainthm}B) Each of the orbits $Pgl(3)\circ V$ for which $\Gamma(V)$ is singular is completely determined by the isomorphism classes of triples of projective algebraic sets $$[D_2\cap P(V),\Gamma(V)=S_4\cap P(V), P(V)],$$ viewed as schemes. 
\item (\S \ref{dimsec}, \S 6.6, Proposition \ref{elemprop}) The orbits, their dimensions, and the specializations of these orbits are as shown in Table~\ref{table1}.
\item (\S 6.1, Proposition \ref{dualprop}, and Note of Translators \ref{dualnote}). The equivalence classes and their specializations satisfy a left right symmetry, that arises from a duality between an orbit and its dual.\par
\end{enumerate}
\end{thm}
An alternative method, used to verify this classification, is more algebraic. It is easy to show, according to the theory of intersections on $\Grass(2,R_2)$, that every $3$-dimensional vector space contains a dimension 2 subspace having a particular form, that is, a vector subspace belonging to the closure of the $Pgl(3)$ orbit of $\langle xy,x^2+yz\rangle$. This method easily reduces the problem of classification to eliminating repetitions in a finite list (Section 5.2).\par
We give without further delay in Table \ref{table1} the table of orbits of nets of conics, and their specializations, reserving for later a detailed explanation. We give for each orbit
\begin{enumerate}[1)]
\item a vector space belonging to it
\item a diagram indicating the isomorphism type of of its associated cubic $\Gamma(V)$, where each point of $\Gamma(V)\cap D_2$ is marked by a circle around it,
\item the dimension of the orbit is indicated by its height in the table.
\end{enumerate}
The arrows oriented downward $V\to V'$ indicate that the orbit of $V'$ is included in the closure of the orbit of $V$.  The arrows downward from the one-parameter family $V_{\sf j}\to V'$ indicate that the orbit of $V'$ is in the closure of each orbit $V_{\sf j}$, with $\sf j$, the invariant of the cubic, held constant. Naturally, the union of orbits of $V_{\sf j}$ for all the values of $\sf j$ form an open dense in $\Grass(3,R_2)$, but the analysis made here is more fine. In what follows, we will show that the table of orbits and their closures is complete and without repetition.\footnote{{\it Note of translators}: A typo in the top right entry in the table has been corrected from the original. The correct entry is $\langle Z^2,X^2-YZ,Y^2-XZ\rangle$. The original top right entry $\langle XY,X^2+YZ,Z^2+X^2$  is in fact isomorphic to $\langle XY, X^2+YZ, Y^2+XZ\rangle$ at the top left.}$ ^,$\footnote{${\sf j}(\lambda)$ is the ${\sf j}$-invariant of the cubic $\phi$: here $\phi$ is always smooth, not isomorphic to the cubic of equation $X^3+Y^3+Z^3=0$. A smooth cubic can always be put into this form of Hesse $X^3+Y^3+Z^3-3\lambda XYZ=0$, satisfying the hypotheses for $\lambda$ listed in the Table under \#8b. The symbol $\omega$ denotes a primitive cube root of $1$. Note that when $\lambda =-1,-\omega$ or $-\omega^2$ the associated cubic is singular and isomorphic to \#6a - then the net has a common point; if $\lambda $ is $0,2,2\omega,2\omega^2$ then the net is isomorphic to \#6d (p.\pageref{excludedvalues}). According to \cite[Proposition~2]{Fr} we may take ${\sf j}(\lambda)=
\frac{\lambda^3(\lambda^3-8)^3}{27(\lambda +1)^3(\lambda +\omega)^3(\lambda+\omega^2)^3}.$ See Proposition~\ref{4.2prop}ff pp. \pageref{4.2prop} -\pageref{Table1pres}.}


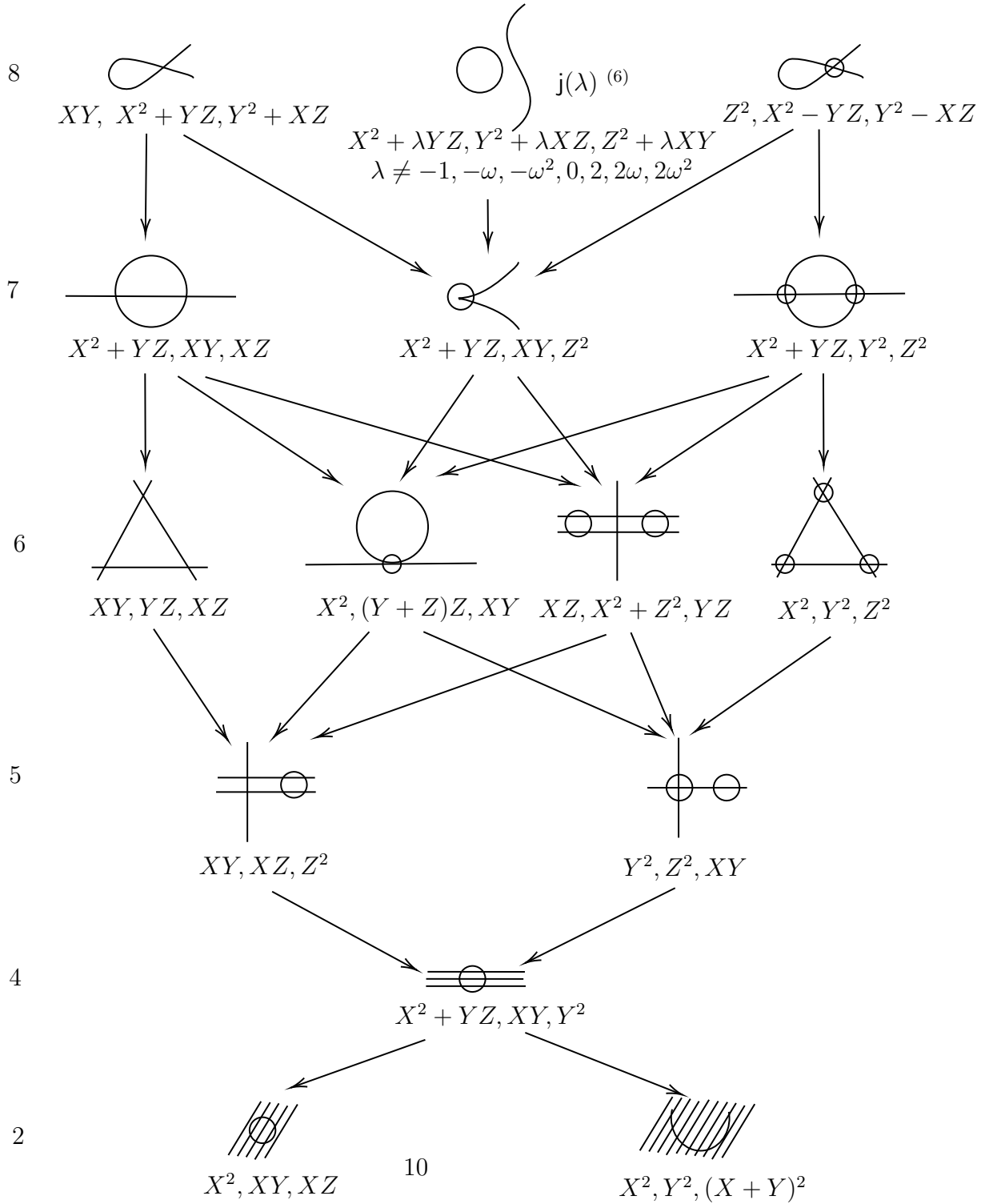
\begin{table}
    \caption{The orbits of $Pgl(3)$ acting on the planes of conics.}
    \label{table1}
    \centering

\tikzset{every picture/.style={line width=0.75pt}} 

\begin{tikzpicture}[x=0.75pt,y=0.75pt,yscale=-0.9,xscale=0.9]

\draw  [color={rgb, 255:red, 0; green, 0; blue, 0 }  ][line width=0.75] [line join = round][line cap = round] (111.5,54.32) .. controls (101.2,62.56) and (85.71,77.09) .. (74.5,83.32) .. controls (67.02,87.47) and (55.94,84.96) .. (54.5,76.32) .. controls (53.74,71.73) and (57.91,65.15) .. (62.5,64.32) .. controls (70.34,62.89) and (91.25,69.57) .. (99.5,72.32) .. controls (102.05,73.17) and (113.5,75.71) .. (113.5,77.32) ;
\draw    (80.78,116.5) -- (80.07,185.17) ;
\draw [shift={(80.05,187.17)}, rotate = 270.59000000000003] [color={rgb, 255:red, 0; green, 0; blue, 0 }  ][line width=0.75]    (10.93,-3.29) .. controls (6.95,-1.4) and (3.31,-0.3) .. (0,0) .. controls (3.31,0.3) and (6.95,1.4) .. (10.93,3.29)   ;
\draw   (59,227) .. controls (59,213.19) and (70.19,202) .. (84,202) .. controls (97.81,202) and (109,213.19) .. (109,227) .. controls (109,240.81) and (97.81,252) .. (84,252) .. controls (70.19,252) and (59,240.81) .. (59,227) -- cycle ;
\draw    (24,230.25) -- (143.5,230.75) ;
\draw    (79.78,284.5) -- (80.04,347.17) ;
\draw [shift={(80.05,349.17)}, rotate = 269.76] [color={rgb, 255:red, 0; green, 0; blue, 0 }  ][line width=0.75]    (10.93,-3.29) .. controls (6.95,-1.4) and (3.31,-0.3) .. (0,0) .. controls (3.31,0.3) and (6.95,1.4) .. (10.93,3.29)   ;
\draw    (71.72,359.38) -- (115.72,429.38) ;
\draw    (42.37,420.38) -- (123.72,420.38) ;
\draw    (46.37,429.38) -- (84.37,359.38) ;
\draw   (298.23,71.88) .. controls (298.23,63.11) and (305.34,56) .. (314.12,56) .. controls (322.89,56) and (330,63.11) .. (330,71.88) .. controls (330,80.66) and (322.89,87.77) .. (314.12,87.77) .. controls (305.34,87.77) and (298.23,80.66) .. (298.23,71.88) -- cycle ;
\draw    (336.9,114.77) .. controls (376.9,84.77) and (305.9,55.77) .. (345.9,25.77) ;
\draw  [color={rgb, 255:red, 0; green, 0; blue, 0 }  ][line width=0.75] [line join = round][line cap = round] (580.5,54.32) .. controls (570.2,62.56) and (554.71,77.09) .. (543.5,83.32) .. controls (536.02,87.47) and (524.94,84.96) .. (523.5,76.32) .. controls (522.74,71.73) and (526.91,65.15) .. (531.5,64.32) .. controls (539.34,62.89) and (560.25,69.57) .. (568.5,72.32) .. controls (571.05,73.17) and (582.5,75.71) .. (582.5,77.32) ;
\draw   (556.23,70.38) .. controls (556.23,66.86) and (559.09,64) .. (562.62,64) .. controls (566.14,64) and (569,66.86) .. (569,70.38) .. controls (569,73.91) and (566.14,76.77) .. (562.62,76.77) .. controls (559.09,76.77) and (556.23,73.91) .. (556.23,70.38) -- cycle ;
\draw  [color={rgb, 255:red, 0; green, 0; blue, 0 }  ][line width=0.75] [line join = round][line cap = round] (299.53,231.67) .. controls (311.24,231.67) and (322.89,223.89) .. (329.15,218.65) .. controls (332.61,215.76) and (341.19,209.03) .. (341.19,206.87) ;
\draw   (291.83,230.07) .. controls (292.21,225.06) and (296.58,221.31) .. (301.59,221.7) .. controls (306.59,222.08) and (310.34,226.45) .. (309.96,231.46) .. controls (309.58,236.47) and (305.21,240.22) .. (300.2,239.83) .. controls (295.19,239.45) and (291.44,235.08) .. (291.83,230.07) -- cycle ;
\draw   (528,227) .. controls (528,213.19) and (539.19,202) .. (553,202) .. controls (566.81,202) and (578,213.19) .. (578,227) .. controls (578,240.81) and (566.81,252) .. (553,252) .. controls (539.19,252) and (528,240.81) .. (528,227) -- cycle ;
\draw    (492.05,229.5) -- (612.05,228.5) ;
\draw   (246.28,418) .. controls (246.28,414.47) and (249.14,411.62) .. (252.67,411.62) .. controls (256.19,411.62) and (259.05,414.47) .. (259.05,418) .. controls (259.05,421.53) and (256.19,424.38) .. (252.67,424.38) .. controls (249.14,424.38) and (246.28,421.53) .. (246.28,418) -- cycle ;
\draw   (522.62,229) .. controls (522.62,225.47) and (525.47,222.62) .. (529,222.62) .. controls (532.53,222.62) and (535.38,225.47) .. (535.38,229) .. controls (535.38,232.53) and (532.53,235.38) .. (529,235.38) .. controls (525.47,235.38) and (522.62,232.53) .. (522.62,229) -- cycle ;
\draw   (570.62,229) .. controls (570.62,225.47) and (573.47,222.62) .. (577,222.62) .. controls (580.53,222.62) and (583.38,225.47) .. (583.38,229) .. controls (583.38,232.53) and (580.53,235.38) .. (577,235.38) .. controls (573.47,235.38) and (570.62,232.53) .. (570.62,229) -- cycle ;
\draw   (228,392) .. controls (228,378.19) and (239.19,367) .. (253,367) .. controls (266.81,367) and (278,378.19) .. (278,392) .. controls (278,405.81) and (266.81,417) .. (253,417) .. controls (239.19,417) and (228,405.81) .. (228,392) -- cycle ;
\draw    (192.05,418.5) -- (312.05,417.5) ;
\draw    (319.78,162.5) -- (320.03,194.17) ;
\draw [shift={(320.05,196.17)}, rotate = 269.55] [color={rgb, 255:red, 0; green, 0; blue, 0 }  ][line width=0.75]    (10.93,-3.29) .. controls (6.95,-1.4) and (3.31,-0.3) .. (0,0) .. controls (3.31,0.3) and (6.95,1.4) .. (10.93,3.29)   ;
\draw    (547.72,357.38) -- (591.72,427.38) ;
\draw    (518.37,418.38) -- (599.72,418.38) ;
\draw    (522.37,427.38) -- (560.37,357.38) ;
\draw    (551.78,113.5) -- (552.04,183.17) ;
\draw [shift={(552.05,185.17)}, rotate = 269.79] [color={rgb, 255:red, 0; green, 0; blue, 0 }  ][line width=0.75]    (10.93,-3.29) .. controls (6.95,-1.4) and (3.31,-0.3) .. (0,0) .. controls (3.31,0.3) and (6.95,1.4) .. (10.93,3.29)   ;
\draw    (410.43,359.75) -- (410.43,429.75) ;
\draw    (368.62,384.67) -- (453.78,384.73) ;
\draw   (374.14,388.99) .. controls (374.53,383.98) and (378.9,380.23) .. (383.91,380.62) .. controls (388.91,381) and (392.66,385.37) .. (392.28,390.38) .. controls (391.89,395.38) and (387.52,399.13) .. (382.52,398.75) .. controls (377.51,398.37) and (373.76,394) .. (374.14,388.99) -- cycle ;
\draw   (427.98,388.99) .. controls (428.36,383.98) and (432.73,380.23) .. (437.74,380.62) .. controls (442.75,381) and (446.5,385.37) .. (446.11,390.38) .. controls (445.73,395.38) and (441.36,399.13) .. (436.35,398.75) .. controls (431.34,398.37) and (427.59,394) .. (427.98,388.99) -- cycle ;
\draw   (548.62,368) .. controls (548.62,364.47) and (551.47,361.62) .. (555,361.62) .. controls (558.53,361.62) and (561.38,364.47) .. (561.38,368) .. controls (561.38,371.53) and (558.53,374.38) .. (555,374.38) .. controls (551.47,374.38) and (548.62,371.53) .. (548.62,368) -- cycle ;
\draw   (580.62,418) .. controls (580.62,414.47) and (583.47,411.62) .. (587,411.62) .. controls (590.53,411.62) and (593.38,414.47) .. (593.38,418) .. controls (593.38,421.53) and (590.53,424.38) .. (587,424.38) .. controls (583.47,424.38) and (580.62,421.53) .. (580.62,418) -- cycle ;
\draw   (521.62,418) .. controls (521.62,414.47) and (524.47,411.62) .. (528,411.62) .. controls (531.53,411.62) and (534.38,414.47) .. (534.38,418) .. controls (534.38,421.53) and (531.53,424.38) .. (528,424.38) .. controls (524.47,424.38) and (521.62,421.53) .. (521.62,418) -- cycle ;
\draw    (106.78,117.5) -- (275.05,213.51) ;
\draw [shift={(276.78,214.5)}, rotate = 209.71] [color={rgb, 255:red, 0; green, 0; blue, 0 }  ][line width=0.75]    (10.93,-3.29) .. controls (6.95,-1.4) and (3.31,-0.3) .. (0,0) .. controls (3.31,0.3) and (6.95,1.4) .. (10.93,3.29)   ;
\draw    (533.78,113.5) -- (358.52,213.51) ;
\draw [shift={(356.78,214.5)}, rotate = 330.28999999999996] [color={rgb, 255:red, 0; green, 0; blue, 0 }  ][line width=0.75]    (10.93,-3.29) .. controls (6.95,-1.4) and (3.31,-0.3) .. (0,0) .. controls (3.31,0.3) and (6.95,1.4) .. (10.93,3.29)   ;
\draw    (102.78,286.5) -- (213.11,359.4) ;
\draw [shift={(214.78,360.5)}, rotate = 213.45] [color={rgb, 255:red, 0; green, 0; blue, 0 }  ][line width=0.75]    (10.93,-3.29) .. controls (6.95,-1.4) and (3.31,-0.3) .. (0,0) .. controls (3.31,0.3) and (6.95,1.4) .. (10.93,3.29)   ;
\draw    (309.78,285.5) -- (261.92,354.85) ;
\draw [shift={(260.78,356.5)}, rotate = 304.61] [color={rgb, 255:red, 0; green, 0; blue, 0 }  ][line width=0.75]    (10.93,-3.29) .. controls (6.95,-1.4) and (3.31,-0.3) .. (0,0) .. controls (3.31,0.3) and (6.95,1.4) .. (10.93,3.29)   ;
\draw    (340.78,286.73) -- (395.56,357.92) ;
\draw [shift={(396.78,359.5)}, rotate = 232.42000000000002] [color={rgb, 255:red, 0; green, 0; blue, 0 }  ][line width=0.75]    (10.93,-3.29) .. controls (6.95,-1.4) and (3.31,-0.3) .. (0,0) .. controls (3.31,0.3) and (6.95,1.4) .. (10.93,3.29)   ;
\draw    (122,287) -- (378.86,361.94) ;
\draw [shift={(380.78,362.5)}, rotate = 196.26] [color={rgb, 255:red, 0; green, 0; blue, 0 }  ][line width=0.75]    (10.93,-3.29) .. controls (6.95,-1.4) and (3.31,-0.3) .. (0,0) .. controls (3.31,0.3) and (6.95,1.4) .. (10.93,3.29)   ;
\draw    (554.78,284.5) -- (555.04,347.17) ;
\draw [shift={(555.05,349.17)}, rotate = 269.76] [color={rgb, 255:red, 0; green, 0; blue, 0 }  ][line width=0.75]    (10.93,-3.29) .. controls (6.95,-1.4) and (3.31,-0.3) .. (0,0) .. controls (3.31,0.3) and (6.95,1.4) .. (10.93,3.29)   ;
\draw    (539.78,284.5) -- (428.45,358.39) ;
\draw [shift={(426.78,359.5)}, rotate = 326.43] [color={rgb, 255:red, 0; green, 0; blue, 0 }  ][line width=0.75]    (10.93,-3.29) .. controls (6.95,-1.4) and (3.31,-0.3) .. (0,0) .. controls (3.31,0.3) and (6.95,1.4) .. (10.93,3.29)   ;
\draw    (516.78,286.5) -- (289.7,354.92) ;
\draw [shift={(287.78,355.5)}, rotate = 343.23] [color={rgb, 255:red, 0; green, 0; blue, 0 }  ][line width=0.75]    (10.93,-3.29) .. controls (6.95,-1.4) and (3.31,-0.3) .. (0,0) .. controls (3.31,0.3) and (6.95,1.4) .. (10.93,3.29)   ;
\draw    (151.43,542.75) -- (151.43,612.75) ;
\draw    (130.62,567.67) -- (198.78,567.73) ;
\draw    (129.52,577.75) -- (199.35,577.75) ;
\draw   (174.98,571.99) .. controls (175.36,566.98) and (179.73,563.23) .. (184.74,563.62) .. controls (189.75,564) and (193.5,568.37) .. (193.11,573.38) .. controls (192.73,578.38) and (188.36,582.13) .. (183.35,581.75) .. controls (178.34,581.37) and (174.59,577) .. (174.98,571.99) -- cycle ;
\draw    (85.78,463.5) -- (136.67,539.84) ;
\draw [shift={(137.78,541.5)}, rotate = 236.31] [color={rgb, 255:red, 0; green, 0; blue, 0 }  ][line width=0.75]    (10.93,-3.29) .. controls (6.95,-1.4) and (3.31,-0.3) .. (0,0) .. controls (3.31,0.3) and (6.95,1.4) .. (10.93,3.29)   ;
\draw    (236.78,465.5) -- (171.13,538.02) ;
\draw [shift={(169.78,539.5)}, rotate = 312.15999999999997] [color={rgb, 255:red, 0; green, 0; blue, 0 }  ][line width=0.75]    (10.93,-3.29) .. controls (6.95,-1.4) and (3.31,-0.3) .. (0,0) .. controls (3.31,0.3) and (6.95,1.4) .. (10.93,3.29)   ;
\draw    (402.78,467.5) -- (201.67,538.83) ;
\draw [shift={(199.78,539.5)}, rotate = 340.47] [color={rgb, 255:red, 0; green, 0; blue, 0 }  ][line width=0.75]    (10.93,-3.29) .. controls (6.95,-1.4) and (3.31,-0.3) .. (0,0) .. controls (3.31,0.3) and (6.95,1.4) .. (10.93,3.29)   ;
\draw    (453.43,539.75) -- (453.43,609.75) ;
\draw    (431.52,574.75) -- (501.35,574.75) ;
\draw   (477.98,573.99) .. controls (478.36,568.98) and (482.73,565.23) .. (487.74,565.62) .. controls (492.75,566) and (496.5,570.37) .. (496.11,575.38) .. controls (495.73,580.38) and (491.36,584.13) .. (486.35,583.75) .. controls (481.34,583.37) and (477.59,579) .. (477.98,573.99) -- cycle ;
\draw    (274.78,465.5) -- (438.95,537.69) ;
\draw [shift={(440.78,538.5)}, rotate = 203.74] [color={rgb, 255:red, 0; green, 0; blue, 0 }  ][line width=0.75]    (10.93,-3.29) .. controls (6.95,-1.4) and (3.31,-0.3) .. (0,0) .. controls (3.31,0.3) and (6.95,1.4) .. (10.93,3.29)   ;
\draw    (560,469) -- (466.42,534.35) ;
\draw [shift={(464.78,535.5)}, rotate = 325.07] [color={rgb, 255:red, 0; green, 0; blue, 0 }  ][line width=0.75]    (10.93,-3.29) .. controls (6.95,-1.4) and (3.31,-0.3) .. (0,0) .. controls (3.31,0.3) and (6.95,1.4) .. (10.93,3.29)   ;
\draw    (420,466) -- (448.01,532.66) ;
\draw [shift={(448.78,534.5)}, rotate = 247.20999999999998] [color={rgb, 255:red, 0; green, 0; blue, 0 }  ][line width=0.75]    (10.93,-3.29) .. controls (6.95,-1.4) and (3.31,-0.3) .. (0,0) .. controls (3.31,0.3) and (6.95,1.4) .. (10.93,3.29)   ;
\draw    (277.62,703.67) -- (345.78,703.73) ;
\draw    (276.52,713.75) -- (346.35,713.75) ;
\draw   (299.98,707.99) .. controls (300.36,702.98) and (304.73,699.23) .. (309.74,699.62) .. controls (314.75,700) and (318.5,704.37) .. (318.11,709.38) .. controls (317.73,714.38) and (313.36,718.13) .. (308.35,717.75) .. controls (303.34,717.37) and (299.59,713) .. (299.98,707.99) -- cycle ;
\draw    (276.62,708.67) -- (344.78,708.73) ;
\draw    (168.78,647.5) -- (268.29,701.55) ;
\draw [shift={(270.05,702.5)}, rotate = 208.51] [color={rgb, 255:red, 0; green, 0; blue, 0 }  ][line width=0.75]    (10.93,-3.29) .. controls (6.95,-1.4) and (3.31,-0.3) .. (0,0) .. controls (3.31,0.3) and (6.95,1.4) .. (10.93,3.29)   ;
\draw    (451,645) -- (347.84,696.61) ;
\draw [shift={(346.05,697.5)}, rotate = 333.41999999999996] [color={rgb, 255:red, 0; green, 0; blue, 0 }  ][line width=0.75]    (10.93,-3.29) .. controls (6.95,-1.4) and (3.31,-0.3) .. (0,0) .. controls (3.31,0.3) and (6.95,1.4) .. (10.93,3.29)   ;
\draw  [draw opacity=0] (488.58,801.82) .. controls (488.79,802.97) and (488.93,804.16) .. (488.98,805.37) .. controls (489.51,817.97) and (480.82,828.57) .. (469.58,829.04) .. controls (458.33,829.52) and (448.78,819.69) .. (448.25,807.1) .. controls (448.15,804.63) and (448.39,802.25) .. (448.94,800) -- (468.61,806.23) -- cycle ; \draw   (488.58,801.82) .. controls (488.79,802.97) and (488.93,804.16) .. (488.98,805.37) .. controls (489.51,817.97) and (480.82,828.57) .. (469.58,829.04) .. controls (458.33,829.52) and (448.78,819.69) .. (448.25,807.1) .. controls (448.15,804.63) and (448.39,802.25) .. (448.94,800) ;
\draw    (346,746.33) -- (454.19,789.1) ;
\draw [shift={(456.05,789.83)}, rotate = 201.57] [color={rgb, 255:red, 0; green, 0; blue, 0 }  ][line width=0.75]    (10.93,-3.29) .. controls (6.95,-1.4) and (3.31,-0.3) .. (0,0) .. controls (3.31,0.3) and (6.95,1.4) .. (10.93,3.29)   ;
\draw    (276,751.33) -- (183.93,784.16) ;
\draw [shift={(182.05,784.83)}, rotate = 340.38] [color={rgb, 255:red, 0; green, 0; blue, 0 }  ][line width=0.75]    (10.93,-3.29) .. controls (6.95,-1.4) and (3.31,-0.3) .. (0,0) .. controls (3.31,0.3) and (6.95,1.4) .. (10.93,3.29)   ;
\draw   (444.98,573.99) .. controls (445.36,568.98) and (449.73,565.23) .. (454.74,565.62) .. controls (459.75,566) and (463.5,570.37) .. (463.11,575.38) .. controls (462.73,580.38) and (458.36,584.13) .. (453.35,583.75) .. controls (448.34,583.37) and (444.59,579) .. (444.98,573.99) -- cycle ;
\draw   (152.98,813.99) .. controls (153.36,808.98) and (157.73,805.23) .. (162.74,805.62) .. controls (167.75,806) and (171.5,810.37) .. (171.11,815.38) .. controls (170.73,820.38) and (166.36,824.13) .. (161.35,823.75) .. controls (156.34,823.37) and (152.59,819) .. (152.98,813.99) -- cycle ;
\draw    (163.67,834.19) -- (186.56,796.56) ;
\draw    (138.11,832.13) -- (161,794.5) ;
\draw    (144.53,832.8) -- (167.42,795.17) ;
\draw    (150.6,833.5) -- (173.49,795.87) ;
\draw    (156.67,835.19) -- (179.56,797.56) ;
\draw    (451.87,830.99) -- (474.76,793.36) ;
\draw    (426.31,828.93) -- (449.2,791.3) ;
\draw    (432.73,829.6) -- (455.62,791.97) ;
\draw    (438.8,830.3) -- (461.69,792.67) ;
\draw    (444.87,831.99) -- (467.76,794.36) ;
\draw    (483.87,833.99) -- (506.76,796.36) ;
\draw    (458.31,831.93) -- (481.2,794.3) ;
\draw    (464.73,832.6) -- (487.62,794.97) ;
\draw    (470.8,833.3) -- (493.69,795.67) ;
\draw    (476.87,834.99) -- (499.76,797.36) ;
\draw    (368.62,395.67) -- (453.78,395.73) ;
\draw  [color={rgb, 255:red, 0; green, 0; blue, 0 }  ][line width=0.75] [line join = round][line cap = round] (341.6,253.41) .. controls (335.2,244.49) and (322.89,239) .. (315.46,236.51) .. controls (311.37,235.14) and (301.53,231.53) .. (299.88,232.48) ;

\draw (17,93.4) node [anchor=north west][inner sep=0.75pt]    {$XY,\ X^{2} +YZ,Y^{2} +XZ$};
\draw (24,257.4) node [anchor=north west][inner sep=0.75pt]    {$X^{2} +YZ,XY,XZ$};
\draw (482,92.4) node [anchor=north west][inner sep=0.75pt]    {$Z^{2} ,X^{2} -YZ ,Y^{2} -XZ$};
\draw (220,111.4) node [anchor=north west][inner sep=0.75pt]    {$X^{2} +\lambda YZ,Y^{2} +\lambda XZ,Z^{2} +\lambda XY$};
\draw (236,132.4) node [anchor=north west][inner sep=0.75pt]    {$\lambda \neq -1,-\omega,-\omega^{2} ,0,2,2\omega,2\omega^{2}$};
\draw (367,69.4) node [anchor=north west][inner sep=0.75pt]    {${\sf j}( \lambda )\,\,^{( 6)}$};
\draw (256,257.4) node [anchor=north west][inner sep=0.75pt]    {$X^{2} +YZ,XY,Z^{2}$};
\draw (499,257.4) node [anchor=north west][inner sep=0.75pt]    {$X^{2} +YZ,Y^{2} ,Z^{2}$};
\draw (39,440.4) node [anchor=north west][inner sep=0.75pt]    {$XY,YZ,XZ$};
\draw (198,438.4) node [anchor=north west][inner sep=0.75pt]    {$X^{2} ,( Y+Z) Z,XY$};
\draw (354,439.4) node [anchor=north west][inner sep=0.75pt]    {$XZ,X^{2} +Z^{2} ,YZ$};
\draw (521,440.4) node [anchor=north west][inner sep=0.75pt]    {$X^{2} ,Y^{2} ,Z^{2}$};
\draw (116,619.4) node [anchor=north west][inner sep=0.75pt]    {$XY,XZ,Z^{2}$};
\draw (413,620.4) node [anchor=north west][inner sep=0.75pt]    {$Y^{2} ,Z^{2} ,XY$};
\draw (253,724.4) node [anchor=north west][inner sep=0.75pt]    {$X^{2} +YZ,XY,Y^{2}$};
\draw (410,845.08) node [anchor=north west][inner sep=0.75pt]    {$X^{2} ,Y^{2} ,( X+Y)^{2}$};
\draw (119,843.05) node [anchor=north west][inner sep=0.75pt]    {$X^{2} ,XY,XZ$};
\draw (-18,66.4) node [anchor=north west][inner sep=0.75pt]    {$8$};
\draw (-19,218.4) node [anchor=north west][inner sep=0.75pt]    {$7$};
\draw (-14,396.4) node [anchor=north west][inner sep=0.75pt]    {$6$};
\draw (-17,558.4) node [anchor=north west][inner sep=0.75pt]    {$5$};
\draw (-17,700.4) node [anchor=north west][inner sep=0.75pt]    {$4$};
\draw (-15,811.2) node [anchor=north west][inner sep=0.75pt]    {$2$};

\end{tikzpicture}

\end{table}
\normalsize\newpage
Our third main result concerns Artinian algebras.
\begin{thm}(Theorem \ref{smooththm}) The classification up to isomorphism of graded Artinian algebras having Hilbert function $H(A)=(1,3,3)$ is that determined implicitly by Table \ref{table1}. Each such algebra is smoothable - has a flat deformation to ${\sf k}^7$.\par
Every graded algebra of Hilbert function $H(A)=(1,r,2)$ is smoothable - has a flat deformation to ${\sf k}^{r+3}$.
\end{thm}
\subsection{Outline}\label{outlinesec}
We give now the main lines of the paper. In Section 2, we recall general facts about $S_4$, the cubic hypersurface parametrizing the singular conics in $\mathbb P^5=\mathbb P(R_2)$, and on its subset $D_2$ parametrizing the ``double lines" (perfect squares of linear forms). In Section 3.1, we classify up to isomorphism the pencils of planar conics: 2-dimensional vector spaces of homogeneous forms in $R_2$, which are lines of $\mathbb P(R_2)$; or, equivalently, the Artinian algebras of Hilbert function $(1,3,4)$ or their duals of Hilbert function $(1,3,2)$. This classification was already known by Corrado Segre in 1883 (see \cite{Wall}) and J. Emsalem had learned it in ``Sp\'{e}ciales", in the course of E. Riche. It gives an idea of the general method we will follow later. 

In Section 3.2, we give the closures of the orbits of pencils. In Section~4.1.1, we begin the analysis of vector spaces of dimension 3 in $R_2$, in considering the case where $\Gamma(V)=\mathbb P(V)\cap S_4$ is smooth. In Section 4.1.2, we show that such a vector space has the form $J_\phi$, the space generated by the partial derivatives of order 1 of a cubic form $\phi$. Remark \ref{betarem} shows that the map of $\sf j$ invariants ${\sf j}(\phi)$ to $ {\sf j}(\Gamma(J_\phi))$ is generically a $3:1$ map.

In Section 4.1.3, we will show (Proposition \ref{4.4prop}) that when $\Gamma_\pi$ for the net $\pi$ is a smooth cubic, 
then $\pi=J\phi$ for a cubic $\phi$ and $\Gamma$ has a canonical involution $\iota$ without double point, such that the three-dimensional vector space $J_\phi$ has the following property:
$$J_\phi=\{ v\in R_2\mid \forall C\in \Gamma,\ C \text { and } \iota(C) \text { are conjugate in relation to } v\}.$$
Then, in Section 4.2, we classify the vector spaces $V$ for which $\Gamma(V)$ is singular. The idea consists of using a geometric construction to find a basis of a space isomorphic to $V$, in a canonical form. The algebraic method of classification is presented in Section 5: in the present work, we have used it as a verification of the results of the geometric method. But the idea of this method may furnish some simplifications of other problems of algebraic classification.

The specializations, or in other words the closures of orbits, are discussed in Section 6. In view of shortening that section, we use the notion of duality among orbits (\S 6.1), a calculation of the dimension of each orbit (\S 6.2) and an upper semi-continuous invariant, the length of the projective scheme associated to the ideal generated by the vector space. This is useful also for considering the specializations of the orbits of low dimension (\S 6.3). In Section 6.4, we apply these notions to restrict the set of potential specializations of nets. In Section 6.5 we study specializations of cubics. In Section \ref{specjsec}, we give an argument from invariant theory about the closure of each orbit in the one parameter family of smooth orbits, or the orbit of the singular cubic with a double point having distinct tangents. We show: \par\vskip 0.2cm \noindent
{\bf Theorem \ref{specialthm}.} Let $F$ denote an orbit of cubics of dimension less than 8. Either all cubics of the above family specialize to $F$; or exactly one does.
 \par\vskip 0.2cm
In Section \ref{singsec} we give the specializations of singular cubics; we also specify the discriminant cubics of Table \ref{table1}; in the new \S \ref{polarsec} we for completeness list the polar nets $J_\phi$ - those determined by the partial derivatives of a cubic form $\phi$.  In Section \ref{specialnetsec}, we establish explicitly the existence of the specializations of nets of conics announced in Table \ref{table1}. 

In Section 7, we examine the deformations of Artinian algebras. We show in Section 7.1 that the algebras of Hilbert function $(1,3,3)$ are \emph{smoothable}, that is that they deform to algebras isomorphic to ${\sf k}^7$ (given the product structure of the algebras). In Section 7.2, we discuss the form of the algebras of Hilbert function $(1,4,2)$ and, more generally, algebras of Hilbert function
$(1,r,2)$.  The algebras of Hilbert function $(1,4,2)$ constitute the other case of algebras of length 7 presenting a continuous family of distinct isomorphism classes; we describe the family with $r-3$ parameters of isomorphism classes of algebras of Hilbert function $(1,r,2)$ that are sufficiently general. Then in exhibiting the generators and the relations among these generators, we show that these algebras also have smooth deformations. 

Finally, in Section 8 we will show that the study of the regular maps of degree 4 from the projective plane into itself is equivalent to the study of nets of conics. We give an interpretation of the invariants of a net of cubic curves in terms of the critical locus and the discriminant of a regular map of degree 4.  In Section 9 we discuss Hesse pencils of cubics and their connection to nets.

The geometric arguments, including the existence of specializations, in particular the use of the theory of invariants in Section \ref{specjsec}, and the Section~8 are due to J.~Emsalem.
A.~Iarrobino has determined algebraically the list of orbits (Section~6), the duality (Section 6.1) and the deformations of Section 7.2. We thank Bernard Teissier for his encouragement and comments. We note that C.T.C.~Wall (see \cite{Wall}) has considered independently the same problem and found very succinctly an analogous table (except for the specializations of constant $\sf j$, our Section 6.5), both in the complex case that we study, and in the real case that we do not consider; we develop the algebra and the geometry of the situation in more detail (Sections 4.1, 4.2, 5, etc $\dots$). The non-separability of the space of orbits is a phenomenon similar to that described by M.~Nagata in \cite{Nag}.

We thank Martine Aeschbacher who promptly deciphered and made readable a handwritten manuscript that was often obscure.\par
\section{Space of Projective Conics}
{\bf Introduction}: Here we describe the main aspects of the geometry of the space of conics in the projective plane.\par
We denote by $C_5$ the projective space associated to $R_2$ (which has dimension 5), that we identify with the space of planar conics. In this space, the set of singular conics form a remarkable cubic hypersurface that we call $S_4$
in the sequel:
\par $S_4=\{$the $\mathbb C^*$ classes of $F,F\in R_2$, such that  $\det_{X,Y,Z}(F'_X,F'_Y,F'_Z)=0\}$.\newline
$S_4$ in turn contains $D_2$, the set of ``double lines''. $D_2$ is a variety isomorphic to $\mathbb P^2$, so of dimension 2:
$$ D_2=\{ \text {classes of } L^2, L\in R_1\}.$$
\begin{rmk}\label{rmkS2}
It is easy to see that $D_2$ is exactly the singular locus of $S_4$ and that $S_4$ is the union of the planes of $C_5$ tangent to $D_2$. Thus, $S_4$ is also a rational variety. $D_2$ and $S_4$ are evidently stable under the action of $Pgl(3)$. In the general case where a plane $\pi$ of $C_5$ intersects $S_4$ properly, the scheme intersection with $S_4$ is a cubic $\Gamma_\pi$ in $\pi$, and we will try to conclude the maximum information on the class of the plane modulo $Pgl(3)$ from the pair $(\Gamma_\pi, \pi)$. We note finally that a plane of $C_5$ does not in general meet $D_2$ for the obvious reason of dimension, and that an intersection point of a plane with $D_2$, when it exists, can only be a singular point of the intersection cubic in $S_4$, as it is already a singular point of $S_4$. \footnote{We list the discriminant cubics in Section \ref{singsec}, Table \ref{discrimtable}; see Example \ref{8adiscex}.}\end{rmk}

We give now some details on certain families of linear projective subvarieties of $C_5$:\par
\noindent
(2.1) The tangent plane to $D_2$ at the point corresponding to a double line is the set of singular conics having as one of its irreducible components the support of that double line.\par\noindent
(2.2) The hyperplane tangent to $S_4$ at a smooth point - that is, a point corresponding to a singular conic having two distinct irreducible components, is the set of all conics passing through the unique singular point of that conic.\par\noindent
(2.3) The Zariski tangent cone to $S_4$ at a singular point of $S_4$, that is, at a double line, is a quadratic cone of dimension 4. This cone is the set of conics tangent to the support of the double line. It contains $D_2$ and the tangent plane to $D_2$ at the
point considered. One may regard it also as the set of all classes comprised of a sum of a representative of a point of this tangent plane, and of a representative of a point of $D_2$.\par\noindent
(2.4) The maximal linear subvarieties of $S_4$ are the planes comprising the following two disjoint sets:\par
(2.4a) The set of planes $\pi_x:x\in P$ (the plane corresponding to the dual of $R_1)$ where $\pi_x=\{\gamma:\gamma\in S_4,x$ is a singular point of  $\gamma$.\}\par
(2.4b) The set of planes $\pi_\delta: \delta \in P^\ast$ where $\pi_\delta =\{\gamma, \gamma\in S_4$, such that $\delta$ is an irreducible component of $\gamma.\}$\par
One has the following properties where, if $\gamma \in S_4-D_2,$ $ m_\gamma$ denotes the unique singular point of $\gamma$ and $T_\gamma(S_4)$ denotes the hyperplane of $C_5$ tangent at $\gamma$ to $S_4$:\par
\begin{enumerate}[(i)]
\item $\forall \gamma \in S_4-D_2, \pi_{m_\gamma}-D_2=\{\gamma': \gamma'\in S_4-D_2, T_{\gamma'}S_4=T_{\gamma}S_4\}$. Also, $\pi_{m_\gamma}=\overline{\pi_{m_\gamma}-D_2}$.
\item $\forall x,y\in P, x\neq y \rightarrow \pi_x\cap \pi_y=\{u_{xy}\}$, where $u_{xy}$ is the element of $D_2$ whose support contains $x$ and $y$.
\item $\forall \delta\in P^\ast, \pi_\delta$ is the tangent plane to $D_2$ at the point whose support is $\delta$.
\item $\forall \delta,\delta' \in P^\ast, \delta\neq \delta' \Rightarrow \pi_\delta\cap \pi_{\delta'}=\{v_{\delta\delta'}\}$ where $ v_{\delta\delta'}$ is the element of $S_4-D_2$ having as irreducible components $\delta$ and $\delta'$.
\end{enumerate}
\par\noindent
(2.5) Consider the action of $\PGL(5)$ on the projective space of conics $C_5$. The preceding considerations show 
 the following assertion:\par
Let $\sigma\in \PGL(5)=\mathrm{Aut}(C_5)$, the following conditions are equivalent:
\begin{enumerate}[(i)]
\item The automorphism $\sigma$ arises from an element of $Pgl(3)$ acting on $\mathrm{Proj}(R_2)$;
\item $\sigma$ stabilizes $S_4$ (maps $S_4$ onto $S_4$);
\item $\sigma$ stabilizes $D_2$.
\end{enumerate}
For the proofs keep in mind that $D_2$ is the singular locus of $S_4$.

\section{Pencils of conics}

Before discussing the classification of the planes of $C_5$, we will present the classification of the lines of $C_5$. Table 2 in Section 3.1 sets out this classification. The accompanying proof uses as main tools a weak form of Noether's theorem and Bezout's theorem. In Section 3.2, we follow this classification by determining the specializations, which are listed in Table 3, p. \pageref{table3}.

\subsection{Classification of pencils}

The notations are those in \cite{Ful}. We need the following version of Noether's theorem that can be deduced immediately from \cite[p.121, prop.1]{Ful} 

\begin{proposition}
Let $F$ and $G$ be two distinct plane projective curves with the same degree $n$, where $F$ is assumed to be smooth. Let $H$ be a plane projective curve of degree $n$. For $H$ to belong to the pencil generated by $F$ and $G$, it is necessary and sufficient that for every point $M$ in the intersection of the supports of $F$ and $G$, the intersection number $I(M; F\cap H)\geq I(M;F\cap G).$ 
\end{proposition}

In Table \ref{table2}, $U=\langle F,G\rangle$ denotes a line of $C_5$. The first diagram represents $U$ marked by points of $U\cap S_4$ and, possibly, of $U\cap D_2$. The points of $U\cap D_2$ are represented by \begin{tikzpicture}[scale=0.2]
\draw (0,0) circle (1);
\draw (-1.5,0) -- (1.5,0);
\filldraw (0,0) circle (5pt);
\end{tikzpicture}. The points of $U\cap S_4$ which do not belong to $D_2$ are represented by \begin{tikzpicture}[baseline={([yshift=-.8ex]current bounding box.center)},scale=0.2]
\draw (-1.5,0) -- (1.5,0);
\filldraw (0,0) circle (5pt);
\end{tikzpicture}. When a point of $U$ intervenes without belonging to $S_4$, it is noted \begin{tikzpicture}[scale=0.15]
\draw (-1.5,0) -- (1.5,0);
\draw (0,-1.5) -- (0,1.5);
\end{tikzpicture}.
When a point of $U\cap S_4$ has multiplicity $k$ in this intersection, it is repeated $k$ times. 

The second diagram represents the relative situation in the ambient plane $P$ of the conics which correspond to the points in the first diagram. Finally, a system of equations of these conics is provided for a suitably chosen coordinate system. However, in part II of the table, all the points of $U$ belong to $S_4$; we only indicate those which intervene in the chosen presentation. We continue, however, to systematically point out the elements of $U\cap D_2$.

\begin{table}
\caption{Classification of pencils of conics}\label{table2}
\begin{enumerate}[I.]
    \item \underline{$U\not\subset S_4$}.
    
\newcolumntype{L}[1]{>{\raggedright\arraybackslash}m{#1}}
\newcolumntype{C}[1]{>{\centering\arraybackslash}m{#1}}
\newcolumntype{R}[1]{>{\raggedleft\arraybackslash}m{#1}}

{
\setlength{\extrarowheight}{30pt}
\begin{tabular}{ L{3cm} C{3.5cm} L{9cm} } 
 \begin{tikzpicture}[scale=0.4]
\draw (-2.3,0) -- (2.3,0);
\filldraw[black] (-1.5,0) circle (4pt) node[anchor=south] {A};
\filldraw[black] (0,0) circle (4pt) node[anchor=south] {B};
\filldraw[black] (1.5,0) circle (4pt) node[anchor=south] {C};
\draw (2.3,0) node[anchor=west] {$U$};
\draw (-2.3,0) node[anchor=east] {a)};
\end{tikzpicture}
& \begin{tikzpicture}[x=0.75pt,y=0.75pt,yscale=-1,xscale=1] 

\draw    (296.5,51.67) -- (390,123) ;
\draw    (394.5,53.67) -- (314.5,134.67) ;
\draw    (282.5,86.67) -- (381.5,109.67) ;
\draw    (290.5,108.67) -- (404.5,75.67) ;
\draw    (346.5,32.67) -- (348,139) ;
\draw    (332.5,23.67) -- (381.5,134.67) ;
\draw    (305.5,92.25) .. controls (301.5,94.25) and (310.5,106.25) .. (311.5,102.25) ;
\draw    (349.5,98.25) .. controls (349.5,101.25) and (357.5,101.25) .. (358.5,97.25) ;
\draw    (346.5,42.25) .. controls (343.5,37.25) and (339.5,39.25) .. (341.5,44.25) ;

\draw (284.5,89.67) node [anchor=north west][inner sep=0.75pt]   [align=left] {B};
\draw (351,107) node [anchor=north west][inner sep=0.75pt]   [align=left] {C};
\draw (338,12) node [anchor=north west][inner sep=0.75pt]   [align=left] {A};

\end{tikzpicture} & $(A,B,C)=((X^2-Z^2),(Y^2-Z^2),(X^2-Y^2))$ \\

\begin{tikzpicture}[scale=0.4]
\draw (-2.3,0) node[anchor=east] {b)};
\draw (-2.3,0) -- (2.3,0);
\filldraw[black] (-1.5,-0.2) circle (4pt) ;
\filldraw[black] (0,0) circle (4pt) node[anchor=south] {A};
\filldraw[black] (-1.5,0.2) circle (4pt)node[anchor=south] {2B};
\draw (2.3,0) node[anchor=west] {$U$};
\end{tikzpicture}  & 
\tikzset{every picture/.style={line width=0.75pt}} 

\begin{tikzpicture}[x=0.75pt,y=0.75pt,yscale=-1,xscale=1]

\draw    (282.5,86.67) -- (395.5,111.66) ;
\draw    (290.5,108.67) -- (393.5,81.66) ;
\draw    (351.5,27.66) -- (323.5,132.66) ;
\draw    (332.5,23.67) -- (381.5,134.67) ;
\draw    (305.5,92.25) .. controls (301.5,94.25) and (310.5,106.25) .. (311.5,102.25) ;
\draw    (346.5,42.25) .. controls (343.5,37.25) and (339.5,39.25) .. (341.5,44.25) ;
\draw (284.5,89.67) node [anchor=north west][inner sep=0.75pt]   [align=left] {B};
\draw (338,12) node [anchor=north west][inner sep=0.75pt]   [align=left] {A};

\end{tikzpicture}
&$(B,A)=((X^2-Z^2),(ZY))$\\
\begin{tikzpicture}[scale=0.4]
\draw (-2.3,0) node[anchor=east] {c)};
\draw (-2.3,0) -- (2.3,0);
\filldraw[black] (-1.5,-0.3) circle (1.5pt) ;
\draw (-1.5,-0.3) circle (0.25);
\filldraw[black] (0,0) circle (4pt) node[anchor=south] {A};
\filldraw[black] (-1.5,0.3) circle (1.5pt);
\draw (-1.5,0.3) circle (0.25)node[anchor=south] {2B};
\draw (2.3,0) node[anchor=west] {$U$};
\end{tikzpicture}   & 
\begin{tikzpicture}[x=0.75pt,y=0.75pt,yscale=-1,xscale=1]

\draw    (43,106.37) -- (148,106.37) ;
\draw    (41,97.37) -- (147,97.37) ;
\draw    (104,36.37) -- (83,142.37) ;
\draw    (88.5,30.67) -- (137.5,141.67) ;
\draw    (154.5,95.48) .. controls (159.5,97.48) and (159.5,108.48) .. (154,109.37) ;
\draw    (102,47.37) .. controls (99,42.37) and (94,42.37) .. (96,47.37);
\draw (161.5,93.67) node [anchor=north west][inner sep=0.75pt]   [align=left] {B};
\draw (94,19) node [anchor=north west][inner sep=0.75pt]   [align=left] {A};
\end{tikzpicture}
& $(A,B)=((XY),(Z^2))$\\

\begin{tikzpicture}[scale=0.4]
\draw (-2.3,0) node[anchor=east] {d)};
\draw (-2.3,0) -- (2.3,0);
\filldraw[black] (-1.5,-0.3) circle (3pt) ;
\filldraw[black] (-1.5,0) circle (3pt);
\filldraw[black] (0,-0.4)-- (0,0.4) node[anchor=south] {C};
\filldraw[black] (-1.5,0.3) circle (3pt) node[anchor=south] {3B};
\draw (2.3,0) node[anchor=west] {$U$};
\end{tikzpicture} &
\begin{tikzpicture}[x=0.75pt,y=0.75pt,yscale=-1,xscale=1]

\draw    (41,98.37) -- (147,98.37) ;
\draw    (59,70) -- (59,130) ;
\draw    (51,98.28) .. controls (51,89.28) and (59,87.28) .. (59,87.28) ;
\draw   (59,98.25) .. controls (59,89.69) and (72.21,82.75) .. (88.5,82.75) .. controls (104.79,82.75) and (118,89.69) .. (118,98.25) .. controls (118,106.81) and (104.79,113.75) .. (88.5,113.75) .. controls (72.21,113.75) and (59,106.81) .. (59,98.25) -- cycle ;
\draw (31.5,77.67) node [anchor=north west][inner sep=0.75pt]   [align=left] {B};
\draw (110,62) node [anchor=north west][inner sep=0.75pt]   [align=left] {C};
\end{tikzpicture}
& $(B,C)=((XY),(XZ-Y^2))$\\
\begin{tikzpicture}[scale=0.4]
\draw (-2.3,0) node[anchor=east] {e)};
\draw (-2.3,0) -- (2.3,0);
\filldraw[black] (-1.5,-0.5) circle (1.5pt) ;
\draw (-1.5,-0.5) circle (0.25);
\filldraw[black] (-1.5,0) circle (1.5pt) ;
\draw (-1.5,0) circle (0.25);
\filldraw[black] (0,-0.4)-- (0,0.4) node[anchor=south] {C};
\filldraw[black] (-1.5,0.5) circle (1.5pt);
\draw (-1.5,0.5) circle (0.25)node[anchor=south] {3B};
\draw (2.3,0) node[anchor=west] {$U$};
\end{tikzpicture} &

\tikzset{every picture/.style={line width=0.75pt}} 

\begin{tikzpicture}[x=0.75pt,y=0.75pt,yscale=-1,xscale=1]

\draw    (60,64.75) -- (60,132.28) ;
\draw    (66,136.98) .. controls (64,143.98) and (52,141.63) .. (53,136.98) ;
\draw   (60,98.25) .. controls (60,89.69) and (73.21,82.75) .. (89.5,82.75) .. controls (105.79,82.75) and (119,89.69) .. (119,98.25) .. controls (119,106.81) and (105.79,113.75) .. (89.5,113.75) .. controls (73.21,113.75) and (60,106.81) .. (60,98.25) -- cycle ;
\draw    (57,64.75) -- (57,132.28) ;

\draw (51.5,146.67) node [anchor=north west][inner sep=0.75pt]   [align=left] {B};
\draw (110,62) node [anchor=north west][inner sep=0.75pt]   [align=left] {C};

\end{tikzpicture}
& $(B,C)=((X^2),(XZ-Y^2))$

\end{tabular}
}
\end{enumerate}
\end{table}

 \begin{table}

\begin{enumerate}
\item[II.] \underline{$U\subset S_4$}.

\newcolumntype{L}[1]{>{\raggedright\arraybackslash}m{#1}}
\newcolumntype{C}[1]{>{\centering\arraybackslash}m{#1}}
\newcolumntype{R}[1]{>{\raggedleft\arraybackslash}m{#1}}

{
\setlength{\extrarowheight}{30pt}
\begin{tabular}{ L{3cm} C{3.5cm} L{9cm} } 
\begin{tikzpicture}[scale=0.4]
\draw (-2.3,0) -- (2.3,0);
\filldraw[black] (-1.5,0) circle (4pt) node[anchor=south] {A};
\filldraw[black] (0,0) circle (4pt) node[anchor=south] {B};
\draw (2.3,0) node[anchor=west] {$U$};
\draw (-2.3,0) node[anchor=east] {f)};
\end{tikzpicture}  &

\tikzset{every picture/.style={line width=0.75pt}} 

\begin{tikzpicture}[x=0.75pt,y=0.75pt,yscale=-1,xscale=1]

\draw    (41,98.37) -- (147,98.37) ;
\draw    (48,84.88) -- (119,128.8) ;
\draw    (51,98.28) .. controls (51,92.88) and (54,90.88) .. (57,90.88) ;
\draw    (72,127.88) -- (147,85.8) ;
\draw    (140,89.8) .. controls (143,88.8) and (142,97.8) .. (140,97.8) ;

\draw (153.5,84.67) node [anchor=north west][inner sep=0.75pt]   [align=left] {B};
\draw (35,81) node [anchor=north west][inner sep=0.75pt]   [align=left] {A};

\end{tikzpicture}
& $(A,B)=((XY),(XZ))$\\

\begin{tikzpicture}[scale=0.4]
 \draw (-2.3,0) node[anchor=east] {g)};
\draw (-2.3,0) -- (2.3,0);
\filldraw[black] (-1.5,0) circle (3pt) ;
\draw (-1.5,0) circle (0.3) node[anchor=south] {A};
\filldraw[black] (0,0) circle (4pt) node[anchor=south] {B};
\draw (2.3,0) node[anchor=west] {$U$};
\end{tikzpicture} & 

\tikzset{every picture/.style={line width=0.75pt}} 

\begin{tikzpicture}[x=0.75pt,y=0.75pt,yscale=-1,xscale=1]

\draw    (41,96.37) -- (120,96.37) ;
\draw    (37,119.2) -- (43.61,112.59) -- (105,51.2) ;
\draw    (50,106.43) .. controls (47,108.98) and (43,96.98) .. (48,96.98) ;
\draw    (100,92.98) -- (140,92.8) ;
\draw    (145,87.98) .. controls (151,89.98) and (151,102.98) .. (144,105.98) ;
\draw    (100,99.98) -- (140,99.8) ;

\draw (26.5,101.67) node [anchor=north west][inner sep=0.75pt]   [align=left] {B};
\draw (155,87) node [anchor=north west][inner sep=0.75pt]   [align=left] {A};

\end{tikzpicture}
& $(A,B)=((X^2),(XY))$\\

\begin{tikzpicture}[scale=0.4]
\draw (-2.3,0) node[anchor=east] {h)};
\draw (-2.3,0) -- (2.3,0);
\filldraw[black] (-1.5,0) circle (3pt) ;
\draw (-1.5,0) circle (0.3) node [anchor= south]{B};
\filldraw[black] (0,0) circle (3pt);
\draw (0,0) circle (0.3) node [anchor= south]{C};
\draw (2.3,0) node[anchor=west] {$U$};
\end{tikzpicture}   &

\tikzset{every picture/.style={line width=0.75pt}} 

\begin{tikzpicture}[x=0.75pt,y=0.75pt,yscale=-1,xscale=1]

\draw    (43,106.37) -- (148,106.37) ;
\draw    (42,99.37) -- (148,99.37) ;
\draw    (87,37.38) -- (87,141.38) ;
\draw    (93,37.37) -- (93,141.38) ;
\draw    (154.5,95.48) .. controls (159.5,97.48) and (159.5,108.48) .. (154,109.37) ;
\draw    (99,144.38) .. controls (95,150.38) and (84,150.38) .. (83,143.38) ;

\draw (161.5,93.67) node [anchor=north west][inner sep=0.75pt]   [align=left] {B};
\draw (84,154) node [anchor=north west][inner sep=0.75pt]   [align=left] {C};

\end{tikzpicture}
& $(A,B)=((X^2),(Y^2))$

\end{tabular}
}
\end{enumerate}
\end{table}

\paragraph{Exhaustion of the classes of pencils of conics not contained in $S_4$.} 

We will apply the previous proposition to all the nets of conics not entirely contained in $S_4$ to find in each case a particular algebraic representation of a system of generators.

Let $U$ be such a pencil and $F$ and $G$ two generators of $U$, $F$ chosen to be smooth. By Bezout's theorem  \cite[\S 5.3 p. 57]{Ful} the intersection of $F$ and $G$ can consist of: 
\begin{enumerate}
    \item four distinct points, each of multiplicity one. According to the proposition, the conics of $U$ are those which pass through these four points. Three of these points are never aligned. We can therefore choose a homogeneous coordinate system where these points have the following representations: $(1,1,1), \ (-1,1,1)\ (1,-1,1),\ (1,1,-1)$. Three (and only three) singular conics pass through these four points. They have for equations $S:\ (X^2-Y^2), (Y^2-Z^2), (Z^2-X^2)$ and two of them determine the pencil. 
    \item three distinct points $J$, $K$ and $L$ such that $I(J;F\cap G)= I(K; F\cap G)=1$ and $ I(L;F\cap G)=2$. $J,\ K,\ L$ cannot be aligned. Let $T$ be the tangent to $F$ in $J$; $T$ cannot contain either of $J,K$. 
    We can choose a homogeneous coordinate system such that the equation for $T$ is $Z=0$ and the points $J,\ K,\ L$ have coordinates $(0,1,0),\ (1,0,1),\ (-1,0,1)$ respectively. We verify that the only two singular conics that intersect $F$ in $J,K,L$ with the suitable multiplicities to belong to $U=(\kappa ZY+\mu(X^2-Z^2))$ have equations $(ZY=0)$ and $(X^2-Z^2=0)$. The condition for which this conic is singular is written: $\kappa^2\mu=0$. In the intersection between $U$ with $S_4$, the conic $(ZY)$ has multiplicity one and $(X^2-Z^2)$ has multiplicity 2. This explains case (b) in Table
    \ref{table2}.     
    
    In the following cases, we only give a sketch of proof.
    
    \item two distinct points $J$ and $K$ with multiplicity 2, for example $J=(0,1,0)$ and $K=(1,0,0)$. Let $R$ and $S$ be the tangents to $F$ in $J$ and $K$ respectively. We choose $Z=0$ for the equation of $JK$, $X=0$ for the equation of $R$ and $Y=0$ for the equation of $S$. The only singular conics of the pencil $(\kappa Z^2+\mu XY)$  have for equations: $Z^2=0$ and $XY=0$. The condition for which this conic is singular is written $\kappa^2\mu=0$, which explains the result in case c) in Table \ref{table2}. 
    
    \item one point $J$ with multiplicity three and one single point $K$, as $J=(0,0,1)$ and $K=(1,0,0)$. Let $R$ and $S$ be the tangents to $F$ in $J$ and $K$ respectively. In a suitable homogeneous coordinate system, $R$, $JK$ and $S$ have equations $X=0$, $Y=0$ and $Z=0$ respectively , and $F$ has equation $XZ-Y^2=0$. $(XY)$ is the only singular conic of $U=(\kappa XY+\mu (XZ-Y^2))$, and the condition for which this conic is singular is written $\mu^3=0$, which explains the result in case d) in Table~\ref{table2}. 
    
    \item one point of multiplicity 4. With the conventions of the previous case, the only singular conic of the pencil $U=\kappa X^2+\mu(XZ-Y^2)$ has equation $(X^2=0)$ and we find all the results of case e) in Table \ref{table2}. 
\end{enumerate} 

\paragraph{Exhaustion of the classes of pencils of conics contained in $S_4$} Let $\alpha$ and $\beta$ be lines in $P$. We denote by $v_{\alpha\beta}$ the element of $S_4$ whose irreducible components are $\alpha$ and $\beta$ if $\alpha\neq \beta$ and the element of $D_2$ whose support is $\alpha$ if   $\alpha= \beta$.

From the previous examination, and by elimination, we can affirm that two singular conics $v_{\alpha \alpha'}$ and $v_{\beta \beta'}$ determine a pencil contained in $S_4$ if and only if they admit a common irreducible component or a common singular point, these conditions not being exclusive. The conjunction of the two conditions gives the case g) of Table \ref{table2}, the first condition without the second gives the case f) and the second without the first gives case h).

\paragraph{General classification of pencils of projective plane conics}  According to the preceding discussion, the table presented is exhaustive. On the other hand, it is without repetition and therefore constitutes a classification: indeed, according to the remarks given in Section 2, it becomes clear that for two pencils $U$ and $U'$ to belong to the same orbit under $Pgl(3)$, $U\cap D_2$ and $U'\cap D_2$ on the one hand and $U\cap S_4$ and $U'\cap S_4$ on the other hand, must have the same isomorphism type under linear maps: the types of these intersections described in the different cases are all distinct. 
\subsection{Specialization of pencils}
See Table \ref{table3} on the following page.\par
\begin{table}
\caption{Specializations of orbits of pencils in $C_5$.}\label{table3}

\tikzset{every picture/.style={line width=0.75pt}} 

\begin{tikzpicture}[x=0.75pt,y=0.75pt,yscale=-0.8,xscale=0.8]

\draw    (158.05,56.2) -- (231.57,147.99) ;
\draw    (252.67,81.79) -- (155.49,141.1) ;
\draw    (136.02,86.79) -- (236.49,138.56) ;
\draw    (138.71,109.1) -- (257.31,104.58) ;
\draw    (212.12,50.07) -- (185.98,153.14) ;
\draw    (199.74,37.72) -- (220.51,157.26) ;
\draw    (199.5,163.37) -- (199.5,199.37) ;
\draw [shift={(199.5,201.37)}, rotate = 270] [color={rgb, 255:red, 0; green, 0; blue, 0 }  ][line width=0.75]    (10.93,-3.29) .. controls (6.95,-1.4) and (3.31,-0.3) .. (0,0) .. controls (3.31,0.3) and (6.95,1.4) .. (10.93,3.29)   ;
\draw    (249.37,261.84) -- (149.14,315.82) ;
\draw    (132.62,260.55) -- (230.16,317.66) ;
\draw    (210.59,227.98) -- (178.94,329.49) ;
\draw    (198.89,214.98) -- (213.19,335.46) ;
\draw    (211.5,351.78) -- (240.07,379.6) ;
\draw [shift={(241.5,381)}, rotate = 224.24] [color={rgb, 255:red, 0; green, 0; blue, 0 }  ][line width=0.75]    (10.93,-3.29) .. controls (6.95,-1.4) and (3.31,-0.3) .. (0,0) .. controls (3.31,0.3) and (6.95,1.4) .. (10.93,3.29)   ;
\draw    (181.5,350.78) -- (152.93,378.6) ;
\draw [shift={(151.5,380)}, rotate = 315.76] [color={rgb, 255:red, 0; green, 0; blue, 0 }  ][line width=0.75]    (10.93,-3.29) .. controls (6.95,-1.4) and (3.31,-0.3) .. (0,0) .. controls (3.31,0.3) and (6.95,1.4) .. (10.93,3.29)   ;
\draw    (62,466.37) -- (167,466.37) ;
\draw    (60,457.37) -- (166,457.37) ;
\draw    (123,396.37) -- (102,502.37) ;
\draw    (107.5,390.67) -- (156.5,501.67) ;
\draw    (236,451.37) -- (342,451.37) ;
\draw    (255,417.75) -- (254,498.28) ;
\draw   (254,451.25) .. controls (254,442.69) and (267.21,435.75) .. (283.5,435.75) .. controls (299.79,435.75) and (313,442.69) .. (313,451.25) .. controls (313,459.81) and (299.79,466.75) .. (283.5,466.75) .. controls (267.21,466.75) and (254,459.81) .. (254,451.25) -- cycle ;
\draw    (126.5,519.37) -- (156.11,550.13) ;
\draw [shift={(157.5,551.57)}, rotate = 226.09] [color={rgb, 255:red, 0; green, 0; blue, 0 }  ][line width=0.75]    (10.93,-3.29) .. controls (6.95,-1.4) and (3.31,-0.3) .. (0,0) .. controls (3.31,0.3) and (6.95,1.4) .. (10.93,3.29)   ;
\draw    (277.5,520) -- (247.92,549.16) ;
\draw [shift={(246.5,550.57)}, rotate = 315.4] [color={rgb, 255:red, 0; green, 0; blue, 0 }  ][line width=0.75]    (10.93,-3.29) .. controls (6.95,-1.4) and (3.31,-0.3) .. (0,0) .. controls (3.31,0.3) and (6.95,1.4) .. (10.93,3.29)   ;
\draw    (169,561.75) -- (169,629.28) ;
\draw   (169,595.25) .. controls (169,586.69) and (182.21,579.75) .. (198.5,579.75) .. controls (214.79,579.75) and (228,586.69) .. (228,595.25) .. controls (228,603.81) and (214.79,610.75) .. (198.5,610.75) .. controls (182.21,610.75) and (169,603.81) .. (169,595.25) -- cycle ;
\draw    (166,561.75) -- (166,629.28) ;
\draw    (64.52,720.68) -- (140.88,792.74) ;
\draw    (68.6,714.9) -- (145.69,787.65) ;
\draw    (143.86,700.7) -- (72.49,776.34) ;
\draw    (148.24,704.81) -- (76.85,780.46) ;
\draw    (220,739.37) -- (290.5,739.75) ;
\draw    (227,725.88) -- (298,769.8) ;
\draw    (251,768.88) -- (326,726.8) ;
\draw    (147,914) -- (252,914) ;
\draw    (146,911) -- (252,911) ;
\draw    (191,849.02) -- (191,953.02) ;
\draw    (191,901.02) .. controls (195.5,897.78) and (200.5,904.78) .. (199,908) ;
\draw    (146,908) -- (252,908) ;
\draw    (265,743.37) -- (321.5,742.75) ;
\draw    (180.5,650.78) -- (151.91,679.52) ;
\draw [shift={(150.5,680.93)}, rotate = 314.86] [color={rgb, 255:red, 0; green, 0; blue, 0 }  ][line width=0.75]    (10.93,-3.29) .. controls (6.95,-1.4) and (3.31,-0.3) .. (0,0) .. controls (3.31,0.3) and (6.95,1.4) .. (10.93,3.29)   ;
\draw    (221.5,650.78) -- (249.11,679.49) ;
\draw [shift={(250.5,680.93)}, rotate = 226.11] [color={rgb, 255:red, 0; green, 0; blue, 0 }  ][line width=0.75]    (10.93,-3.29) .. controls (6.95,-1.4) and (3.31,-0.3) .. (0,0) .. controls (3.31,0.3) and (6.95,1.4) .. (10.93,3.29)   ;
\draw    (121.5,811.37) -- (149.12,840.43) ;
\draw [shift={(150.5,841.88)}, rotate = 226.46] [color={rgb, 255:red, 0; green, 0; blue, 0 }  ][line width=0.75]    (10.93,-3.29) .. controls (6.95,-1.4) and (3.31,-0.3) .. (0,0) .. controls (3.31,0.3) and (6.95,1.4) .. (10.93,3.29)   ;
\draw    (260.5,811) -- (232.92,838.47) ;
\draw [shift={(231.5,839.88)}, rotate = 315.12] [color={rgb, 255:red, 0; green, 0; blue, 0 }  ][line width=0.75]    (10.93,-3.29) .. controls (6.95,-1.4) and (3.31,-0.3) .. (0,0) .. controls (3.31,0.3) and (6.95,1.4) .. (10.93,3.29)   ;

\draw (438,66) node [anchor=north west][inner sep=0.75pt]   [align=left] {Dimension of the orbit};
\draw (499,96) node [anchor=north west][inner sep=0.75pt]   [align=left] {8};
\draw (500,265) node [anchor=north west][inner sep=0.75pt]   [align=left] {7};
\draw (500,436) node [anchor=north west][inner sep=0.75pt]   [align=left] {6};
\draw (501,585) node [anchor=north west][inner sep=0.75pt]   [align=left] {5};
\draw (500,725) node [anchor=north west][inner sep=0.75pt]   [align=left] {4};
\draw (499,897) node [anchor=north west][inner sep=0.75pt]   [align=left] {3};

\end{tikzpicture}

\end{table}
 \begin{enumerate}[A)]
    \item We begin by calculating the dimension of each orbit. These dimensions are indicated on Table \ref{table3}. The reasonings leading to the calculation of this dimension are, in fact simple, of the same type as those we will take up again later (\S 6.2) for the orbits of planes in $C_5$. So we do not give them here, except for the following example of type (1): To give a pencil of type 1 is equivalent to give 4 points of the projective plane, none being colinear with two others. There is a variety parametrizing subsets of the plane comprised of four points having this property, which is the quotient of an open in the product of four copies of the projective plane by a finite group, and that variety thus has dimension 8.
    \item The only specializations that we might envision given the dimension calculation, and the fact that the specialization of a type having an element of $D_2$ should, itself, contain an element of $D_2$, are the specializations indicated by the arrows of the Table, and their composites.
    \item
    One verifies by the following calculations that all the specializations
    exist.\par
    \begin{enumerate}[(1)]
  \item Existence of      
  
  \tikzset{every picture/.style={line width=0.75pt}} 
\begin{tikzpicture}[x=0.75pt,y=0.75pt,yscale=-0.9,xscale=0.9]

\draw    (169.05,58.2) -- (242.57,149.99) ;
\draw    (263.67,83.79) -- (166.49,143.1) ;
\draw    (147.02,88.79) -- (247.49,140.56) ;
\draw    (149.71,111.1) -- (268.31,106.58) ;
\draw    (223.12,52.07) -- (196.98,155.14) ;
\draw    (210.74,39.72) -- (231.51,159.26) ;
\draw    (301.5,111.55) -- (403.5,110.57) ;
\draw [shift={(405.5,110.55)}, rotate = 539.45] [color={rgb, 255:red, 0; green, 0; blue, 0 }  ][line width=0.75]    (10.93,-3.29) .. controls (6.95,-1.4) and (3.31,-0.3) .. (0,0) .. controls (3.31,0.3) and (6.95,1.4) .. (10.93,3.29)   ;
\draw    (541.37,87.37) -- (441.14,141.36) ;
\draw    (424.62,86.08) -- (522.16,143.19) ;
\draw    (502.59,53.51) -- (470.94,155.03) ;
\draw    (490.89,40.51) -- (505.19,161) ;

\draw (135,171.4) node [anchor=north west][inner sep=0.75pt]    {$\left[\left( X^{2} -Z^{2} ,Y^{2} -X^{2}\right)\right]$};
\draw (434,167.4) node [anchor=north west][inner sep=0.75pt]    {$[\left( X^{2} -Z^{2} ,( ZY)\right]$};

\end{tikzpicture}

    One considers the family of pencils: $U_t=[(X-Z)(X+tZ),(Y^2-X^2)]$. For $t\not=0,\ U_t$ belongs to the orbit of [$(X^2-Z^2),(Y^2-X^2)$], and for $t=0,\ U_0$ belongs to the orbit of [$(ZY),(X^2-Z^2)$].
    \par
    \item Existence of 

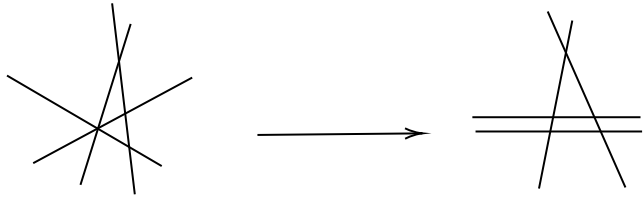
\tikzset{every picture/.style={line width=0.75pt}} 

\begin{tikzpicture}[x=0.75pt,y=0.75pt,yscale=-0.8,xscale=0.8]

\draw    (440,108.37) -- (545,108.37) ;
\draw    (438,99.37) -- (544,99.37) ;
\draw    (501,38.37) -- (480,144.37) ;
\draw    (485.5,32.67) -- (534.5,143.67) ;
\draw    (302.5,110.55) -- (404.5,109.57) ;
\draw [shift={(406.5,109.55)}, rotate = 539.45] [color={rgb, 255:red, 0; green, 0; blue, 0 }  ][line width=0.75]    (10.93,-3.29) .. controls (6.95,-1.4) and (3.31,-0.3) .. (0,0) .. controls (3.31,0.3) and (6.95,1.4) .. (10.93,3.29)   ;
\draw    (261.37,74.37) -- (161.14,128.36) ;
\draw    (144.62,73.08) -- (242.16,130.19) ;
\draw    (222.59,40.51) -- (190.94,142.03) ;
\draw    (210.89,27.51) -- (225.19,148) ;

\draw (136,162.4) node [anchor=north west][inner sep=0.75pt]    {$[( X^2-Z^2) ,( ZY)]$};
\draw (454,161.4) node [anchor=north west][inner sep=0.75pt]    {$[( XY) ,( Z^2)]$};

\end{tikzpicture}

    One considers the family of pencils $U_t=$ [$(X^2-tZ^2),(ZY)$]. The desired specialization occurs for $t=0$.
    \item Existence of 

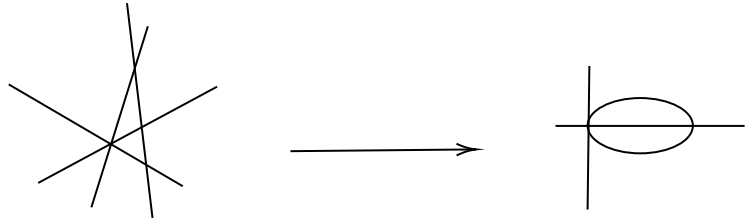
\tikzset{every picture/.style={line width=0.75pt}} 

\begin{tikzpicture}[x=0.75pt,y=0.75pt,yscale=-0.9,xscale=0.9]

\draw    (450,116.37) -- (556,116.37) ;
\draw    (469,82.75) -- (468,163.28) ;
\draw   (468,116.25) .. controls (468,107.69) and (481.21,100.75) .. (497.5,100.75) .. controls (513.79,100.75) and (527,107.69) .. (527,116.25) .. controls (527,124.81) and (513.79,131.75) .. (497.5,131.75) .. controls (481.21,131.75) and (468,124.81) .. (468,116.25) -- cycle ;
\draw    (301.5,130.55) -- (403.5,129.57) ;
\draw [shift={(405.5,129.55)}, rotate = 539.45] [color={rgb, 255:red, 0; green, 0; blue, 0 }  ][line width=0.75]    (10.93,-3.29) .. controls (6.95,-1.4) and (3.31,-0.3) .. (0,0) .. controls (3.31,0.3) and (6.95,1.4) .. (10.93,3.29)   ;
\draw    (260.37,94.37) -- (160.14,148.36) ;
\draw    (143.62,93.08) -- (241.16,150.19) ;
\draw    (221.59,60.51) -- (189.94,162.03) ;
\draw    (209.89,47.51) -- (224.19,168) ;

\draw (135,182.4) node [anchor=north west][inner sep=0.75pt]    {$[( X^2-Z^2) ,( ZY)]$};
\draw (453,181.4) node [anchor=north west][inner sep=0.75pt]    {$\left[( XY) ,XZ-Y^{2}\right]$};

\end{tikzpicture}

    One considers the family of pencils: $U_t=[(Z(X+tY)),(ZY-X^2)]$ which realizes for $t=0$ the looked for specialization.
\newpage 
    \item Existence of

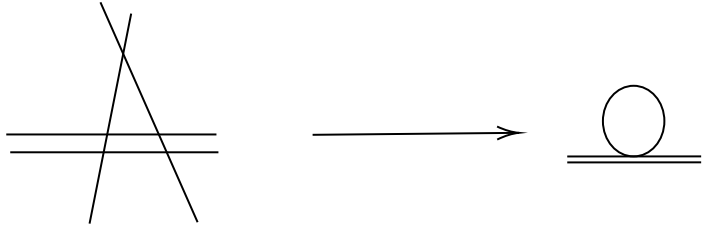
\tikzset{every picture/.style={line width=0.75pt}} 

\begin{tikzpicture}[x=0.75pt,y=0.75pt,yscale=-1,xscale=1]

\draw    (140,118.37) -- (245,118.37) ;
\draw    (138,109.37) -- (244,109.37) ;
\draw    (201,48.37) -- (180,154.37) ;
\draw    (185.5,42.67) -- (234.5,153.67) ;
\draw    (292.5,109.55) -- (394.5,108.57) ;
\draw [shift={(396.5,108.55)}, rotate = 539.45] [color={rgb, 255:red, 0; green, 0; blue, 0 }  ][line width=0.75]    (10.93,-3.29) .. controls (6.95,-1.4) and (3.31,-0.3) .. (0,0) .. controls (3.31,0.3) and (6.95,1.4) .. (10.93,3.29)   ;
\draw    (420.92,120.48) -- (488.46,120.44) ;
\draw   (454.42,120.46) .. controls (445.86,120.47) and (438.92,112.49) .. (438.91,102.64) .. controls (438.91,92.79) and (445.84,84.81) .. (454.4,84.8) .. controls (462.96,84.79) and (469.91,92.77) .. (469.91,102.62) .. controls (469.92,112.47) and (462.98,120.45) .. (454.42,120.46) -- cycle ;
\draw    (420.93,123.48) -- (488.46,123.44) ;
\end{tikzpicture}\par 
One considers the family $U_t=[(YZ-X^2-tZ^2),(Y^2)]$, which for $t=0$ realizes the specialization shown.
    \item 
    Existence of 

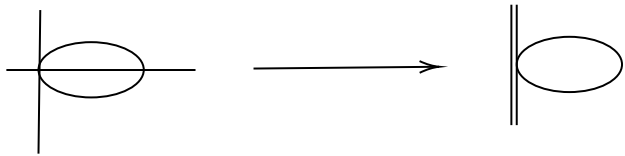
\tikzset{every picture/.style={line width=0.75pt}} 

\begin{tikzpicture}[x=0.75pt,y=0.75pt,yscale=-0.9,xscale=0.9]

\draw    (150,82.37) -- (256,82.37) ;
\draw    (169,48.75) -- (168,129.28) ;
\draw   (168,82.25) .. controls (168,73.69) and (181.21,66.75) .. (197.5,66.75) .. controls (213.79,66.75) and (227,73.69) .. (227,82.25) .. controls (227,90.81) and (213.79,97.75) .. (197.5,97.75) .. controls (181.21,97.75) and (168,90.81) .. (168,82.25) -- cycle ;
\draw    (436,45.75) -- (436,113.28) ;
\draw   (436,79.25) .. controls (436,70.69) and (449.21,63.75) .. (465.5,63.75) .. controls (481.79,63.75) and (495,70.69) .. (495,79.25) .. controls (495,87.81) and (481.79,94.75) .. (465.5,94.75) .. controls (449.21,94.75) and (436,87.81) .. (436,79.25) -- cycle ;
\draw    (433,45.75) -- (433,113.28) ;
\draw    (288.5,81.55) -- (390.5,80.57) ;
\draw [shift={(392.5,80.55)}, rotate = 539.45] [color={rgb, 255:red, 0; green, 0; blue, 0 }  ][line width=0.75]    (10.93,-3.29) .. controls (6.95,-1.4) and (3.31,-0.3) .. (0,0) .. controls (3.31,0.3) and (6.95,1.4) .. (10.93,3.29)   ;
\end{tikzpicture}

 One considers the family $U_t=[(Y(tX+Y)),(YZ-X^2)]$ as $t= 0$.
    \item 
    Existence of 

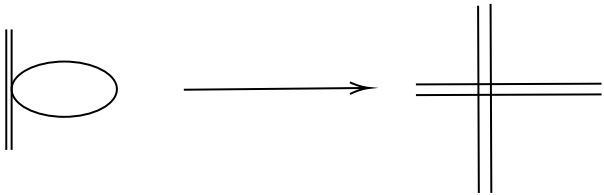
\tikzset{every picture/.style={line width=0.75pt}} 

\begin{tikzpicture}[x=0.75pt,y=0.75pt,yscale=-0.9,xscale=0.9]

\draw    (174,56.75) -- (174,124.28) ;
\draw   (174,90.25) .. controls (174,81.69) and (187.21,74.75) .. (203.5,74.75) .. controls (219.79,74.75) and (233,81.69) .. (233,90.25) .. controls (233,98.81) and (219.79,105.75) .. (203.5,105.75) .. controls (187.21,105.75) and (174,98.81) .. (174,90.25) -- cycle ;
\draw    (171,56.75) -- (171,124.28) ;
\draw    (435.41,43.44) -- (435.81,148.44) ;
\draw    (442.4,42.41) -- (442.81,148.41) ;
\draw    (504.56,87.17) -- (400.56,87.57) ;
\draw    (504.6,93.17) -- (400.58,93.57) ;
\draw    (270.5,90.55) -- (372.5,89.57) ;
\draw [shift={(374.5,89.55)}, rotate = 539.45] [color={rgb, 255:red, 0; green, 0; blue, 0 }  ][line width=0.75]    (10.93,-3.29) .. controls (6.95,-1.4) and (3.31,-0.3) .. (0,0) .. controls (3.31,0.3) and (6.95,1.4) .. (10.93,3.29)   ;
\end{tikzpicture}
\par The family $U_t=[(X^2),(tXZ-Y^2)]$ realizes for $t= 0$ the specialization.
    \item 
    Existence of 

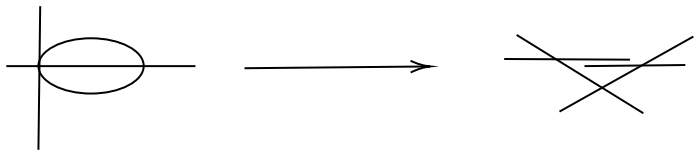
\tikzset{every picture/.style={line width=0.75pt}} 

\begin{tikzpicture}[x=0.75pt,y=0.75pt,yscale=-0.9,xscale=0.9]

\draw    (148,89.37) -- (254,89.37) ;
\draw    (167,55.75) -- (166,136.28) ;
\draw   (166,89.25) .. controls (166,80.69) and (179.21,73.75) .. (195.5,73.75) .. controls (211.79,73.75) and (225,80.69) .. (225,89.25) .. controls (225,97.81) and (211.79,104.75) .. (195.5,104.75) .. controls (179.21,104.75) and (166,97.81) .. (166,89.25) -- cycle ;
\draw    (427,85.37) -- (497.5,85.75) ;
\draw    (434,71.88) -- (505,115.8) ;
\draw    (458,114.88) -- (533,72.8) ;
\draw    (472,89.37) -- (528.5,88.75) ;
\draw    (281.5,90.55) -- (383.5,89.57) ;
\draw [shift={(385.5,89.55)}, rotate = 539.45] [color={rgb, 255:red, 0; green, 0; blue, 0 }  ][line width=0.75]    (10.93,-3.29) .. controls (6.95,-1.4) and (3.31,-0.3) .. (0,0) .. controls (3.31,0.3) and (6.95,1.4) .. (10.93,3.29)   ;
\end{tikzpicture}
\par Use the family $U_t=[(XY),(YZ-tX^2)]$ for $t= 0.$
    \item 

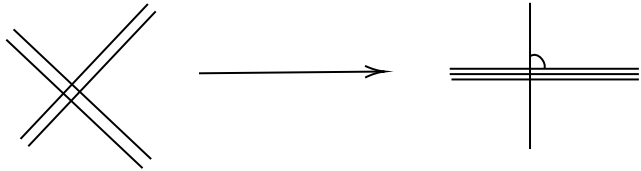
\tikzset{every picture/.style={line width=0.75pt}} 

\begin{tikzpicture}[x=0.75pt,y=0.75pt,yscale=-0.9,xscale=0.9]

\draw    (173.52,61.68) -- (249.88,133.74) ;
\draw    (177.6,55.9) -- (254.69,128.65) ;
\draw    (252.86,41.7) -- (181.49,117.34) ;
\draw    (257.24,45.81) -- (185.85,121.46) ;
\draw    (423,84) -- (528,84) ;
\draw    (422,81) -- (528,81) ;
\draw    (467,41.02) -- (467,123.02) ;
\draw    (467,71.02) .. controls (471.5,67.78) and (476.5,74.78) .. (475,78) ;
\draw    (422,78) -- (528,78) ;
\draw    (281.5,80.55) -- (383.5,79.57) ;
\draw [shift={(385.5,79.55)}, rotate = 539.45] [color={rgb, 255:red, 0; green, 0; blue, 0 }  ][line width=0.75]    (10.93,-3.29) .. controls (6.95,-1.4) and (3.31,-0.3) .. (0,0) .. controls (3.31,0.3) and (6.95,1.4) .. (10.93,3.29)   ;
\end{tikzpicture} \newline is realized by the
    family $U_t=[(X^2),(XY+tY^2)]$ as $t= 0.$
    \item 

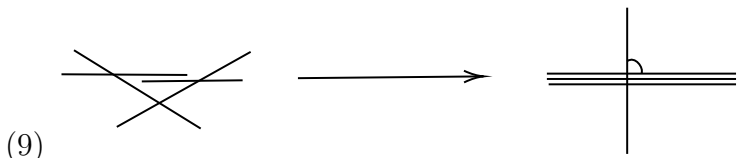
\tikzset{every picture/.style={line width=0.75pt}} 

\begin{tikzpicture}[x=0.75pt,y=0.75pt,yscale=-0.9,xscale=0.9]

\draw    (422,84) -- (527,84) ;
\draw    (421,81) -- (527,81) ;
\draw    (466,41.02) -- (466,123.02) ;
\draw    (466,71.02) .. controls (470.5,67.78) and (475.5,74.78) .. (474,78) ;
\draw    (421,78) -- (527,78) ;
\draw    (281.5,80.55) -- (383.5,79.57) ;
\draw [shift={(385.5,79.55)}, rotate = 539.45] [color={rgb, 255:red, 0; green, 0; blue, 0 }  ][line width=0.75]    (10.93,-3.29) .. controls (6.95,-1.4) and (3.31,-0.3) .. (0,0) .. controls (3.31,0.3) and (6.95,1.4) .. (10.93,3.29)   ;
\draw    (149,78.37) -- (219.5,78.75) ;
\draw    (156,64.88) -- (227,108.8) ;
\draw    (180,107.88) -- (255,65.8) ;
\draw    (194,82.37) -- (250.5,81.75) ;

\end{tikzpicture}
\newline  is realized by the family $U_t=[(XY),(X(tZ+X))]$ as $t= 0.$
    \end{enumerate}
\end{enumerate}
\section{Nets of conics}\label{netsection}
\setcounter{subsection}{-1}
\subsection{Introduction to \S4.1 and \S4.2: Classification of planes of conics}\label{intro4sec}\vskip 0.2cm
{\underline{Notations}:}\label{typenotation}
If $\pi$ is a plane of $C_5$ we denote by $\Gamma_\pi$ the schematic intersection of $\pi$ with $S_4$, the degenerate conics.  When that intersection is proper in $S_4$, then $\Gamma_\pi$ will be entirely characterized by the class of a cubic form in the projective space associated to $R_3$. We denote by  $\delta_\pi$ the schematic intersection of $\pi$ with the family of double lines $D_2$. 
\begin{definition}[Definition and Note]\label{typedef} One says that the sequences
$(\delta_\pi\subset \Gamma_\pi\subset \pi)$ and $(\delta_{\pi'}\subset \Gamma_{\pi'}\subset {\pi'})$ are isomorphic if there is a projective linear isomorphism from $\pi$ to $\pi'$ (that is, arising from an element of $Pgl(3)$ mapping $\mathbb P^5=\mathbb P(R_2)$ to itself) that transforms $\Gamma_\pi$ to $\Gamma_{\pi'}$ and $\delta_\pi$ to $\delta_{\pi'}$; or, equivalently, from an element of $\Pgl(5)$ stabilizing $S_4$ or $D_2$ (as in (2.5) of Section 2).
\vskip 0.2cm\par\noindent
The isomorphism class of a triple $(\delta_\pi\subset \Gamma_\pi\subset \pi)$ will be called its type. When $\Gamma_\pi$ is a cubic, a figure representing $\Gamma_\pi$ with the points of $\delta_\pi$ indicated, will symbolize the type of $(\delta_\pi\subset \Gamma_\pi\subset \pi)$. When $\Gamma_\pi=\pi$ the points or curve $ \delta$ pictured within a
hatched plane (that represents $\Gamma_\pi=\pi$) will symbolize the type of $(\delta_\pi\subset \Gamma_\pi\subset \pi)$.
\end{definition}
Examples:\quad\begin{tikzpicture}[baseline={([yshift=-.8ex]current bounding box.center)},scale=0.4]
\draw  [line width=0.75] [line join = round][line cap = round] (0,0) .. controls  (-0.5,0.4) and (-1.3,1.1) .. (-1.9,1.5) .. controls (-2.2,1.7) and (-2.8,1.5) .. (-2.85,1.1) .. controls (-2.9,0.9) and (-2.7,0.55) .. (-2.5,0.5) .. controls (-2,0.4) and (-1,0.8) .. (-0.6,0.9) .. controls (-0.5,0.94) and (0.103,1.0695) .. (0.1,1.15) ;
\draw (-0.9,0.8) circle (0.3);
\end{tikzpicture}\quad 
\begin{tikzpicture}[baseline={([yshift=-.8ex]current bounding box.center)},scale=0.4]
\draw (-0.8,1) .. controls (-0.3,0) and (0.7,0) .. (1.2,1);
\draw[color=gray] (0.5,-0.75) -- (1.5,1.25);
\draw[color=gray] (0.3,-0.7) -- (1.3,1.3);
\draw[color=gray] (0.1,-0.65) -- (1.1,1.35);
\draw[color=gray] (-0.1,-0.6) -- (0.9,1.4);
\draw[color=gray] (-0.3,-0.55) -- (0.7,1.45);
\draw[color=gray] (-0.5,-0.5) -- (0.5,1.5);
\draw[color=gray] (-0.7,-0.45) -- (0.3,1.55);
\draw[color=gray] (-0.9,-0.4) -- (0.1,1.6);
\draw[color=gray] (-1.1,-0.35) -- (-0.1,1.65);
\draw[color=gray] (-1.3,-0.3) -- (-0.3,1.7);
\draw[color=gray] (-1.5,-0.25) -- (-0.5,1.76);
\end{tikzpicture}.\par
On the other hand if $\phi$ is a smooth cubic with defining equation $F$, we will denote by $J_\phi$, the plane  of $C_5$  containing the conics of equations the partials
$F'_X,F'_Y,F'_Z,$ and we will term this the \emph{Jacobian plane} of the cubic $\phi.$
\vskip 0.2cm\par\noindent
\underline{ General introduction to the classification of the planes of $C_5$.} We show the following result. 
\setcounter{theorem}{0}
\begin{theorem}\label{mainthm}
There are two disjoint collections of isomorphism classes of planes $\pi\subset C_5$.
\begin{enumerate}[A)]
\item The set of classes of planes $\pi$ such that $\Gamma_\pi$ is smooth. For such a plane $\pi$, there is a unique cubic $\phi$ such that $\pi=J_\phi$. This is a smooth cubic, not isomorphic to the cubic $(X^3+Y^3+Z^3)$. The isomorphism class of the plane is completely determined by the isomorphism class of $\phi$. If $\sf j$ is the fundamental invariant of $\phi$ we will denote the corresponding class $\pi$ by the symbol 
\begin{tikzpicture}[scale=0.3]
\draw (0,0) circle (0.3);
\draw    (0.7,0.8) .. controls  (0,-0.3) and (1.2,0.2) .. (0.4,-0.8);
\end{tikzpicture}$_{\sf j}$, representing the smooth cubic indexed by its ${\sf j}$-invariant.
\item\label{classification} The set of classes of planes $\pi$ such that $\Gamma_\pi$ is not a smooth cubic. Each isomorphism class in this set is completely determined by the type of the sequence $(\delta_\pi\subset\Gamma_\pi\subset \pi)$ with respect to any one of its representatives. The admissible types are those underlying the isomorphism classes that are pictured in Table \ref{table1}.
\end{enumerate}
\end{theorem}
Demonstrating these results is the object of \S 4.1 and \S 4.2: \S 4.1 treats the case $\Gamma_\pi$ is smooth (Case 1) and \S 4.2 treats the case $\Gamma_\pi$ is singular
(Case 2-10).\vskip 0.2cm
Suppose we are given a plane $\pi$ of $C_5.$ We will organize our study in \S 4.2 following the situation of $\pi$ with respect to the locus $S_4$ of degenerate conics.\par
\paragraph{General classification of cases for $\Gamma$.}
\begin{enumerate}[I)]
\item We consider first the case $
\pi\nsubseteq S_4$. \par
In this case $\pi\cap S_4$, considered as a scheme, is a cubic $\Gamma_
\pi$. A priori, following the well known classification of cubics in the projective plane over an algebraically closed field\footnote{Establishing that classification is an easy exercise after \cite[Exercise 5.24]{Ful}.} one has to consider the following cases.
\begin{enumerate}[1)]
\item $\Gamma$ is smooth (\S 4.1): \begin{tikzpicture}[baseline={([yshift=-.8ex]current bounding box.center)},scale=0.4] 
\draw (0,0) circle (0.3);
\draw    (0.7,0.8) .. controls  (0,-0.3) and (1.2,0.2) .. (0.4,-0.8);
\end{tikzpicture};
\vskip 0.2cm\par {\bf Or} $\Gamma $ is singular\, (\S 4.2,\,\,\# 2-10):
\item $\Gamma$ is singular, with singularity a point having two distinct tangents:
\begin{tikzpicture}[scale=0.3]

\draw  [line width=0.75] [line join = round][line cap = round] (0,0) .. controls  (-0.5,0.4) and (-1.3,1.1) .. (-1.9,1.5) .. controls (-2.2,1.7) and (-2.8,1.5) .. (-2.85,1.1) .. controls (-2.9,0.9) and (-2.7,0.55) .. (-2.5,0.5) .. controls (-2,0.4) and (-1,0.8) .. (-0.6,0.9) .. controls (-0.5,0.94) and (0.103,1.0695) .. (0.1,1.15) ;
\end{tikzpicture}
;
\item $\Gamma$ is singular at a cusp:\quad \begin{tikzpicture}[baseline={([yshift=-.8ex]current bounding box.center)},scale=0.3]
\draw  [color={rgb, 255:red, 0; green, 0; blue, 0 }  ][line width=0.75] [line join = round][line cap = round] (0,1.3) .. controls (0.5,1.3) and (0.95,1) .. (1.2,0.75) .. controls (1.4,0.6) and (1.5,0.4) .. (1.7,0.15) ;
\draw  [color={rgb, 255:red, 0; green, 0; blue, 0 }  ][line width=0.75] [line join = round][line cap = round] (0,1.3) .. controls (0.5,1.3) and (0.95,1.6) .. (1.2,1.85) .. controls (1.4,2) and (1.5,2.2) .. (1.7,2.55) ;
\end{tikzpicture}
\item $\Gamma$ decomposes into a conic union a secant line:\quad\begin{tikzpicture}[baseline={([yshift=-.8ex]current bounding box.center)},scale=0.3]
\draw (0,0) circle (1);
\draw (-1.5,0) -- (1.5,0);
\end{tikzpicture};
\item $\Gamma$ decomposes into a conic union a tangent line:\quad\begin{tikzpicture}[scale=0.3]
\draw (0,0) circle (1);
\draw (-1.5,-1) -- (1.5,-1);
\end{tikzpicture}
;
\item $\Gamma$ decomposes into three non-concurrent lines: \begin{tikzpicture}[scale=0.2]
\draw (-0.4,0) -- (3.4,0);
\draw (-0.2,-0.4) -- (1.7,3.4);
\draw (3.2,-0.4) -- (1.3,3.4);
\end{tikzpicture};
\item $\Gamma$ decomposes into three concurrent lines: \begin{tikzpicture}[scale=0.4]
\draw (0,-0.5) -- (0,1);
\draw (-0.4,-0.4) -- (0.8,0.8);
\draw (0.4,-0.4) -- (-0.8,0.8);
\end{tikzpicture}
;
\item $\Gamma$ decomposes into a double line union a simple line: \begin{tikzpicture}[baseline={([yshift=-.8ex]current bounding box.center)},scale=0.15]
\draw (1,2) -- (1,-2);
\draw (-1,0.2) -- (3,0.2);
\draw (-1,-0.2) -- (3,-0.2);
\end{tikzpicture};
\item $\Gamma$ consists of a triple line:
\begin{tikzpicture}[scale=0.4]
\draw (0,0) -- (2,0);
\draw (0,-0.2) -- (2,-0.2);
\draw (0,0.2) -- (2,0.2);
\end{tikzpicture}.

\end{enumerate}
\item We consider next the case where $\pi\subset S_4$. There are two families of such planes, according to whether\par
(10a) $\delta_\pi$ consists of a parabola \begin{tikzpicture}[baseline={([yshift=-.8ex]current bounding box.center)},scale=0.4]
\draw (-0.8,1) .. controls (-0.3,0) and (0.7,0) .. (1.2,1);
\draw[color=gray] (0.5,-0.75) -- (1.5,1.25);
\draw[color=gray] (0.3,-0.7) -- (1.3,1.3);
\draw[color=gray] (0.1,-0.65) -- (1.1,1.35);
\draw[color=gray] (-0.1,-0.6) -- (0.9,1.4);
\draw[color=gray] (-0.3,-0.55) -- (0.7,1.45);
\draw[color=gray] (-0.5,-0.5) -- (0.5,1.5);
\draw[color=gray] (-0.7,-0.45) -- (0.3,1.55);
\draw[color=gray] (-0.9,-0.4) -- (0.1,1.6);
\draw[color=gray] (-1.1,-0.35) -- (-0.1,1.65);
\draw[color=gray] (-1.3,-0.3) -- (-0.3,1.7);
\draw[color=gray] (-1.5,-0.25) -- (-0.5,1.76);
\end{tikzpicture}\quad
 or\par
 (10b) $\delta_\pi$ consists of a point\quad 
\begin{tikzpicture}[baseline={([yshift=-.8ex]current bounding box.center)},scale=0.4]
\draw (0,0) circle (0.5);
\filldraw (0,0) circle (2pt);
\draw[color=gray] (-0.5,-1) -- (0.5,1);
\draw[color=gray] (-0.3,-1.1) -- (0.7,0.9);
\draw[color=gray] (-0.1,-1.2) -- (0.9,0.8);
\draw[color=gray] (0.1,-1.3) -- (1.1,0.7);
\draw[color=gray] (-0.7,-0.9) -- (0.3,1.1);
\draw[color=gray] (-0.9,-0.8) -- (0.1,1.2);
\draw[color=gray] (-1.1,-0.7) -- (-0.1,1.3);
\end{tikzpicture}.

\end{enumerate}
\subsection{The planes of conics whose associated cubic $\Gamma_\pi$ is smooth}\par
\underline{Introduction}.\par 
In this paragraph we will give first an exhaustion of the cases where $\Gamma_\pi$ is smooth, by examining the cubic
$\Gamma_\pi$  geometrically. Then we introduce from the representation of a plane $\pi$, a cubic $\phi$ such that $\pi=J_{\phi}$. Hence the classification (Section 4.1.2). Finally, we develop in Section 4.1.3 a new point of view by associating to each plane $J_{\phi}$ a third cubic $H$ endowed with a regular involution $i$ without double point such that:
$$\pi=\{\gamma : \gamma\in C_5,\ \forall x\in H,\ x \text{ and } i(x) \text{ are conjugate with respect to }\gamma \}.$$
Here $H$ is the Hessian of $\phi$, and the problem of classification done in Section~4.1 comes down to classifying the smooth cubics endowed with a regular involution without a double point.
\subsubsection{The case where the associated cubic is smooth (Case 1)}\label{4.1sec}

Let $\pi$ be a plane in $C_5$ such that $\Gamma_{\pi}$ is smooth. In what follows, we write $\Gamma$ for $\Gamma_{\pi}$. We can choose in $\pi$ distinct points $O,A,B,C,D$ (see figure below) 

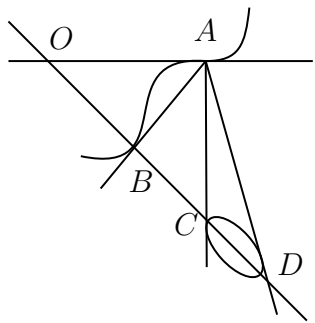
\begin{center}

\tikzset{every picture/.style={line width=0.75pt}} 

\begin{tikzpicture}[x=0.75pt,y=0.75pt,yscale=-1,xscale=1]

\draw    (100,100) -- (253.5,100) ;
\draw    (100,80) -- (250.5,231) ;
\draw    (199.5,100) -- (200,204) ;
\draw    (146.5,164.33) -- (199.5,100) ;
\draw    (199.5,100) -- (235.5,228) ;
\draw   (201.32,179.91) .. controls (204.92,176.63) and (213.53,180.24) .. (220.55,187.96) .. controls (227.57,195.68) and (230.35,204.59) .. (226.75,207.87) .. controls (223.14,211.14) and (214.53,207.54) .. (207.51,199.81) .. controls (200.49,192.09) and (197.72,183.18) .. (201.32,179.91) -- cycle ;
\draw    (136.5,148) .. controls (189.5,159) and (148.5,96) .. (199.5,100) ;
\draw    (199.5,100) .. controls (216.5,100) and (219.5,92) .. (221.5,73) ;

\draw (192,77.4) node [anchor=north west][inner sep=0.75pt]    {$A$};
\draw (234,195.4) node [anchor=north west][inner sep=0.75pt]    {$D$};
\draw (182,175.4) node [anchor=north west][inner sep=0.75pt]    {$C$};
\draw (159,153.4) node [anchor=north west][inner sep=0.75pt]    {$B$};
\draw (119,81.4) node [anchor=north west][inner sep=0.75pt]    {$O$};

\end{tikzpicture}

\end{center}

\noindent verifying the following properties: $A$ is an inflection point of $\Gamma$, $OA$ is the inflectional tangent to $\Gamma$ at $A$; $B, C, D$ are aligned with $O$ and belong to $\Gamma$, $AB,$ $AC,$ $AD$ being the respective tangents to $\Gamma$ in $B, C, D$. (This is the geometric version of putting the cubic equation in Weierstrass form).

According to Remark \ref{rmkS2}, none of the points of $\pi$ belong to $D_2$, because otherwise it would be a singular point of $\Gamma$. We use the notation of Table \ref{table2} to identify the type of pencils which then appear: the pencil OBCD is of type a); the pencils $AB,\ AC,\ AD$ are of type b); $AO$ is of type d). According to Table \ref{table2}, we can choose a coordinate system where: $B=(X^2-Y^2)$, $C=(Y^2-Z^2)$, $D=(Z^2-X^2)$. The bi-ratio $(O, A, B, C)$ then determines the equation of a conic $O$ which is of the form: $\kappa X^2-(1+\kappa)Y^2+Z^2=0$, with $\kappa$ uniquely determined and $\kappa\neq 0,-1$. The pencils $AB, AC, AD$ being of type b), the singular points of $B, C, D$ lie on the irreducible components of the conic determined by $A$ outside its singular point. $OA$ being of type d), an irreducible component of $A$ must be tangent to the conic determined by $O$ at the singular point of $A$. 

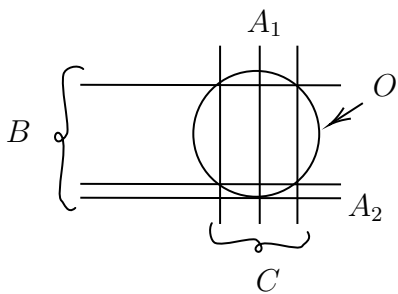
\begin{center}

\tikzset{every picture/.style={line width=0.75pt}} 

\begin{tikzpicture}[x=0.75pt,y=0.75pt,yscale=-1,xscale=1]

\draw    (100,110) -- (231.5,110.27) ;
\draw    (100,160) -- (231.5,160.27) ;
\draw   (157,134.75) .. controls (157,117.21) and (171.21,103) .. (188.75,103) .. controls (206.29,103) and (220.5,117.21) .. (220.5,134.75) .. controls (220.5,152.29) and (206.29,166.5) .. (188.75,166.5) .. controls (171.21,166.5) and (157,152.29) .. (157,134.75) -- cycle ;
\draw    (170.5,90.25) -- (170.5,180.25) ;
\draw    (190.5,90.25) -- (190.5,180.25) ;
\draw    (209.5,90.25) -- (209.5,180.25) ;
\draw    (242.5,119.25) -- (227.18,129.16) ;
\draw [shift={(225.5,130.25)}, rotate = 327.09000000000003] [color={rgb, 255:red, 0; green, 0; blue, 0 }  ][line width=0.75]    (10.93,-3.29) .. controls (6.95,-1.4) and (3.31,-0.3) .. (0,0) .. controls (3.31,0.3) and (6.95,1.4) .. (10.93,3.29)   ;
\draw    (100,167) -- (231.5,167.27) ;
\draw  [color={rgb, 255:red, 0; green, 0; blue, 0 }  ][line width=0.75] [line join = round][line cap = round] (101.5,102.06) .. controls (98.64,99.45) and (92.28,101.23) .. (91.5,104.79) .. controls (90.08,111.24) and (94.13,127.89) .. (93.5,133.91) .. controls (93.32,135.65) and (90.31,139.29) .. (88.5,138.47) .. controls (85.69,137.19) and (88.15,131.47) .. (90.5,135.73) .. controls (94.08,142.25) and (87.23,165.52) .. (90.5,168.5) .. controls (91.98,169.84) and (93.52,175.77) .. (97.5,172.14) ;
\draw  [color={rgb, 255:red, 0; green, 0; blue, 0 }  ][line width=0.75] [line join = round][line cap = round] (166.35,179.84) .. controls (164.53,182.69) and (165.69,189.06) .. (168.14,189.86) .. controls (172.57,191.32) and (184.1,187.39) .. (188.24,188.06) .. controls (189.44,188.25) and (191.92,191.29) .. (191.33,193.09) .. controls (190.42,195.89) and (186.5,193.39) .. (189.47,191.07) .. controls (194,187.53) and (209.99,194.54) .. (212.07,191.29) .. controls (213.01,189.82) and (217.11,188.33) .. (214.65,184.32) ;

\draw (246,103.4) node [anchor=north west][inner sep=0.75pt]    {$O$};
\draw (233.5,163.67) node [anchor=north west][inner sep=0.75pt]    {$A_{2}$};
\draw (182.5,70.67) node [anchor=north west][inner sep=0.75pt]    {$A_{1}$};
\draw (61,126.4) node [anchor=north west][inner sep=0.75pt]    {$B$};
\draw (187,201.4) node [anchor=north west][inner sep=0.75pt]    {$C$};

\end{tikzpicture}

\end{center}

Finally, the conic $A$ is determined as follows: one of its irreducible components $A_1$ must pass through two of the singular points $B,C$ and $D$. The other component $A_2$ is the tangent to $O$ in one of the two intersection points of $A_1$ and $O$. The last condition ($A_2$ passes through the singular point of the third conic) is also ipso facto satisfied. Then there are six possibilities for $A$, two by two isomorphic: 
$$A=(Y[\kappa^{1/2}X\pm iZ]),\ (X[Z\pm(1+\kappa)^{1/2}Y]),\ (Z[\kappa^{1/2}X\pm (1+\kappa)^{1/2}Y]).$$
By reordering $B, C, D$ we therefore have in all cases the following presentation of $\pi$: $$B=(X^2-Y^2);\ C=(Y^2-Z^2);\  D=(Z^2-X^2);$$ $$A = (X[Z+(1+\kappa)^{1/2}Y]);\ O=(\kappa X^2-(1+\kappa)Y^2+Z^2).$$
A change of coordinates leads to a new representation of $\pi$: $$\pi = (YZ-X^2), (Y-Z)(Y-\alpha Z), (XY); \ \alpha \neq 0,1.$$
We have then the following proposition. 
\begin{proposition}
Let $\pi$ be a plane in $C_5$ such that $\Gamma_{\pi}$ is a smooth cubic. There exists a homogeneous coordinate system in P such that: 
$$\pi=[(YZ-X^2), (Y-Z)(Y-\alpha Z), (XY)] \text{ with } \alpha \neq 0,1.$$
Conversely, any plane $\pi$ in $C_5$ having in a suitable coordinate system such a presentation has an associated smooth cubic $\Gamma_{\pi}$.
\end{proposition}

\subsubsection{Jacobian plane of a smooth cubic}
The above provides an exhaustion of all planes $\pi$ such that $\Gamma_{\pi}$ is smooth, but is not in itself a classification up to $Pgl(3)$ isomorphism. To obtain this classification, let us prove the following proposition:

\begin{proposition}\label{4.2prop}
Let $\pi$ be a plane of $C_5$ such that $\Gamma_{\pi}$ is smooth.
\begin{enumerate}[a)]
    \item There exists a unique cubic projective plane curve $\phi$ of $P$ such that $\pi=J_{\phi}$.
    \item This cubic $\phi$ is smooth, not isomorphic to $X^3+Y^3+Z^3$.
    \item Conversely, for every smooth cubic $\phi$ not isomorphic to $X^3+Y^3+Z^3$, $\Gamma_{J_{\phi}}$ is smooth. 
    \item The fundamental invariant ${\sf j}_\phi$ of the cubic $\phi$ classifies the plane $\pi$ up to $Pgl(3)$ isomorphism.
\end{enumerate}
\end{proposition}
Let $\pi$ be a plane of $C_5$ and $V$ the corresponding vector space in $R_2$. By Section 4.1.1, up to a change of coordinates we can suppose that there exists $\alpha\neq 0,1$ such that $V$ is generated by $YZ-X^2,(Y-Z)(Y-\alpha Z),XY$. We want to find an element $F=\phi$ of $R_3$  whose partial derivatives of order 1 also generate $V$. It is therefore a matter of finding $a_i,b_i,c_i$ such that:
$$\exists F\in R_3,\ \forall i,\ 1\leq i\leq 3\ \Rightarrow\ a_i(YZ-X^2)+b_iXY+c_i(Y-Z)(Y-\alpha Z)=F'_{X_i} $$ 
where $X_1=X,$ $X_2=Y,$ $X_3=Z$.

This condition is equivalent to the following condition:
\small
\begin{align}
&\forall i, \forall j \mid i,j\in\{1,2,3\}, i\neq j\text { we have }\notag\\
&\frac{d}{dX_i}[a_j(YZ-X^2)+b_jXY+c_j(Y-Z)(Y-\alpha Z)]=\notag\\&\frac{d}{dX_j}[a_i(YZ-X^2)+b_iXY+c_i(Y-Z)(Y-\alpha Z)].\label{partialeq}
\end{align}
\normalsize
By identifying the coefficients of the monomials of the two sides (which are of degree 1, therefore there are three coefficients) of each of these three relations, we obtain a system of 9 linear equations with 9 unknowns (the $a_i$, $b_i$, $c_i$) easy to solve, which is exactly of rank 8; we find:
\begin{equation}\label{2eq}\begin{pmatrix} a_1&b_1&c_1\\a_2&b_2&c_2\\a_3&b_3&c_3\end{pmatrix} = \rho \begin{pmatrix} 0 & 2(\alpha-1)^2&0\\ -(\alpha-1)^2&0&-(\alpha+1)\\ 0&0&2\alpha\end{pmatrix}.
\end{equation}
So $F$ therefore exists, unique up to multiplication by a non-zero scalar, and is defined by\small
\begin{align} 3F&=XF'_X+YF'_Y+ZF'_Z\notag\\&=\frac{\rho}{3}(3(\alpha-1)^2X^2Y+2\alpha^2Z^3-3\alpha(\alpha+1)YZ^2+6\alpha Y^2Z-(\alpha+1)Y^3).\label{3eqn}
\end{align}
\normalsize
We verify that $\{F'_X,F'_Y,F'_Z\}$ is of rank 3 and thus indeed constitutes a system of generators of V, and we take $\phi: F=0$. Which ends up proving a).

Let us show b). The partial derivatives of $X^3 + Y^3 + Z^3$ generate a vector space $V$ whose associated plane $\pi$ has a cubic $\Gamma_{\pi}$ constituted by three non-intersecting lines. Therefore $\phi$  is not isomorphic to $(X^3 + Y^3 + Z^3)$ since otherwise $\Gamma_{J_{\phi}}$ would not be smooth. 

Let us show in the same way that $\phi$ is not singular, that is, $\phi$ is not a specialization of a family of cubics isomorphic to $(X^3+Y^3+XYZ)=\phi_0$ (double point cubics with distinct tangents, \#8a in Table \ref{table1}). However, $\Gamma_{J_{\phi_0}}$ is not smooth, and this property holds for a cubic isomorphic to $\phi_0$ and for specializations. Therefore, $\phi$ is not singular.\qquad\qquad\qquad\qquad\qquad \qquad \qquad Q.E.D.

Let us now show c). If $\phi$ is smooth, we can find a system of
coordinates in which $\phi$ has ``Hessian form": $\phi = (X^3 + Y^3 + Z^3 + 3\lambda XYZ)$, where $\lambda \neq -1, -\omega, -\omega^2$ ($\omega$ and $\omega^2$ are the cube roots of 1).  That $\phi$  is not isomorphic to $(X^3+Y^3+Z^3)$ means then that $\lambda\neq 0,2,2\omega,2\omega^2$. Saying that ${(a F'_X+b F'_Y+c F'_Z)}$ is a singular conic is equivalent to the relation: 
\begin{equation}(2^3+2\lambda^3)abc-2\lambda^2(a^3+b^3+c^3)=0.\label{Gammaeq}
\end{equation} 
``$\Gamma_{J_{\phi}}$ is smooth'' is equivalent then to $-\frac{2^3+2\lambda^3}{2\lambda^2}\neq -3, -3\omega, -3\omega^2$ and $\lambda\neq 0$. The first condition decomposes into the conditions: $\lambda\neq -1,2,-\omega,2\omega,-\omega^2,2\omega^2$ which are all satisfied by hypothesis. \qquad \qquad\qquad \qquad\qquad \qquad Q.E.D.

d) is clear. 

\paragraph{Presentation in the Table 1.} \label{Table1pres} Considering the cubic $\phi$ of the form: $$(X^3+Y^3+Z^3+3\lambda XYZ),\ \lambda\neq -1,-\omega,-\omega^2,0,2,2\omega,2\omega^2,$$ we write the plane $\pi$ in the form $(X^2+\lambda YZ,Y^2+\lambda ZX, Z^2+\lambda XY)$. In doing so, we do not forget that it is not $\lambda$ which classifies $\pi$, but rather the fundamental invariant ${\sf j}_\lambda$ of $\phi$, which is the rational function of $\lambda$ noted in the table ${\sf j}_\lambda$.\footnote{{\it Note of translators}: this formula is found in \cite{Serre}, as follows:\par
p. 138, (18): Weierstrass formula $y^2=4x^3-g_2x-g_3$;\par
p.138,  (15): definition of the discriminant $\Delta=g_2^3-27g_3^2$;\par
p.145 (22): expression for the ${\sf j}$-invariant, ${\sf j}=1728g_2^3/\Delta$.\par
Any homography - rational function of this $\sf j$ can be used as an invariant, such as  $\Delta/(g_3)^2$ or $\Delta/(g_2)^3$. As we noted earlier, \cite[Proposition 2]{Fr} identifies ${\sf j}_\lambda=
\frac{\lambda^3(\lambda^3-8)^3}{27(\lambda +1)^3(\lambda +\omega)^3(\lambda+\omega^2)^3}.$ The mapping $\lambda\to {\sf j}_\lambda$ for smooth cubics $\phi$ is $12:1$, except for the homographic cubic $\phi_0=X^3+Y^3+Z^3$, where the mapping is $4:1$ \cite[Prop. 3]{Fr}.
We corrected some of the notation to avoid a double use of $\beta$ in the original.}
\footnote{The Hessian cubic $\Gamma(V), V=J_\phi$ has $\lambda^\prime=-\frac{4+\lambda^3}{3\lambda^2}$ in terms of $\lambda$ for $\phi$, when written $\Gamma(V):X^3+Y^3+Z^3 +3\lambda^\prime XYZ$ (from \cite[Theorem~4.7]{CilOt}).}
\begin{rmk}\label{betarem}
The equation $\frac{2^3+2\lambda^3}{2\lambda^2}=\beta$ generically admits three roots $\alpha_1, \alpha_2, \alpha_3$ such that the curves, $(X^3 + Y^3 + Z^3 + 3\alpha_i XYZ)$, $1 \leq i \leq 3$ are two by two non-isomorphic. Thus, generally, given an isomorphism type  ${\sf j}_\Gamma$ of a smooth cubic curve $\Gamma$, there exist three types of plane $\pi$ (that is, isomorphism classes ${\sf j}_\phi$ of nets), giving rise to the type of curve $\Gamma$: that is, $\Gamma_\phi$ is in the class ${\sf j}_\Gamma$. 
\end{rmk}
\vskip 0.2cm\noindent
\subsubsection{Canonical involution of the Hessian of a smooth cubic}\label{involutionsec}
Let $\pi$ be a plane in $C_5$. A conic belongs then to $\pi$ if  three linear conditions on its coefficients hold. Among such conditions there are some remarkable ones, such as being forced to go through a point, a condition which is obviously very particular. A remarkable condition of a more general nature is that the conic admits two conjugate points. We will therefore wonder if, being given a plane $\pi$ of conics, there are enough pairs of conjugate points with respect to all the conics of $\pi$ to determine $\pi$.

In the case where $\Gamma_{\pi}$ is smooth, the answer is positive. It is given with other information by the following proposition whose proof is easy.

\begin{proposition}\label{4.4prop}
Let $\pi$ be a plane in $C_5$ such that $\Gamma_{\pi}$ is smooth. Let $\phi$ be the cubic such that $\pi=J_{\phi}$ and let $H$ be the Hessian of $\phi$.
\begin{enumerate}[a)]
    \item For a point of $P$ to admit a common conjugate with respect to all the conics of $\pi$, it is necessary and sufficient that it belongs to $H$. Its common conjugate is unique, distinct from it, and belongs also to $H$. There is then a canonical and regular involution on $H$ without double point, $i: H \rightarrow H$ such that $${\pi} = \{\gamma\ :\ \gamma \in C_5 \text{ and } \forall x\in H,\ x \text{ and } i(x) \text{ are conjugate with respect to } \gamma\}$$
    \item Conversely, given a smooth cubic $H$, endowed with a regular involution, without double point $i$, the set $$\{\gamma\ :\ \gamma \in C_5 \text{ and } \forall x\in H,\ x \text{ and } i(x) \text{ are conjugate with respect to } \gamma\},$$ is a plane $\pi$ of $C_5$ such that $\Gamma_{\pi}$ is smooth. If $\phi$ is the cubic such that $\pi = J_{\phi}$, $H$ is the Hessian of $\phi$ and $i$ is  obtained on $H$ by the process described in the direct proposition in (a).
    \item The classification of smooth cubics endowed with a regular involution without a double point is therefore equivalent to the classification of planes $\pi$ of $C_5$ such that $\Gamma_{\pi}$ is smooth.
    \item There are exactly three involutions without double points on a smooth cubic which, as soon as we have fixed an algebraic group structure on the cubic (i.e. an origin), can be considered as the translations by the three points of order two of this group.
\end{enumerate}
\end{proposition}

A smooth cubic is the Hessian of exactly three smooth cubics not isomorphic to $(X^3+Y^3+Z^3)$. 

The three involutions without a double point of a smooth cubic and their three cubics of which the cubic is the Hessian correspond one-to-one by the process described in a).

We have an equivalent definition of the cubic $H$ thanks to the following proposition, whose proof is left to the reader:

\begin{proposition}\label{Hesse2prop}
Under the same hypotheses as the previous proposition, $H$ is the set of all the singular points of the singular conics of $\pi$. Given a singular point $A$ of a singular conic of $\pi$, its polar with respect to $\phi$ is a singular conic of $\pi$, whose unique singular point $B$ is the image of  $A$ by the  involution $i$ of $H$. The map from $\Gamma_\pi$ to $H$ which sends the singular conic $\gamma$ to its unique singular point is an isomorphism from $\Gamma_\pi$ to $H$.
\end{proposition}
See also the ``Note on the Hessian form of a smooth cubic'' at the end of the article, page \pageref{Hessiansec}. 

\subsection{The planes $\pi$ of conics having a singular cubic $\Gamma_\pi$.}
We will henceforth not give the full demonstration that allows us to arrive at the presentation of each case of this exhaustion. We will indicate only how to choose a system of conics in the plane $\pi$ that is distinguished by its disposition relative to $\Gamma_\pi$, and the conclusions that follow relative to the different pencils of conics to consider. The verifications will be easy.  We take again the order of examination adopted in the general Introduction to \S1 and \S2. The Case~1 (smooth) was treated in \S4.1, so we begin here with Case 2.\par
Recall the conclusion of this section, which was announced as Theorem~\ref{mainthm}B in \S 4.0, the general introduction to \S4: the $Pgl(3)$ class of a plane $\pi$ for which $\Gamma_\pi$ is not a smooth cubic is completely determined by the linear isomorphism class of the triplet of projective algebraic schemes $(\delta_\pi\subset \Gamma_\pi\subset \pi$).
On the figures representing the plane $\pi$, the points associated to double lines are noted: $\circ$. We follow the classification that was introduced in ``General classification of cases for $\Gamma$'' in the introduction  \S \ref{intro4sec} to \S 4. In I. we consider $\pi\nsubseteq S_4$ (Case 2-9), and in II. $\pi\subset S_4$ (Case 10a,10b).\vskip 0.2cm\noindent
I. $\pi \nsubseteq S_4$.\par\noindent
\underline{Case 2}). ($\Gamma$ has singularity a point $C$ having two distinct tangents). Two cases occur:\par
\begin{enumerate}[a)]
\item C has two distinct irreducible components ( \begin{tikzpicture}[scale=0.2]

\draw  [line width=0.75] [line join = round][line cap = round] (0,0) .. controls  (-0.5,0.4) and (-1.3,1.1) .. (-1.9,1.5) .. controls (-2.2,1.7) and (-2.8,1.5) .. (-2.85,1.1) .. controls (-2.9,0.9) and (-2.7,0.55) .. (-2.5,0.5) .. controls (-2,0.4) and (-1,0.8) .. (-0.6,0.9) .. controls (-0.5,0.94) and (0.103,1.0695) .. (0.1,1.15) ;
\end{tikzpicture}).

Choose then in $\pi$ points $A,B,C,O$ satisfying the following conditions: $C$ is the double point, $A$ is a point of inflection, $AB$ is the tangent at $B$ to $\Gamma, O\in BC$ and $AO$ is the tangent at $A$ to $\Gamma$. 
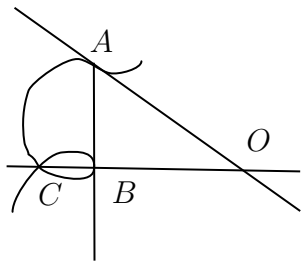
\begin{center}
\tikzset{every picture/.style={line width=0.75pt}} 

\begin{tikzpicture}[x=0.75pt,y=0.75pt,yscale=-1,xscale=1]

\draw    (253.57,205.91) -- (105.5,203.32) ;
\draw    (253.58,225.91) -- (109.5,123.32) ;
\draw    (149.83,250.95) -- (149.5,151.08) ;
\draw  [color={rgb, 255:red, 0; green, 0; blue, 0 }  ][line width=0.75] [line join = round][line cap = round] (108.5,226.32) .. controls (108.5,217.21) and (124.3,197.15) .. (133.5,196.32) .. controls (142.02,195.54) and (151.12,195.6) .. (149.5,205.32) .. controls (148.16,213.38) and (126.73,208.55) .. (122.5,204.32) .. controls (120.36,202.17) and (120.7,199.72) .. (116.5,197.32) .. controls (116.65,195.62) and (114.83,193.64) .. (114.5,191.32) .. controls (113.12,181.65) and (113.84,173.3) .. (116.5,165.32) .. controls (118.01,160.79) and (135.91,147.67) .. (142.5,149.32) .. controls (147.96,150.68) and (152.79,155.37) .. (157.5,156.32) .. controls (162.32,157.28) and (173.5,154.07) .. (173.5,150.32) ;

\draw (146,133.4) node [anchor=north west][inner sep=0.75pt]    {$A$};
\draw (157,209.4) node [anchor=north west][inner sep=0.75pt]    {$B$};
\draw (120,210.4) node [anchor=north west][inner sep=0.75pt]    {$C$};
\draw (225,183.4) node [anchor=north west][inner sep=0.75pt]    {$O$};

\end{tikzpicture}
\end{center}
One verifies that $BC, AB, AC$ are of type $b$ and $OA$ is of type $d$. From this we have a relatively unique disposition of $A,B,C,O$ represented by the figure below:
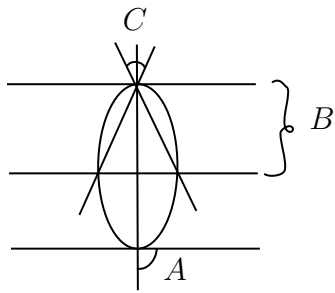
\begin{center}

\tikzset{every picture/.style={line width=0.75pt}} 

\begin{tikzpicture}[x=0.75pt,y=0.75pt,yscale=-1,xscale=1]

\draw   (135.19,105) .. controls (146.23,105.06) and (155.09,123.69) .. (154.96,146.6) .. controls (154.84,169.52) and (145.79,188.05) .. (134.74,187.99) .. controls (123.69,187.93) and (114.84,169.3) .. (114.97,146.39) .. controls (115.09,123.47) and (124.14,104.94) .. (135.19,105) -- cycle ;
\draw    (123.5,83.97) -- (164.5,168.97) ;
\draw    (134.5,84.97) -- (135,209) ;
\draw    (69,105) -- (194.5,105) ;
\draw    (70,150) -- (195.5,150) ;
\draw    (71,188) -- (196.5,188) ;
\draw    (143.5,85.97) -- (105.5,170.97) ;
\draw    (144.5,188) .. controls (144.5,193.45) and (139.5,199.45) .. (134.5,198) ;
\draw    (129.5,96) .. controls (131.5,92) and (138.5,94) .. (138.5,97) ;
\draw  [color={rgb, 255:red, 0; green, 0; blue, 0 }  ][line width=0.75] [line join = round][line cap = round] (199.02,150.43) .. controls (201.89,152.15) and (208.24,150.95) .. (209.02,148.59) .. controls (210.43,144.31) and (206.36,133.28) .. (206.98,129.28) .. controls (207.16,128.13) and (210.16,125.71) .. (211.97,126.26) .. controls (214.78,127.1) and (212.33,130.9) .. (209.97,128.07) .. controls (206.38,123.76) and (213.2,108.31) .. (209.92,106.34) .. controls (208.45,105.46) and (206.9,101.53) .. (202.92,103.95) ;

\draw (146.5,191.4) node [anchor=north west][inner sep=0.75pt]    {$A$};
\draw (126,65.4) node [anchor=north west][inner sep=0.75pt]    {$C$};
\draw (220,115.4) node [anchor=north west][inner sep=0.75pt]    {$B$};

\end{tikzpicture}

\end{center}
From this, we have in a convenient system of homogeneous coordinates, the following representation:
$$A=(XY),B=Z(Y-Z), C=(X^2-Z^2),O=(YZ-X^2).$$
\item $C$ is a double line \begin{tikzpicture}[baseline={([yshift=-.8ex]current bounding box.center)},scale=0.25]
\draw  [line width=0.75] [line join = round][line cap = round] (0,0) .. controls  (-0.5,0.4) and (-1.3,1.1) .. (-1.9,1.5) .. controls (-2.2,1.7) and (-2.8,1.5) .. (-2.85,1.1) .. controls (-2.9,0.9) and (-2.7,0.55) .. (-2.5,0.5) .. controls (-2,0.4) and (-1,0.8) .. (-0.6,0.9) .. controls (-0.5,0.94) and (0.103,1.0695) .. (0.1,1.15) ;
\draw (-0.9,0.8) circle (0.35);
\end{tikzpicture}\par One chooses points $A,B,C,O$ in the plane $\pi$ satisfying the same conditions as in a). $BC$ and $AC$ are of type $c)$, $AB$ is of type $b)$, $OA$ is of type $d)$. From this we have the following disposition:  

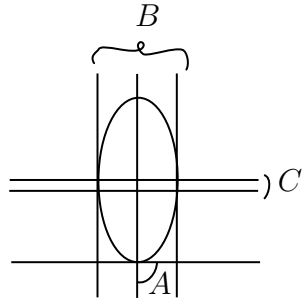
\begin{center}
    \tikzset{every picture/.style={line width=0.75pt}} 

\begin{tikzpicture}[x=0.75pt,y=0.75pt,yscale=-1,xscale=1]

\draw   (135.19,105) .. controls (146.23,105.06) and (155.09,123.69) .. (154.96,146.6) .. controls (154.84,169.52) and (145.79,188.05) .. (134.74,187.99) .. controls (123.69,187.93) and (114.84,169.3) .. (114.97,146.39) .. controls (115.09,123.47) and (124.14,104.94) .. (135.19,105) -- cycle ;
\draw    (134.5,92.97) -- (134.5,206.03) ;
\draw    (70.21,146.5) -- (195.71,146.5) ;
\draw    (70,152) -- (195.5,152) ;
\draw    (71,188) -- (196.5,188) ;
\draw    (114.5,92.75) -- (114.5,205.75) ;
\draw    (144.5,188) .. controls (144.5,193.45) and (139.5,199.45) .. (134.5,198) ;
\draw    (198.5,144.03) .. controls (202.5,145.03) and (201.5,153.03) .. (198.5,156.03) ;
\draw  [color={rgb, 255:red, 0; green, 0; blue, 0 }  ][line width=0.75] [line join = round][line cap = round] (159.32,90.79) .. controls (161,87.9) and (159.72,81.56) .. (157.34,80.82) .. controls (153.04,79.47) and (142.07,83.7) .. (138.07,83.14) .. controls (136.91,82.98) and (134.45,80.01) .. (134.97,78.19) .. controls (135.77,75.37) and (139.6,77.77) .. (136.81,80.16) .. controls (132.55,83.81) and (117,77.22) .. (115.09,80.52) .. controls (114.22,82.01) and (110.32,83.62) .. (112.79,87.56) ;
\draw    (154.5,92.75) -- (154.5,205.75) ;

\draw (138.5,191.4) node [anchor=north west][inner sep=0.75pt]    {$A$};
\draw (204,139.4) node [anchor=north west][inner sep=0.75pt]    {$C$};
\draw (132,56.4) node [anchor=north west][inner sep=0.75pt]    {$B$};
\end{tikzpicture}
\end{center}

and a system of coordinates where $\pi$ is represented by the following equations:
$$A=(X(Y+Z), B=(X^2-Z^2), C=(Y^2), O=(X^2+Y^2-Z^2).$$
\end{enumerate}
\underline{Case 3}). ($\Gamma$ is singular at a cusp):\par
\begin{enumerate}[a)]
\item Suppose that the cusp corresponds to a conic having two distinct irreducible components. One verifies the impossibility of this eventuality.
\item  Suppose that the cusp corresponds to a double line (
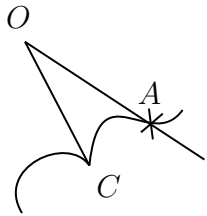
\begin{tikzpicture}[scale=0.2]
\draw  [color={rgb, 255:red, 0; green, 0; blue, 0 }  ][line width=0.75] [line join = round][line cap = round] (0,1.3) .. controls (0.5,1.3) and (0.95,1) .. (1.2,0.75) .. controls (1.4,0.6) and (1.5,0.4) .. (1.7,0.15) ;
\draw  [color={rgb, 255:red, 0; green, 0; blue, 0 }  ][line width=0.75] [line join = round][line cap = round] (0,1.3) .. controls (0.5,1.3) and (0.95,1.6) .. (1.2,1.85) .. controls (1.4,2) and (1.5,2.2) .. (1.7,2.55) ;
\draw (0,0.6) -- (0,2);
\draw (0.5,0.8) -- (-0.5,1.8);
\draw (-0.5,0.8) -- (0.5,1.8);
\end{tikzpicture}).
\begin{center}
    
\tikzset{every picture/.style={line width=0.75pt}} 

\begin{tikzpicture}[x=0.75pt,y=0.75pt,yscale=-1,xscale=1]

\draw    (120,51) -- (210.5,110.47) ;
\draw    (120,51) -- (152.5,113.47) ;
\draw    (152.5,113.47) .. controls (157.5,64.45) and (183.5,107.45) .. (199.5,85.45) ;
\draw    (118.5,137.45) .. controls (104.5,114.45) and (141.5,97.45) .. (152.5,113.47) ;
\draw    (182.5,84.72) -- (184.5,100.27) ;
\draw    (189.5,87.27) -- (178.5,96.72) ;

\draw (154.5,116.87) node [anchor=north west][inner sep=0.75pt]    {$C$};
\draw (109,31.4) node [anchor=north west][inner sep=0.75pt]    {$O$};
\draw (175.5,69.87) node [anchor=north west][inner sep=0.75pt]    {$A$};

\end{tikzpicture}
\end{center}
One chooses in the plane $\Pi$ points $A,O,C$ such that:
$C$ is the cusp, $A$ is an inflection point, $OC$ is the tangent to $\Gamma$ at $C$, $OA$ is the tangent to $\Gamma$ at $A$.  One then finds that $AC$ is of type c), $OA$ of type d), $OC$ of type e). From this we have the disposition in the figure below,
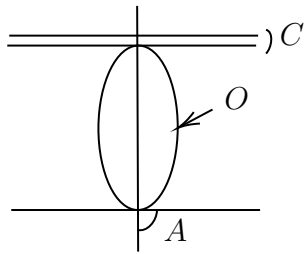
\begin{center}

\tikzset{every picture/.style={line width=0.75pt}} 

\begin{tikzpicture}[x=0.75pt,y=0.75pt,yscale=-1,xscale=1]

\draw   (135.19,105) .. controls (146.23,105.06) and (155.09,123.69) .. (154.96,146.6) .. controls (154.84,169.52) and (145.79,188.05) .. (134.74,187.99) .. controls (123.69,187.93) and (114.84,169.3) .. (114.97,146.39) .. controls (115.09,123.47) and (124.14,104.94) .. (135.19,105) -- cycle ;
\draw    (134.5,84.97) -- (135,209) ;
\draw    (69,105) -- (194.5,105) ;
\draw    (70,99.98) -- (195.5,99.98) ;
\draw    (71,188) -- (196.5,188) ;
\draw    (144.5,188) .. controls (144.5,193.45) and (139.5,199.45) .. (134.5,198) ;
\draw    (199.5,97.03) .. controls (203.5,98.03) and (202.5,106.03) .. (199.5,109.03) ;
\draw    (172.5,137.28) -- (156.73,145.67) ;
\draw [shift={(154.96,146.6)}, rotate = 332.01] [color={rgb, 255:red, 0; green, 0; blue, 0 }  ][line width=0.75]    (10.93,-3.29) .. controls (6.95,-1.4) and (3.31,-0.3) .. (0,0) .. controls (3.31,0.3) and (6.95,1.4) .. (10.93,3.29)   ;

\draw (146.5,191.4) node [anchor=north west][inner sep=0.75pt]    {$A$};
\draw (205,92.4) node [anchor=north west][inner sep=0.75pt]    {$C$};
\draw (177,124.4) node [anchor=north west][inner sep=0.75pt]    {$O$};

\end{tikzpicture}
\end{center}
and the following equations after a suitable change of coordinates:
$$A=(XY),C=(Z^2),O=(YZ-X^2).$$
\end{enumerate}
\underline{Case 4}. ($\Gamma$ decomposes into a conic union a secant line.) Suppose the intersection of these two irreducible components of $\Gamma$ are two distinct points $C$ and $D$. Three cases can be envisioned:
\begin{enumerate}[a)]
\item Neither $C$ nor $D$
is a double line (\begin{tikzpicture}[scale=0.2]
\draw (0,0) circle (1);
\draw (-1.5,0) -- (1.5,0);
\end{tikzpicture}). Choose $O$ the intersection point of the tangents at $C$ and $D$ to the conic that is one of the irreducible components of $\Gamma$. $OC$ and $OD$ are of type d). $CD$ is of type f).

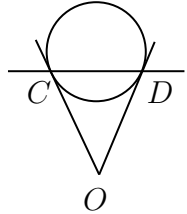
\begin{center}
	
	\tikzset{every picture/.style={line width=0.75pt}} 
	
	\begin{tikzpicture}[x=0.75pt,y=0.75pt,yscale=-1,xscale=1]
		
		\draw   (100,140) .. controls (100,126.19) and (111.19,115) .. (125,115) .. controls (138.81,115) and (150,126.19) .. (150,140) .. controls (150,153.81) and (138.81,165) .. (125,165) .. controls (111.19,165) and (100,153.81) .. (100,140) -- cycle ;
		\draw    (80.5,149.83) -- (169.5,149.83) ;
		\draw    (154.5,135.02) -- (126.5,202.02) ;
		\draw    (94.5,134.02) -- (126.5,202.02) ;
		
		\draw (88.5,153.23) node [anchor=north west][inner sep=0.75pt]    {$C$};
		\draw (149,153.4) node [anchor=north west][inner sep=0.75pt]    {$D$};
		\draw (117,207.4) node [anchor=north west][inner sep=0.75pt]    {$O$};

	\end{tikzpicture}
	
\end{center}

From this we have the following disposition:
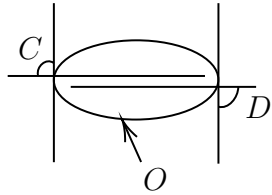
\begin{center}

\tikzset{every picture/.style={line width=0.75pt}} 

\begin{tikzpicture}[x=0.75pt,y=0.75pt,yscale=-1,xscale=1]

\draw   (176.46,148.33) .. controls (176.5,159.37) and (157.96,168.4) .. (135.05,168.5) .. controls (112.13,168.59) and (93.51,159.71) .. (93.47,148.66) .. controls (93.43,137.62) and (111.97,128.59) .. (134.88,128.5) .. controls (157.8,128.4) and (176.42,137.28) .. (176.46,148.33) -- cycle ;
\draw    (70.21,146.5) -- (169.71,146.5) ;
\draw    (102,152) -- (195.5,152) ;
\draw    (93.5,109.75) -- (93.5,190.02) ;
\draw    (186.5,152) .. controls (186.5,157.45) and (181.5,163.45) .. (176.5,162) ;
\draw    (85.5,146) .. controls (84.5,142) and (89.5,137) .. (93.5,140) ;
\draw    (176.5,109.75) -- (176.5,190.75) ;
\draw    (137.5,189.92) -- (128.78,170.43) ;
\draw [shift={(127.96,168.6)}, rotate = 425.9] [color={rgb, 255:red, 0; green, 0; blue, 0 }  ][line width=0.75]    (10.93,-3.29) .. controls (6.95,-1.4) and (3.31,-0.3) .. (0,0) .. controls (3.31,0.3) and (6.95,1.4) .. (10.93,3.29)   ;

\draw (74,125.4) node [anchor=north west][inner sep=0.75pt]    {$C$};
\draw (188.5,155.4) node [anchor=north west][inner sep=0.75pt]    {$D$};
\draw (137,191.4) node [anchor=north west][inner sep=0.75pt]    {$O$};
\end{tikzpicture}
\end{center} 

and the system of equations in a convenient coordinate system:
$$C=(XZ),D=(YZ),O=(XY-Z^2).$$

\item Only one of the two conics $C,D$ is a double line.  This case never occurs.
\item $C$ and $D$ are each a double line.  Choose $O$ as before: $OC$ and $OD$ are of type e). $CD$ is of type h).From these we have the disposition below: 
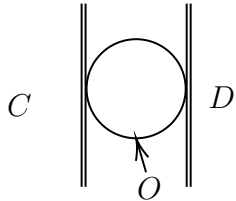
\begin{center}

\tikzset{every picture/.style={line width=0.75pt}} 

\begin{tikzpicture}[x=0.75pt,y=0.75pt,yscale=-1,xscale=1]

\draw   (100,140) .. controls (100,126.19) and (111.19,115) .. (125,115) .. controls (138.81,115) and (150,126.19) .. (150,140) .. controls (150,153.81) and (138.81,165) .. (125,165) .. controls (111.19,165) and (100,153.81) .. (100,140) -- cycle ;
\draw    (99.5,97.02) -- (99.5,189.52) ;
\draw    (97.5,97.02) -- (97.5,189.52) ;
\draw    (152.5,97.02) -- (152.5,189.52) ;
\draw    (150.5,97.02) -- (150.5,189.52) ;
\draw    (130,182.17) -- (125.56,166.92) ;
\draw [shift={(125,165)}, rotate = 433.76] [color={rgb, 255:red, 0; green, 0; blue, 0 }  ][line width=0.75]    (10.93,-3.29) .. controls (6.95,-1.4) and (3.31,-0.3) .. (0,0) .. controls (3.31,0.3) and (6.95,1.4) .. (10.93,3.29)   ;

\draw (58.5,141.23) node [anchor=north west][inner sep=0.75pt]    {$C$};
\draw (160,137.4) node [anchor=north west][inner sep=0.75pt]    {$D$};
\draw (124,183.4) node [anchor=north west][inner sep=0.75pt]    {$O$};

\end{tikzpicture}
\end{center}
and the following equations in a suitably chosen coordinate system:
$$C=(X^2),D=(Y^2), O=(XY-Z^2).$$
\end{enumerate}
\underline{Case 5} ($\Gamma$ decomposes into a conic union a tangent line):
Let $\gamma$ be the conic and $\delta$ the line. Let $U=\delta\cap\gamma$ be the point of tangency. Two cases may be envisioned:
\begin{enumerate}[a)]
    \item  $U$ is not a double line. This case does not occur.
    \item $U$ is a double line. Let $A$ and $B$
be two elements of $\pi$, $A$ on $\delta$ and $B$ on $\gamma$, both distinct from $U$, such that $AB$ is the tangent at $B$ to $\gamma$. $AB$ is of type b),$AU$ of type g),$UB$ of type c). From this we have the following disposition:
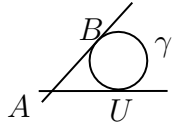
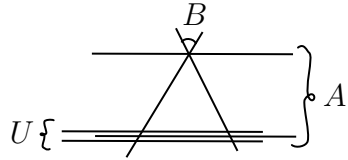
\begin{center}
    \tikzset{every picture/.style={line width=0.75pt}} 

\begin{tikzpicture}[x=0.75pt,y=0.75pt,yscale=-0.8,xscale=0.8]

\draw   (150.74,67.83) .. controls (160.33,64.9) and (170.48,70.29) .. (173.42,79.89) .. controls (176.36,89.48) and (170.96,99.64) .. (161.37,102.57) .. controls (151.78,105.51) and (141.62,100.11) .. (138.68,90.52) .. controls (135.75,80.93) and (141.14,70.77) .. (150.74,67.83) -- cycle ;
\draw    (187.06,104.3) -- (105.8,104.38) ;
\draw    (164.81,48.57) -- (107.8,111.38) ;
\draw    (140,184) -- (178.8,261.38) ;
\draw    (108.8,265.38) -- (156.8,186.38) ;
\draw    (87,200) -- (213.8,200.38) ;
\draw    (68,249) -- (194.8,249.38) ;
\draw    (89,252) -- (215.8,252.38) ;
\draw    (68,255) -- (194.8,255.38) ;
\draw  [color={rgb, 255:red, 0; green, 0; blue, 0 }  ][line width=0.75] [line join = round][line cap = round] (213.05,257.65) .. controls (215.92,259.93) and (222.27,258.35) .. (223.05,255.23) .. controls (224.45,249.57) and (220.37,234.99) .. (220.99,229.71) .. controls (221.17,228.19) and (224.17,224.99) .. (225.98,225.71) .. controls (228.79,226.82) and (226.34,231.84) .. (223.98,228.1) .. controls (220.39,222.4) and (227.19,201.98) .. (223.92,199.39) .. controls (222.44,198.21) and (220.89,193.02) .. (216.91,196.21) ;
\draw    (143.5,193) .. controls (145.5,189) and (152.5,191) .. (152.5,194) ;
\draw  [color={rgb, 255:red, 0; green, 0; blue, 0 }  ][line width=0.75] [line join = round][line cap = round] (62.92,242.04) .. controls (61.28,241.35) and (57.63,241.77) .. (57.17,242.69) .. controls (56.34,244.36) and (58.62,248.71) .. (58.24,250.27) .. controls (58.13,250.72) and (56.39,251.65) .. (55.35,251.42) .. controls (53.74,251.07) and (55.17,249.6) .. (56.51,250.72) .. controls (58.55,252.44) and (54.55,258.44) .. (56.42,259.23) .. controls (57.26,259.59) and (58.13,261.14) .. (60.43,260.22) ;

\draw (129,56.4) node [anchor=north west][inner sep=0.75pt]    {$B$};
\draw (84,105.4) node [anchor=north west][inner sep=0.75pt]    {$A$};
\draw (148.43,107.74) node [anchor=north west][inner sep=0.75pt]    {$U$};
\draw (177,68.4) node [anchor=north west][inner sep=0.75pt]    {$\gamma $};
\draw (142,166.4) node [anchor=north west][inner sep=0.75pt]    {$B$};
\draw (231,217.4) node [anchor=north west][inner sep=0.75pt]    {$A$};
\draw (33.43,240.74) node [anchor=north west][inner sep=0.75pt]    {$U$};

\end{tikzpicture}
\end{center} and the equations in a convenient coordinate system: $A=(ZY),$ $B=(X^2+Z^2),\ U=(Y^2).$ We may change coordinates to have the presentation $\langle X^2,XY,(Y+Z)Z\rangle.$
\par\noindent
{\underline{Note}}: The plane $\pi$ in this case can be defined in the following way. There are two points $A$ and $B$
of $P$ and a line $\delta$ of $P$ that contains neither $A$ nor $B$ such that the conics of $\pi$ are exactly the conics passing through $A,B$ and such that the polar of $AB$ with respect to each point belongs to $\delta$.\end{enumerate}
\par\noindent
{\underline{Case 6}} ($\Gamma$ decomposes into three non-concurrent lines): We should consider a priori the following four cases: \par
\begin{enumerate}[a)]
\item \begin{tikzpicture}[baseline={([yshift=-1.6ex]current bounding box.center)},scale=0.2]
\draw  (-0.5,0) -- (3.5,0) node [anchor = west] {C};
\draw  (1.8,3.4) -- (-0.1,-0.6);
\draw (3.1,-0.6) -- (1.2,3.4);

\draw (1.5,3.5) node [anchor = south]{A};
\draw (-0.4,0) node [anchor = east]{B};
\end{tikzpicture}. None of $A,B,C$ are double lines. One obtains the following disposition and the possible representation: 

$A=(XY), B=(YZ), C=(ZX),$ \quad \tikzset{every picture/.style={line width=0.75pt}} 
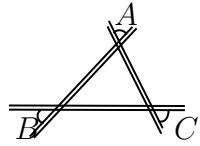
\begin{tikzpicture}[baseline={([yshift=-.8ex]current bounding box.center)},x=0.75pt,y=0.75pt,yscale=-0.7,xscale=0.7]

\draw    (173,85.66) -- (211.8,163.04) ;
\draw    (116,169) -- (190,90) ;
\draw    (101,146.66) -- (227.8,147.04) ;
\draw    (101,149.66) -- (227.8,150.04) ;
\draw    (176.5,94.66) .. controls (178.5,90.66) and (185.5,92.66) .. (185.5,95.66) ;
\draw    (172,89.66) -- (210.8,167.04) ;
\draw    (119,170) -- (193,91) ;
\draw    (216,150) .. controls (216,155.45) and (213.5,159.45) .. (208.5,158) ;
\draw    (124,150) .. controls (119,153) and (123,160) .. (126,159) ;

\draw (103,153.06) node [anchor=north west][inner sep=0.75pt]    {$B$};
\draw (175,71.06) node [anchor=north west][inner sep=0.75pt]    {$A$};
\draw (218,153.4) node [anchor=north west][inner sep=0.75pt]    {$C$};

\end{tikzpicture}

\par
\underline{Note}: The plane $\pi$ is constituted by the conics passing through the three given points, that are not collinear.\par
\item\begin{tikzpicture}[baseline={([yshift=-1.8ex]current bounding box.center)},scale=0.2]
\draw  (-0.5,0) -- (3.5,0);
\draw  (1.8,3.6) -- (-0.1,-0.6);
\draw (3.1,-0.6) -- (1.2,3.6);
\draw (1.5,3) circle (0.5);
\end{tikzpicture}.\quad One only of the three points $A,B,C$ is a double line. This case does not occur.\par
\item  \begin{tikzpicture}[baseline={([yshift=-1.4ex]current bounding box.center)},scale=0.2]
\draw  (-0.5,0) -- (3.5,0);
\draw  (1.8,3.6) -- (-0.1,-0.6);
\draw (3.1,-0.6) -- (1.2,3.6);
\draw (0,0) circle (0.5);
\draw (3,0) circle (0.5);
\end{tikzpicture}. \quad  Two of the three points $A,B,C$ are double lines.  This case does not occur.\par
\item \begin{tikzpicture}[baseline={([yshift=-1.9ex]current bounding box.center)},scale=0.2]
\draw  (-0.5,0) -- (3.5,0) node [anchor = west] {C};
\draw  (1.8,3.6) -- (-0.1,-0.6);
\draw (3.1,-0.6) -- (1.2,3.6);
\draw (1.5,3) circle (0.5);
\draw (0,0) circle (0.5);
\draw (3,0) circle (0.5);
\draw (1.5,3.5) node [anchor = south]{A};
\draw (-0.4,0) node [anchor = east]{B};
\end{tikzpicture} $A,B$ and $C$ are each double lines. One has the following disposition and the possible representation:

$A=(X^2),B=(Y^2),C=(Z^2)$, \quad \tikzset{every picture/.style={line width=0.75pt}} 
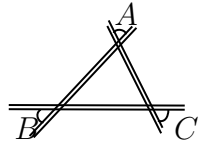
\begin{tikzpicture}[baseline={([yshift=-.8ex]current bounding box.center)},x=0.75pt,y=0.75pt,yscale=-0.7,xscale=0.7]

\draw    (173,85.66) -- (211.8,163.04) ;
\draw    (116,169) -- (190,90) ;
\draw    (101,149.66) -- (227.8,150.04) ;
\draw    (176.5,94.66) .. controls (178.5,90.66) and (185.5,92.66) .. (185.5,95.66) ;
\draw    (172,89.66) -- (210.8,167.04) ;
\draw    (119,170) -- (193,91) ;
\draw    (216,150) .. controls (216,155.45) and (213.5,159.45) .. (208.5,158) ;
\draw    (124,150) .. controls (119,153) and (123,160) .. (126,159) ;
\draw    (101,146.66) -- (227.8,147.04) ;

\draw (103,153.06) node [anchor=north west][inner sep=0.75pt]    {$B$};
\draw (175,71.06) node [anchor=north west][inner sep=0.75pt]    {$A$};
\draw (218,153.4) node [anchor=north west][inner sep=0.75pt]    {$C$};

\end{tikzpicture}
\par
\underline{Note} The plane $\pi$ is constituted in this case by the conics admitting a given autopolar triangle.
\end{enumerate}\par
7) $\Gamma$ decomposes into three concurrent lines.\quad  \begin{tikzpicture}[baseline={([yshift=-1.2ex]current bounding box.center)},scale=0.4] 
\draw (0,-0.5) -- (0,1);
\draw (-0.4,-0.4) -- (0.8,0.8);
\draw (0.4,-0.4) -- (-0.8,0.8);
\end{tikzpicture} \par This case never occurs, with or without double line.\vskip 0.3cm\par
8) ($\Gamma$ decomposes into a double line union a single line):\quad \begin{tikzpicture}[baseline={([yshift=-.8ex]current bounding box.center)},scale=0.2] 
\draw (1,2) -- (1,-2);
\draw (-1,0.2) -- (3,0.2);
\draw (-1,-0.2) -- (3,-0.2);
\end{tikzpicture} \par
Three cases present themselves: 
\begin{enumerate}[a)]
\item  \begin{tikzpicture}[baseline={([yshift=-.8ex]current bounding box.center)},scale=0.4] 
\draw (0.1,1.6)--(0.9,1.6);
\draw (0.5,2) node[anchor=west]{C};
\draw (0.5,2) -- (0.5,-2);
\draw (2,0) circle(0.5);
\draw (2,-0.5) node [anchor=north]{B};
\draw (0,-0.5) node [anchor=north]{A};
\draw (-1,0.2) -- (3,0.2);
\draw (-1,-0.2) -- (3,-0.2);
\end{tikzpicture} \quad $A=(XZ), C=(XY), B=(Z^2)$.\par
\underline{Note} Here it is a question of conics restricted to pass through two given points, and being tangent at one of the two points to a given double line that does not pass through the other point.
\item 
\begin{tikzpicture}[baseline={([yshift=-.8ex]current bounding box.center)},scale=0.4] 
\draw (0.1,1.6)--(0.9,1.6);
\draw (0.5,2) node[anchor=west]{C};
\draw (0.5,2) -- (0.5,-2);
\draw (2,0) circle(0.5);
\draw (2,-0.5) node [anchor=north]{B};
\draw (0,-0.5) node [anchor=north]{A};
\draw (-1,0.2) -- (3,0.2);
\draw (-1,-0.2) -- (3,-0.2);
\draw (0.5, 0) circle (0.5);
\end{tikzpicture} $A=(Y^2),B=(Z^2), C=(XY)$.\par
\underline{Note}: It is a question of conics satisfying the following conditions:
There are two given points $m$ and $m^\prime$ of $P$ that are conjugate with respect to each.  There is a point $n\in P, n\notin [m,m^\prime]$, the line determined by $m$ and $m^\prime$, and a line $\delta$ of $P$, with $m,n\in \delta$ such that each conic passes through $n$ and is tangent at $n$ to $\delta$.
\par
\item  
\begin{tikzpicture}[baseline={([yshift=-.8ex]current bounding box.center)},scale=0.4] 
\draw (0.4,1.6)--(1.4,1.6);
\draw (1,2) node[anchor=west]{C};
\draw (1,2) -- (1,-2);
\draw (2.3,0) circle(0.5);
\draw (2.3,-0.5) node [anchor=north]{V};
\draw (-0.3,-0.5) node [anchor=north]{U};
\draw (0.7,-0.2) node [anchor=north]{A};
\draw (-1.2,0.2) -- (3.2,0.2);
\draw (-1.2,-0.2) -- (3.2,-0.2);
\draw (-0.3, 0) circle (0.5);
\end{tikzpicture} $U=(X^2), V=(Y^2),A=(X^2-Y^2),C=((X+Y)(Z))$.\par
\underline{Note}: There are $m,m^\prime\in P$, $n\in P-[m,m^\prime]$ and $ \delta $ a line of $P$ with $n\in \delta$ but $ m,m^\prime\notin \delta$ such that the conics of $\pi$ are those that admit the pair
$m$ and $m^\prime$ as conjugates, that also pass through $n$, and that are tangent at $n$ to $\delta$.\par
{\underline{Note}}: One can also put the vector space in the form $\langle ZX,ZY,X^2+Z^2\rangle$.
\end{enumerate}
9) ($\Gamma$ consists of a triple line):\quad \begin{tikzpicture}[scale=0.5]
\draw (0,0) -- (2,0);
\draw (0,-0.2) -- (2,-0.2);
\draw (0,0.2) -- (2,0.2);
\end{tikzpicture}\quad Only the following case occurs:\vskip 0.2cm\par
 \begin{tikzpicture}[scale=0.5]
\draw (0,0) -- (2,0);
\draw (0,-0.2) -- (2,-0.2);
\draw (0,0.2) -- (2,0.2);
\draw (1,0) circle (0.4);
\end{tikzpicture} $O=(YZ-X^2), A=(Y^2), B=(XY)$.\par
   \begin{tikzpicture}[baseline={([yshift=-.8ex]current bounding box.center)},scale=0.5]
\draw (0,0) -- (3,0);
\draw (0,-0.2) -- (3,-0.2);
\draw (0,0.2) -- (3,0.2);
\draw (1,0) circle (0.4);
\draw (2.5,0.5) -- (2.5,-0.5);
\draw (2.5,0.5) node[anchor=south]{B};
\draw (1,0.5) node[anchor=south]{A};
\draw (1,-1) node[anchor=north]{O};
\filldraw (1,-1) circle (2pt);
\end{tikzpicture} ($OB$ is of type d), $OA$ of type e), $AB$ of type g)).\par
  \underline{Note}: The plane $\pi$ is constituted by the osculating conics at a given point to a given conic.
\vskip 0.3cm\noindent
  II. {$\pi\subset S_4$.} Two cases present themselves:
  \begin{enumerate}[a)]
  \item[10a)] $\pi\cap D_2$ is a parabola.\par
  \begin{tikzpicture}[baseline={([yshift=-.8ex]current bounding box.center)},scale=0.6]
\draw (-0.8,1) .. controls (-0.3,0) and (0.7,0) .. (1.2,1);
\draw (-0.8,1) circle (0.3);
\draw (-0.9,1) node [anchor=east]{A};
\draw (1.2,1) circle (0.3);
\draw (1.3,1) node [anchor=west]{B};
\draw (0.2,0.25) circle (0.3);
\draw (0.2,0) node [anchor=north]{C};
\draw[color=gray] (0.5,-0.75) -- (1.5,1.25);
\draw[color=gray] (0.3,-0.7) -- (1.3,1.3);
\draw[color=gray] (0.1,-0.65) -- (1.1,1.35);
\draw[color=gray] (-0.1,-0.6) -- (0.9,1.4);
\draw[color=gray] (-0.3,-0.55) -- (0.7,1.45);
\draw[color=gray] (-0.5,-0.5) -- (0.5,1.5);
\draw[color=gray] (-0.7,-0.45) -- (0.3,1.55);
\draw[color=gray] (-0.9,-0.4) -- (0.1,1.6);
\draw[color=gray] (-1.1,-0.35) -- (-0.1,1.65);
\draw[color=gray] (-1.3,-0.3) -- (-0.3,1.7);
\draw[color=gray] (-1.5,-0.25) -- (-0.5,1.76);
\end{tikzpicture} $A=(X^2), B=(Y^2), C=((X+Y)^2)$.\par
   \item[10b)] $\pi\cap D_2$ is reduced to a point.\par
    
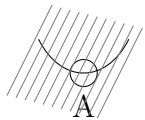
\begin{tikzpicture}[baseline={([yshift=-.8ex]current bounding box.center)},scale=0.6]
\draw (-0.8,1) .. controls (-0.3,0) and (0.7,0) .. (1.2,1);
\draw (0.2,0.25) circle (0.3);
\draw (0.2,0) node [anchor=north]{A};
\draw[color=gray] (0.5,-0.75) -- (1.5,1.25);
\draw[color=gray] (0.3,-0.7) -- (1.3,1.3);
\draw[color=gray] (0.1,-0.65) -- (1.1,1.35);
\draw[color=gray] (-0.1,-0.6) -- (0.9,1.4);
\draw[color=gray] (-0.3,-0.55) -- (0.7,1.45);
\draw[color=gray] (-0.5,-0.5) -- (0.5,1.5);
\draw[color=gray] (-0.7,-0.45) -- (0.3,1.55);
\draw[color=gray] (-0.9,-0.4) -- (0.1,1.6);
\draw[color=gray] (-1.1,-0.35) -- (-0.1,1.65);
\draw[color=gray] (-1.3,-0.3) -- (-0.3,1.7);
\draw[color=gray] (-1.5,-0.25) -- (-0.5,1.76);
\end{tikzpicture}\quad $A=(X^2),B=(XY),C=(XZ)$.
\end{enumerate}
\vskip 0.2cm\par\noindent
    \underline{Conclusion: Classification of the planes $\pi$ of $C_5$}.\par
    The preceding study constitutes not only an exhaustion of all the cases, but also a classification. In effect, there are no isomorphisms between any two of the cases described: for, if there were, the dispositions of $\pi\cap D_2\subset \pi\cap S_4$ would be algebraically isomorphic, which is never the case. The classification of planes $\pi$ up to $Pgl(3)$ action is thus well described by Table \ref{table1}. 

\section{Algebraic method for the classification of planes of conics}

We prove with an algebraic method the following proposition.
\begin{proposition}\label{6.1prop}
Every three dimensional vector space $V$ either has an associated smooth cubic $\Gamma(V)$ \footnote{If $V$ is a vector space of dimension 3 of $R_2$, let $\pi$ be the correspondent plane of $C_5$. We denote by $\Gamma(V)$ the cubic $\Gamma_{\pi}$.} or belongs to one of the 14 special orbits in Table \ref{table1}.
\end{proposition}\par
There are two stages in this exposition: in Section 5.1, we notice with the help of intersection theory on the Grassmannian G = $\Grass(2, R_2)$, that every vector space $V$ of dimension 3 of $R$ contains at least one 2 dimensional subspace that belongs to the closure of the orbit under $Pgl(3)$ of the vector space $W = \langle XY, X^2 + YZ\rangle $. We know, from Section 3, that this closure is the union of the orbits of $W$, $\langle Y^2, X^2 + YZ\rangle $, $\langle XY,XZ\rangle$, $\langle XY,X^2\rangle$ and of $\langle X^2,Y^2\rangle$ (see tables \ref{table2} and \ref{table3}). Then, in Section 5.2, we show the possibilities for a third generator of $V$ up to an action of an element of $Pgl(3)$. We then check the list of special orbits in Table \ref{table1}. We find that when $\Gamma(V)$ is smooth, $V$ is isomorphic to a vector space $W_\lambda, \lambda\neq 0,1$ of a one-parameter family. However, we need the considerations of Section 4.1.2 to make the classification in this case, which then completes the general classification.

\subsection{The intersections on $\Grass(2,R_2$)}

The Grassmannian variety $G = \Grass(2, R_2)$ which parametrizes the 2-dimensional vector subspaces of $R_2$ is of dimension 8. If $V$ is a vector subspace of dimension 3 of $R_2$, there is a plane $P(V)$ in $G$ which parametrizes the 2 dimensional subspaces of $V$. The class of $P(V)$ in the cohomology ring $H^2(\Grass(2, R_2))$ is independent of $V$, and is in fact a Schubert cycle. We would like to find, in a general space $V$, a subspace of dimension 2 of the most particular form to find a canonical expression of $V$; we will find a closed subvariety $\mathfrak Z$ of $G$ such that $\mathfrak Z\cap P(V)$ is always non-empty and in general finite. Since the cohomology class of $P(V)$ is independent of $V$, the intersection theory on $G$ affirms that if a closed subvariety $\mathfrak Z$ of $G$ of dimension 6 intersects $P(V)$ in a non-empty finite set for a given $V$ (that is, the intersection between $\mathfrak Z$ and $P(V)$ is proper and non-empty) then $\mathfrak Z$ intersects effectively all the $P(V)$; furthermore $\mathfrak Z$ will intersect a generic $P(V)$ in a non-empty finite set. In other words, if a 6-dimensional closed subvariety $\mathfrak Z$ of $G$ contains a finite number of subspaces of dimension 2 of a particular 3-dimensional vector space, then $\mathfrak Z$ contains at least a 2 dimensional vector subspace of every 3-dimensional vector space in $R_2$.

We prove now that if $W=\langle XY,X^2+YZ\rangle$, the closure $\mathfrak Z$ of the orbit of $W$ (which, according to the results in Table \ref{table3} of Section 3, is a sub-variety of $G$ of dimension 6), only contains as 2 dimensional subspaces of the 3 dimensional vector space $V_0=\langle XY,X^2+YZ,Z^2\rangle$, the space $W$ itself\footnote{Typo in original corrected.} and the space $\langle X^2+YZ,Z^2\rangle$.

Indeed, there is no vector subspace of dimension two of $V_0$ whose elements all have in common a linear factor, and there is no subspace isomorphic to $\langle X^2,Y^2\rangle$. The only elements of $\mathfrak Z$ that can belong to $P(V_0)$ belong then either to the orbit of $W$ or to that of $\langle X^2+YZ,Z^2\rangle$. In any case, the elements correspond to lines in the plane $\pi$ associated to $V_0$ that meet $\Gamma_\pi$ in a triple point. However, $\Gamma_\pi$ is here an irreducible cubic with a cusp, and the plane $\pi$ contains only two such lines, the tangent in the unique inflection point that corresponds to $W$ and the tangent at the cusp that corresponds to $\langle X^2+YZ,Z^2\rangle$. 

Algebraic arguments that are not reproduced here can replace these geometric arguments. 

\subsection{Algebraic classification}

We classify now completely the orbits of the subspace $V$ of dimension 3 of $R_2$, such that $\Gamma(V)$ is singular, and we classify, up to an action of finite group of automorphisms, the orbits of $V$ such that $\Gamma(V)$ is smooth. We give a representation for each orbit with an index that corresponds to its dimension and the symbol of the associated cubic $\Gamma(V)$ as in Table \ref{table1}. 

\begin{enumerate}
    \item[\underline{Case 1}] $\langle XY,X^2+YZ\rangle \,\subset V$. Then $$V=\langle XY,X^2+YZ,aX^2+bY^2+cZ^2+dXZ
    \rangle.$$ 
    \begin{enumerate}
        \item[\underline{Case 1.A}] $c\neq 0$ or $d\neq 0$. By replacing $Z$ by $Z-\alpha X$ with $c\alpha^2-d\alpha+a=0$, we eliminate $a$ (that is, we replace $V$ by $V^g$ with $g\in GL(3)$, $g:(X,Y,Z)\rightarrow (X,Y,Z-\alpha X)$. From now on we will omit such explanations. If now $b\neq 0$, a ``change in scale'' leads to a family: 
        $$V_c=\langle XY,X^2+YZ,Y^2+cZ^2+XZ\rangle\quad \text{if } d\neq 0;$$ and to
        $$V'=\langle XY,X^2+YZ,Y^2+Z^2\rangle\quad \text{if } d= 0.$$
        We study these two cases later.
        
        If $b=0$, $V=\langle XY,X^2+YZ,Z^2\rangle=\quad $ \begin{tikzpicture}[scale=0.2]
\draw  [color={rgb, 255:red, 0; green, 0; blue, 0 }  ][line width=0.75] [line join = round][line cap = round] (0,1.3) .. controls (0.5,1.3) and (0.95,1) .. (1.2,0.75) .. controls (1.4,0.6) and (1.5,0.4) .. (1.7,0.15) ;
\draw  [color={rgb, 255:red, 0; green, 0; blue, 0 }  ][line width=0.75] [line join = round][line cap = round] (0,1.3) .. controls (0.5,1.3) and (0.95,1.6) .. (1.2,1.85) .. controls (1.4,2) and (1.5,2.2) .. (1.7,2.55) ;
\draw (0,1.3) circle (0.4);
\end{tikzpicture}\quad or \par $V=\langle XY,X^2+YZ,XZ\rangle  =$ \begin{tikzpicture}[scale=0.2]
\draw (0,0) circle (1);
\draw (-1.5,0) -- (1.5,0);
\end{tikzpicture} \quad or\par
\par $V=\langle XY,X^2+YZ,Z^2+XZ\rangle$.

        In the latter case, the substitution $Z=(Z'+\frac{1}{2}X)$ and a ``change in scale'' leads to the form $V=\langle XY,X^2+YZ,Z^2+YZ\rangle$, and the substitution $X=(Y'-X')/2,\ Y=(Z'-Y'-X')/2,$ ${Z=(Y'-X')/2}$ provides the expression\par 
        \begin{center}
        $\langle XY,X^2+YZ,Y^2+XZ \rangle_8=$ \begin{tikzpicture}[scale=0.2]

\draw  [line width=0.75] [line join = round][line cap = round] (0,0) .. controls  (-0.5,0.4) and (-1.3,1.1) .. (-1.9,1.5) .. controls (-2.2,1.7) and (-2.8,1.5) .. (-2.85,1.1) .. controls (-2.9,0.9) and (-2.7,0.55) .. (-2.5,0.5) .. controls (-2,0.4) and (-1,0.8) .. (-0.6,0.9) .. controls (-0.5,0.94) and (0.103,1.0695) .. (0.1,1.15) ;
\end{tikzpicture}
 \end{center}
        
        We see that this is a limited subset of the group of cases previously considered.\par
        We now study the continuous family $V_c=\langle XY,X^2+YZ,Y^2+cZ^2+XZ \rangle$ and also $V_0=\langle XY,X^2+YZ,X^2+Y^2 \rangle$. In Section~4.1.2 Proposition \ref{4.2prop}ff we studied the family $\V_\lambda=\langle XY,X^2+YZ,(Y-Z)(Y- \lambda Z) \rangle$ with $\lambda\in {\sf k}$, and we have shown that if $\lambda\neq 0,1$, $\V_\lambda$ has a non singular associated cubic $\Gamma(V_\lambda)$. It is clear that $V_0\simeq \V_0$ and that $\V_{-1}\simeq V'.$ If $c \neq 0$, consider the following substitution in $V_c: Z=rZ'-\frac{1}{rc}X',\ X=r^{1/2}X',\ Y=Y'$. Then $$V_c=\langle XY,X^2+YZ,Y^2+\frac{1}{4c}YZ+cr^2Z^2 \rangle=\V_{\lambda} \text{ if } (\frac{r}{4c}-1)=cr^2=\lambda.$$
        Note that: 
        \begin{enumerate}[a)]
            \item Unless $c = 0$, (case already considered), we can always choose $r$ such that $\frac{4}{rc}-1 = cr^2$. Therefore, each $V_c$ belongs to the orbit of an element of the family $\V_\lambda$.
            \item In the map $\{V_c\}\rightarrow \{\V_\lambda\},\V_0$ is reached only by $V_0$ and \ $\V_1$ is reached only when $r = 8c$, and $cr^2=64c^3 = 1$ or $c = \sqrt[3]{1/64}$. In this case, $V_c\simeq \V_1= \langle XY, X^2 + YZ, Y^2 + 2YZ+ Z^2 \rangle =$ $ \langle XY, X^2 + YZ),(Y + Z)^2\rangle$=\, \begin{tikzpicture}[scale=0.2]
\draw  [line width=0.75] [line join = round][line cap = round] (0,0) .. controls  (-0.5,0.4) and (-1.3,1.1) .. (-1.9,1.5) .. controls (-2.2,1.7) and (-2.8,1.5) .. (-2.85,1.1) .. controls (-2.9,0.9) and (-2.7,0.55) .. (-2.5,0.5) .. controls (-2,0.4) and (-1,0.8) .. (-0.6,0.9) .. controls (-0.5,0.94) and (0.103,1.0695) .. (0.1,1.15) ;
\draw (-0.9,0.8) circle (0.4);
\end{tikzpicture}.
            
            We conclude that every vector space of the one parameter family $V_c$ has a non singular associated cubic $\Gamma(V_c)$ except for: $$\langle XY,X^2+YZ,Y^2+XZ\rangle_8=\text{\begin{tikzpicture}[scale=0.2]
\draw  [line width=0.75] [line join = round][line cap = round] (0,0) .. controls  (-0.5,0.4) and (-1.3,1.1) .. (-1.9,1.5) .. controls (-2.2,1.7) and (-2.8,1.5) .. (-2.85,1.1) .. controls (-2.9,0.9) and (-2.7,0.55) .. (-2.5,0.5) .. controls (-2,0.4) and (-1,0.8) .. (-0.6,0.9) .. controls (-0.5,0.94) and (0.103,1.0695) .. (0.1,1.15) ;
\end{tikzpicture}}$$ and $$ \langle XY,X^2+YZ,(X+Z)^2\rangle_8=\begin{tikzpicture}[scale=0.2]
\draw  [line width=0.75] [line join = round][line cap = round] (0,0) .. controls  (-0.5,0.4) and (-1.3,1.1) .. (-1.9,1.5) .. controls (-2.2,1.7) and (-2.8,1.5) .. (-2.85,1.1) .. controls (-2.9,0.9) and (-2.7,0.55) .. (-2.5,0.5) .. controls (-2,0.4) and (-1,0.8) .. (-0.6,0.9) .. controls (-0.5,0.94) and (0.103,1.0695) .. (0.1,1.15) ;
\draw (-0.9,0.8) circle (0.3);
\end{tikzpicture}
$$
        \end{enumerate}
        
        \item[\underline{Case 1.B}] If $d$ and $c$ are both non zero in case 1, then: 
        $$V=\langle XY,X^2+YZ,X^2+Y^2\rangle_6=\quad \begin{tikzpicture}[baseline={([yshift=-.8ex]current bounding box.center)},scale=0.2]
\draw (0,0) circle (0.5);
\draw (1,2) -- (1,-2);
\draw (2,0) circle(0.5);
\draw (-1,0.2) -- (3,0.2);
\draw (-1,-0.2) -- (3,-0.2);
\end{tikzpicture}
$$ or $$V=\langle XY,YZ,X^2\rangle\simeq \langle XY,XZ,Z^2\rangle_5=\quad \begin{tikzpicture} [baseline={([yshift=-.8ex]current bounding box.center)},scale=0.2]
\draw (1,2) -- (1,-2);
\draw (2,0) circle(0.5);
\draw (-1,0.2) -- (3,0.2);
\draw (-1,-0.2) -- (3,-0.2);
\end{tikzpicture}
$$ or $$V=\langle XY,X^2+YZ,Y^2\rangle_4=\quad\begin{tikzpicture}[scale=0.3]
\draw (0,0) -- (2,0);
\draw (0,-0.2) -- (2,-0.2);
\draw (0,0.2) -- (2,0.2);
\draw (1,0) circle (0.5);
\end{tikzpicture}
$$
    \end{enumerate}
    \item[\underline{Case 2}] $\langle Y^2,X^2+YZ\rangle\subset V$:
    
    We have then $V=\langle Y^2,X^2+YZ,aZ^2+bXY+cYZ+dZX\rangle.$
    
    Considering the transformations: $X'=X+\lambda Y,\ Z'=Z-2\lambda X+\mu Y$, composed with the ''changes in scale'' we get the following five cases: 
    
    \begin{center}
    \begin{tabular}{ll}
        $\langle Y^2,X^2+YZ,Z(Z+X)\rangle$& which is a form of \,\,\begin{tikzpicture}[scale=0.2]
\draw  [line width=0.75] [line join = round][line cap = round] (0,0) .. controls  (-0.5,0.4) and (-1.3,1.1) .. (-1.9,1.5) .. controls (-2.2,1.7) and (-2.8,1.5) .. (-2.85,1.1) .. controls (-2.9,0.9) and (-2.7,0.55) .. (-2.5,0.5) .. controls (-2,0.4) and (-1,0.8) .. (-0.6,0.9) .. controls (-0.5,0.94) and (0.103,1.0695) .. (0.1,1.15) ;
\draw (-0.9,0.8) circle (0.4);
\end{tikzpicture}

\\
        $\langle Y^2, X^2+YZ,XZ\rangle$& which represents\quad  \begin{tikzpicture}[scale=0.2]
\draw  [color={rgb, 255:red, 0; green, 0; blue, 0 }  ][line width=0.75] [line join = round][line cap = round] (0,1.3) .. controls (0.5,1.3) and (0.95,1) .. (1.2,0.75) .. controls (1.4,0.6) and (1.5,0.4) .. (1.7,0.15) ;
\draw  [color={rgb, 255:red, 0; green, 0; blue, 0 }  ][line width=0.75] [line join = round][line cap = round] (0,1.3) .. controls (0.5,1.3) and (0.95,1.6) .. (1.2,1.85) .. controls (1.4,2) and (1.5,2.2) .. (1.7,2.55) ;
\draw (0,1.3) circle (0.4);
\end{tikzpicture}
\\
         $\langle Y^2, X^2+YZ,Z^2\rangle $& which represents\quad 
\begin{tikzpicture}[scale=0.24]
\draw (0,0) circle (1);
\draw (-1.5,0) -- (1.5,0);
\draw (-1,0) circle (0.3);
\draw (1,0) circle (0.3);
\end{tikzpicture}
\\ 
         $\langle Y^2, X^2+YZ,YZ\rangle$& which is a form of\quad
\begin{tikzpicture}[baseline={([yshift=-.8ex]current bounding box.center)},scale=0.2]
\draw (0,0) circle (0.5);
\draw (0,2) -- (0,-2);
\draw (2,0) circle(0.5);
\draw (-1,0) -- (3,0);
\end{tikzpicture}
\\
         $\langle Y^2, X^2+YZ,XY\rangle$& which represents\quad\begin{tikzpicture}[scale=0.25]
\draw (0,0) -- (2,0);
\draw (0,-0.2) -- (2,-0.2);
\draw (0,0.2) -- (2,0.2);
\draw (1,0) circle (0.5);
\end{tikzpicture}

    \end{tabular}
    \end{center}
    
    \item[\underline{Case 3}] $\langle XY, XZ\rangle\subset V$, so that $V = \langle XY, XZ, aX^2 + bY^2 + cZ^2 + dYZ\rangle$.
    \begin{enumerate}
        \item[\underline{Case 3.A}] $d\neq 0$. Replacing $Y$ by $Y+\alpha Z'$ with : $b\alpha^2+d\alpha + c = 0$, we eliminate $c$. If $d\neq O$, we replace $Z$ by $Z'-(\frac{b}{d})Y$ to eliminate $b$. After multiplication by a scalar, we get: 
        
        \hspace{1cm} $V=\langle XY,XZ,X^2+YZ\rangle$ already considered in 1.A, or
        
        \hspace{1cm}  $V=\langle XY,XZ,YZ\rangle_6=$\,\,\begin{tikzpicture}[scale=0.15]
\draw (-0.4,0) -- (3.4,0);
\draw (-0.2,-0.4) -- (1.7,3.4);
\draw (3.2,-0.4) -- (1.3,3.4);
\end{tikzpicture},\quad or 
        
        \hspace{1cm} $V=\langle XY,XZ,X^2\rangle_2=$\quad \begin{tikzpicture}[baseline={([yshift=-.8ex]current bounding box.center)},scale=0.3]
\draw (0,0) circle (0.5);
\draw[color=gray] (-0.5,-1) -- (0.5,1);
\draw[color=gray] (-0.3,-1.1) -- (0.7,0.9);
\draw[color=gray] (-0.1,-1.2) -- (0.9,0.8);
\draw[color=gray] (0.1,-1.3) -- (1.1,0.7);
\draw[color=gray] (-0.7,-0.9) -- (0.3,1.1);
\draw[color=gray] (-0.9,-0.8) -- (0.1,1.2);
\draw[color=gray] (-1.1,-0.7) -- (-0.1,1.3);
\end{tikzpicture}
        
        \item[\underline{Case 3.B}] $d=0$. The substitution $Y \rightarrow Y + dZ$ or $Z \rightarrow Z + dY$ will make $d$ non-zero, unless $b = c = d = O$, or $V = \langle XY, XZ, X^2\rangle$ which has just been considered or $V = \langle XY, XZ, Z^2\rangle$ considered in 1.B. Thus the discussion of 3.A. applies unless $V \simeq \langle XY, XZ, Y^2\rangle$ (already considered) or that: $V\simeq \langle XY, XZ, X^2 + Y^2\rangle$ in which case the substitution used in 3.A. returns the initial vector space. This last space $V = \langle XY, YZ, X^2 + Y^2\rangle$ is isomorphic to $\langle XY, X^2+ YZ, X^2 + Y^2\rangle$ of case 1.B.
    \end{enumerate}
    \item[\underline{Case 4}] $\langle XY,X^2\rangle\,\subset V$, so that $V=\langle XY,X^2,aY^2+bZ^2+cXZ+dYZ\rangle$.
    \begin{enumerate}
        \item[\underline{Case 4.A}] $b\neq 0$. We eliminate $a$ and $c$ by the substitution $$Z\rightarrow Z'+\sqrt{a/b}Y+(\frac{c}{2})X.$$ A change of scale gives: 
        $$V\simeq \langle XY,X^2,Z^2+YZ\rangle_6=\, \begin{tikzpicture}[scale=0.2]
\draw (0,0) circle (1);
\draw (-1.5,-1) -- (1.5,-1);
\draw (0,-1) circle(0.4);
\end{tikzpicture}
\quad \text{or}\quad V=\langle XY,X^2,Z^2\rangle_5=\,\begin{tikzpicture}[baseline={([yshift=-.8ex]current bounding box.center)},scale=0.2]
\draw (0,0) circle (0.5);
\draw (0,2) -- (0,-2);
\draw (2,0) circle(0.5);
\draw (-1,0) -- (3,0);
\end{tikzpicture}
.$$
        \item[\underline{Case 4.B}] $b=0$. If $d\neq 0$, we eliminate $c$ by $Y\rightarrow Y'-(\frac{c}{d})X$ and a change in scale gives: $V=\langle X^2,XY,Y^2+YZ\rangle =\langle X^2,XY,XZ\rangle$ or ${V=\langle X^2,XY,YZ\rangle}$ of case 1.B. 
        
        If $d=b=0$, $V=\langle XY,X^2,Y^2+XZ\rangle=\langle XY,X^2+YZ,Y^2\rangle$ of case 1.B. or $V=\langle XY,XZ,X^2\rangle$ of case 3.A. or \par
        $V=\langle XY,X^2,Y^2\rangle_2=$\quad\begin{tikzpicture}[baseline={([yshift=-.8ex]current bounding box.center)},scale=0.25]
\draw (-0.8,1) .. controls (-0.3,0) and (0.7,0) .. (1.2,1);
\draw[color=gray] (0.5,-0.75) -- (1.5,1.25);
\draw[color=gray] (0.3,-0.7) -- (1.3,1.3);
\draw[color=gray] (0.1,-0.65) -- (1.1,1.35);
\draw[color=gray] (-0.1,-0.6) -- (0.9,1.4);
\draw[color=gray] (-0.3,-0.55) -- (0.7,1.45);
\draw[color=gray] (-0.5,-0.5) -- (0.5,1.5);
\draw[color=gray] (-0.7,-0.45) -- (0.3,1.55);
\draw[color=gray] (-0.9,-0.4) -- (0.1,1.6);
\draw[color=gray] (-1.1,-0.35) -- (-0.1,1.65);
\draw[color=gray] (-1.3,-0.3) -- (-0.3,1.7);
\draw[color=gray] (-1.5,-0.25) -- (-0.5,1.76);
\end{tikzpicture}
.
    \end{enumerate}
    
    \item[\underline{Case 5}] $\langle X^2,Y^2\rangle\subset V$. Then, $V=\langle X^2,Y^2,aXY,bYZ+cXZ+dZ^2\rangle$.
    \begin{enumerate}
        \item[\underline{Case 5.A}] If $d\neq 0$, replace $Z$ by $Z'+ \frac{a}{2}X+\frac{b}{2}Y$; after a change in scale $$V \simeq  \langle X^2, Y^2, XY + Z^2\rangle_7 =\,\begin{tikzpicture}[scale=0.23]
\draw (0,0) circle (1);
\draw (-1.5,0) -- (1.5,0);
\draw (-1,0) circle (0.35);
\draw (1,0) circle (0.35);
\end{tikzpicture}\quad \text{ or } V = \langle X^2, Y^2, Z^2\rangle_6=\,\,\begin{tikzpicture}[scale=0.2]
\draw (-0.4,0) -- (3.4,0);
\draw (-0.2,-0.4) -- (1.7,3.4);
\draw (3.2,-0.4) -- (1.3,3.4);
\draw (0,0) circle (0.35);
\draw (3,0) circle (0.35);
\draw (1.5,3) circle (0.35);
\end{tikzpicture}
.$$
        
        \item[\underline{Case 5.B}] If $d=0$, $V\simeq \langle X^2, Y^2, XZ\rangle$ (see 3.A) or ${\langle X^2, Y^2, Z^2\rangle}$ (3.B) or ${V \simeq \langle X^2, Y^2, Z(X+Y)\rangle}$. This latter space, after substitution ${X\rightarrow X'+Y'}$, ${Y\rightarrow X'-Y'}$ becomes $V=\langle X'^2+Y'^2,X'Y',X'Z' \rangle$ which, as we showed in 3.B, is isomorphic to ${\langle XY, X^2+YZ,X^2+Y^2\rangle}$ of case 1.B. 
        
        This completes the verification of the 14 special orbits of Table \ref{table1} by the algebraic method. 
    \end{enumerate}
\end{enumerate}

\section{The closures of orbits: specializations}

\underline{Introduction, definitions.}

\begin{definition}
Let $U$ and $V$ be two orbits of planes of $C$ under the action of $Pgl(3)$ operating in $P$. We say that $U$ specializes in $V$ or that $V$ is a specialization of $U$ if $V\subset \overline{U}$.
\end{definition}

\begin{rmk} \begin{enumerate}[a)]
    \item It is the same to say that there exists a family $F_T\subset \Grass(3,R_2)\times T$ of elements $F_t, t\in T$ of $\Grass(3,R_2)$ over a neighborhood $T$ of $0$ in the affine line, whose generic element $F_t, t\in T\backslash 0$ belongs to $U$, and whose element $F_0$ corresponding to 0 belongs to $V$. 
    
    \item The relation ``$V$ is a specialization of $U$'' is an order relation $V<U$ on the set of orbits (order of genericity).
\end{enumerate}
\end{rmk}

We will describe this order in Sections 6.4-6.6. Beforehand, we study in Section 6.1 a duality relation on the set of orbits; in Section 6.2, the dimensions of orbits; and in Section 6.3, the length of the projective schemes naturally associated to elements of $\Grass(3,R_2)$. 

\subsection{Duality}\label{dualitysec}

We consider on $R_2$ the scalar product which makes the basis 
\begin{equation*}
(\frac{1}{\sqrt{2}}x^2, \frac{1}{\sqrt{2}}y^2, \frac{1}{\sqrt{2}}z^2, xy,yz,xz)
\end{equation*}
orthonormal. Let $E$ and $E'$ be subspaces of $R_2$. We can easily verify that if $E$ and $E'$ have the same orbit under the action of $\GL_3(k)$, so have their respective orthogonals as well. The passage to the orthogonal therefore defines an involution (obviously regular) of
$\Grass (3,R_2)$ on itself which passes to the quotient modulo the action of $\GL_3(k)$ to give an involution of the set of orbits.

\begin{definition}
The involution thus obtained over the set of orbits of planes of plane projective conics is called duality. If $X$ is such an orbit, we denote by $X^\perp$ the dual orbit.
\end{definition}

We have the following obvious propositions: 

\begin{proposition}
If $X'$ is a specialization of $X$, then $X'^\perp$ is a specialization of $X^\perp$.
\end{proposition}

\begin{proposition}
$\dim X^\perp=\dim X$.
\end{proposition}
Recall from Table \ref{table1} that\quad \begin{tikzpicture}[baseline={([yshift=-0.6ex]current bounding box.center)},scale=0.4]
\draw (0,0) circle (0.3);
\draw    (0.7,0.8) .. controls  (0,-0.3) and (1.2,0.2) .. (0.4,-0.8);
\draw (0.7,-0.4) node [anchor=west]{$j(\lambda)$};
\end{tikzpicture}designates the orbit of the plane generated by the partial derivatives of $X^3+Y^3+Z^3+3\lambda XYZ$.
\begin{proposition}\label{dualprop}
We have the following list of couples of dual orbits (Table~\ref{dualorbitfig}).
\end{proposition}
\begin{transl}\label{dualnote} J. Emsalem and A. Iarrobino omitted to develop the notion of projective dual of the space of conics, in the dual plane $P^\vee$ of $P$. It is a method of treating intrinsically the equivalence between isomorphism class of a net and that of its orthogonal, while, on the other hand, the scalar product they have used does not respect the action of $Pgl(3)$.
Further, the duality between the two spaces of conics makes obvious the equivalence between \begin{enumerate}[1.]
\item Classification of conics, and the classification of hyperplanes of conics;
\item Classification of pencils of conics, and the classification of webs of conics - linear projective subvarieties of dimension 3 of the space of conics.
\end{enumerate}
It becomes clear from this consideration that the classification of all linear projective subvarieties of the space of conics can be deduced simply from the results of the paper.
\end{transl}
\begin{table}
\begin{center}
\begin{tikzpicture}[scale=0.6]
\draw (0,0) circle (0.3);
\draw    (0.7,0.8) .. controls  (0,-0.3) and (1.2,0.2) .. (0.4,-0.8);
\draw (0.7,-0.4) node [anchor=west]{$j(\lambda)$};
\draw[<-] (2.5,0) -- (3,0);
\draw[->] (3,0) -- (3.5,0);
\draw (4.5,0) circle (0.3);
\draw    (5.2,0.8) .. controls  (4.5,-0.3) and (5.7,0.2) .. (4.9,-0.8);
\draw (5.2,-0.4) node [anchor=west]{$j(-\frac{2}{\lambda})$};
\end{tikzpicture}
\end{center}

\begin{center}

\tikzset{every picture/.style={line width=0.75pt}} 

\begin{tikzpicture}[x=0.75pt,y=0.75pt,yscale=-1,xscale=1]

\draw  [color={rgb, 255:red, 0; green, 0; blue, 0 }  ][line width=0.75] [line join = round][line cap = round] (231.5,34.32) .. controls (221.2,42.56) and (205.71,57.09) .. (194.5,63.32) .. controls (187.02,67.47) and (175.94,64.96) .. (174.5,56.32) .. controls (173.74,51.73) and (177.91,45.15) .. (182.5,44.32) .. controls (190.34,42.89) and (211.25,49.57) .. (219.5,52.32) .. controls (222.05,53.17) and (233.5,55.71) .. (233.5,57.32) ;
\draw    (294,50) -- (324,50) ;
\draw [shift={(326,50)}, rotate = 180] [color={rgb, 255:red, 0; green, 0; blue, 0 }  ][line width=0.75]    (10.93,-3.29) .. controls (6.95,-1.4) and (3.31,-0.3) .. (0,0) .. controls (3.31,0.3) and (6.95,1.4) .. (10.93,3.29)   ;
\draw    (294,50) -- (272,50) ;
\draw [shift={(270,50)}, rotate = 360] [color={rgb, 255:red, 0; green, 0; blue, 0 }  ][line width=0.75]    (10.93,-3.29) .. controls (6.95,-1.4) and (3.31,-0.3) .. (0,0) .. controls (3.31,0.3) and (6.95,1.4) .. (10.93,3.29)   ;
\draw  [color={rgb, 255:red, 0; green, 0; blue, 0 }  ][line width=0.75] [line join = round][line cap = round] (421.5,36.32) .. controls (411.2,44.56) and (395.71,59.09) .. (384.5,65.32) .. controls (377.02,69.47) and (365.94,66.96) .. (364.5,58.32) .. controls (363.74,53.73) and (367.91,47.15) .. (372.5,46.32) .. controls (380.34,44.89) and (401.25,51.57) .. (409.5,54.32) .. controls (412.05,55.17) and (423.5,57.71) .. (423.5,59.32) ;
\draw   (397.23,52.38) .. controls (397.23,48.86) and (400.09,46) .. (403.62,46) .. controls (407.14,46) and (410,48.86) .. (410,52.38) .. controls (410,55.91) and (407.14,58.77) .. (403.62,58.77) .. controls (400.09,58.77) and (397.23,55.91) .. (397.23,52.38) -- cycle ;
\draw   (187.67,109.5) .. controls (187.67,100.94) and (194.6,94) .. (203.15,94) .. controls (211.7,94) and (218.64,100.94) .. (218.64,109.5) .. controls (218.64,118.06) and (211.7,125) .. (203.15,125) .. controls (194.6,125) and (187.67,118.06) .. (187.67,109.5) -- cycle ;
\draw    (166,111.51) -- (240,110.82) ;
\draw    (293,110) -- (323,110) ;
\draw [shift={(325,110)}, rotate = 180] [color={rgb, 255:red, 0; green, 0; blue, 0 }  ][line width=0.75]    (10.93,-3.29) .. controls (6.95,-1.4) and (3.31,-0.3) .. (0,0) .. controls (3.31,0.3) and (6.95,1.4) .. (10.93,3.29)   ;
\draw    (293,110) -- (271,110) ;
\draw [shift={(269,110)}, rotate = 360] [color={rgb, 255:red, 0; green, 0; blue, 0 }  ][line width=0.75]    (10.93,-3.29) .. controls (6.95,-1.4) and (3.31,-0.3) .. (0,0) .. controls (3.31,0.3) and (6.95,1.4) .. (10.93,3.29)   ;
\draw    (358.52,110) -- (428.35,110) ;
\draw   (404.98,109.24) .. controls (405.36,104.23) and (409.73,100.48) .. (414.74,100.87) .. controls (419.75,101.25) and (423.5,105.62) .. (423.11,110.63) .. controls (422.73,115.63) and (418.36,119.38) .. (413.35,119) .. controls (408.34,118.62) and (404.59,114.25) .. (404.98,109.24) -- cycle ;
\draw   (364.98,109.24) .. controls (365.36,104.23) and (369.73,100.48) .. (374.74,100.87) .. controls (379.75,101.25) and (383.5,105.62) .. (383.11,110.63) .. controls (382.73,115.63) and (378.36,119.38) .. (373.35,119) .. controls (368.34,118.62) and (364.59,114.25) .. (364.98,109.24) -- cycle ;
\draw  [color={rgb, 255:red, 0; green, 0; blue, 0 }  ][line width=0.75] [line join = round][line cap = round] (273.42,156.6) .. controls (281.93,156.6) and (290.4,150.02) .. (294.95,145.59) .. controls (297.46,143.14) and (303.7,137.45) .. (303.7,135.62) ;
\draw   (267.82,155.25) .. controls (268.1,151.02) and (271.27,147.84) .. (274.91,148.17) .. controls (278.55,148.49) and (281.28,152.19) .. (281,156.43) .. controls (280.72,160.67) and (277.54,163.84) .. (273.9,163.51) .. controls (270.27,163.19) and (267.54,159.49) .. (267.82,155.25) -- cycle ;
\draw  [color={rgb, 255:red, 0; green, 0; blue, 0 }  ][line width=0.75] [line join = round][line cap = round] (304,175) .. controls (299.35,167.45) and (290.4,162.81) .. (285,160.7) .. controls (282.02,159.54) and (274.88,156.49) .. (273.67,157.29) ;
\draw    (317,166) .. controls (342.61,172.9) and (334.26,124.49) .. (314.89,151.7) ;
\draw [shift={(314,153)}, rotate = 303.69] [color={rgb, 255:red, 0; green, 0; blue, 0 }  ][line width=0.75]    (10.93,-3.29) .. controls (6.95,-1.4) and (3.31,-0.3) .. (0,0) .. controls (3.31,0.3) and (6.95,1.4) .. (10.93,3.29)   ;
\draw    (199.05,177) -- (234.55,231) ;
\draw    (175.37,224.06) -- (241,224.06) ;
\draw    (178.59,231) -- (209.25,177) ;
\draw    (294,211) -- (324,211) ;
\draw [shift={(326,211)}, rotate = 180] [color={rgb, 255:red, 0; green, 0; blue, 0 }  ][line width=0.75]    (10.93,-3.29) .. controls (6.95,-1.4) and (3.31,-0.3) .. (0,0) .. controls (3.31,0.3) and (6.95,1.4) .. (10.93,3.29)   ;
\draw    (294,211) -- (272,211) ;
\draw [shift={(270,211)}, rotate = 360] [color={rgb, 255:red, 0; green, 0; blue, 0 }  ][line width=0.75]    (10.93,-3.29) .. controls (6.95,-1.4) and (3.31,-0.3) .. (0,0) .. controls (3.31,0.3) and (6.95,1.4) .. (10.93,3.29)   ;
\draw    (388.43,170) -- (425.06,229) ;
\draw    (364,221.41) -- (431.72,221.41) ;
\draw    (367.33,229) -- (398.96,170) ;
\draw   (389.18,178.95) .. controls (389.18,175.98) and (391.56,173.57) .. (394.49,173.57) .. controls (397.43,173.57) and (399.81,175.98) .. (399.81,178.95) .. controls (399.81,181.92) and (397.43,184.33) .. (394.49,184.33) .. controls (391.56,184.33) and (389.18,181.92) .. (389.18,178.95) -- cycle ;
\draw   (415.82,221.09) .. controls (415.82,218.12) and (418.2,215.71) .. (421.13,215.71) .. controls (424.07,215.71) and (426.44,218.12) .. (426.44,221.09) .. controls (426.44,224.06) and (424.07,226.47) .. (421.13,226.47) .. controls (418.2,226.47) and (415.82,224.06) .. (415.82,221.09) -- cycle ;
\draw   (366.71,221.09) .. controls (366.71,218.12) and (369.08,215.71) .. (372.02,215.71) .. controls (374.95,215.71) and (377.33,218.12) .. (377.33,221.09) .. controls (377.33,224.06) and (374.95,226.47) .. (372.02,226.47) .. controls (369.08,226.47) and (366.71,224.06) .. (366.71,221.09) -- cycle ;
\draw    (207.65,270) -- (207.65,329) ;
\draw    (171.62,291) -- (245,291.06) ;
\draw   (176.38,294.64) .. controls (176.71,290.42) and (180.48,287.26) .. (184.79,287.59) .. controls (189.11,287.91) and (192.34,291.59) .. (192.01,295.81) .. controls (191.67,300.03) and (187.91,303.19) .. (183.59,302.87) .. controls (179.28,302.55) and (176.05,298.87) .. (176.38,294.64) -- cycle ;
\draw   (222.76,294.64) .. controls (223.09,290.42) and (226.86,287.26) .. (231.18,287.59) .. controls (235.49,287.91) and (238.72,291.59) .. (238.39,295.81) .. controls (238.06,300.03) and (234.29,303.19) .. (229.98,302.87) .. controls (225.66,302.55) and (222.43,298.87) .. (222.76,294.64) -- cycle ;
\draw    (171.62,300.27) -- (245,300.33) ;
\draw    (294,493) -- (324,493) ;
\draw [shift={(326,493)}, rotate = 180] [color={rgb, 255:red, 0; green, 0; blue, 0 }  ][line width=0.75]    (10.93,-3.29) .. controls (6.95,-1.4) and (3.31,-0.3) .. (0,0) .. controls (3.31,0.3) and (6.95,1.4) .. (10.93,3.29)   ;
\draw    (294,493) -- (272,493) ;
\draw [shift={(270,493)}, rotate = 360] [color={rgb, 255:red, 0; green, 0; blue, 0 }  ][line width=0.75]    (10.93,-3.29) .. controls (6.95,-1.4) and (3.31,-0.3) .. (0,0) .. controls (3.31,0.3) and (6.95,1.4) .. (10.93,3.29)   ;
\draw    (296,373) -- (326,373) ;
\draw [shift={(328,373)}, rotate = 180] [color={rgb, 255:red, 0; green, 0; blue, 0 }  ][line width=0.75]    (10.93,-3.29) .. controls (6.95,-1.4) and (3.31,-0.3) .. (0,0) .. controls (3.31,0.3) and (6.95,1.4) .. (10.93,3.29)   ;
\draw    (296,373) -- (274,373) ;
\draw [shift={(272,373)}, rotate = 360] [color={rgb, 255:red, 0; green, 0; blue, 0 }  ][line width=0.75]    (10.93,-3.29) .. controls (6.95,-1.4) and (3.31,-0.3) .. (0,0) .. controls (3.31,0.3) and (6.95,1.4) .. (10.93,3.29)   ;
\draw    (294,297) -- (324,297) ;
\draw [shift={(326,297)}, rotate = 180] [color={rgb, 255:red, 0; green, 0; blue, 0 }  ][line width=0.75]    (10.93,-3.29) .. controls (6.95,-1.4) and (3.31,-0.3) .. (0,0) .. controls (3.31,0.3) and (6.95,1.4) .. (10.93,3.29)   ;
\draw    (294,297) -- (272,297) ;
\draw [shift={(270,297)}, rotate = 360] [color={rgb, 255:red, 0; green, 0; blue, 0 }  ][line width=0.75]    (10.93,-3.29) .. controls (6.95,-1.4) and (3.31,-0.3) .. (0,0) .. controls (3.31,0.3) and (6.95,1.4) .. (10.93,3.29)   ;
\draw   (388.51,312.06) .. controls (388.51,309.34) and (390.75,307.13) .. (393.51,307.13) .. controls (396.27,307.13) and (398.51,309.34) .. (398.51,312.06) .. controls (398.51,314.79) and (396.27,317) .. (393.51,317) .. controls (390.75,317) and (388.51,314.79) .. (388.51,312.06) -- cycle ;
\draw   (374.2,291.95) .. controls (374.2,281.27) and (382.96,272.62) .. (393.77,272.62) .. controls (404.58,272.62) and (413.34,281.27) .. (413.34,291.95) .. controls (413.34,302.63) and (404.58,311.29) .. (393.77,311.29) .. controls (382.96,311.29) and (374.2,302.63) .. (374.2,291.95) -- cycle ;
\draw    (346.05,312.45) -- (440,311.68) ;
\draw    (198.13,347) -- (198.13,408) ;
\draw    (179.5,368.71) -- (240.49,368.77) ;
\draw    (178.52,377.5) -- (241,377.5) ;
\draw   (219.19,372.48) .. controls (219.54,368.12) and (223.45,364.85) .. (227.93,365.18) .. controls (232.41,365.52) and (235.76,369.33) .. (235.42,373.69) .. controls (235.08,378.05) and (231.16,381.32) .. (226.68,380.99) .. controls (222.2,380.65) and (218.85,376.84) .. (219.19,372.48) -- cycle ;
\draw    (378.13,347) -- (378.13,406) ;
\draw    (358.52,376.5) -- (421,376.5) ;
\draw   (400.09,375.86) .. controls (400.43,371.64) and (404.34,368.48) .. (408.82,368.8) .. controls (413.3,369.13) and (416.66,372.81) .. (416.31,377.03) .. controls (415.97,381.25) and (412.06,384.41) .. (407.58,384.09) .. controls (403.1,383.76) and (399.74,380.08) .. (400.09,375.86) -- cycle ;
\draw   (370.56,375.86) .. controls (370.9,371.64) and (374.81,368.48) .. (379.3,368.8) .. controls (383.78,369.13) and (387.13,372.81) .. (386.79,377.03) .. controls (386.44,381.25) and (382.53,384.41) .. (378.05,384.09) .. controls (373.57,383.76) and (370.22,380.08) .. (370.56,375.86) -- cycle ;
\draw    (265.62,418.22) -- (333.78,418.29) ;
\draw    (264.52,428.31) -- (334.35,428.31) ;
\draw   (287.98,422.54) .. controls (288.36,417.54) and (292.73,413.79) .. (297.74,414.17) .. controls (302.75,414.55) and (306.5,418.93) .. (306.11,423.93) .. controls (305.73,428.94) and (301.36,432.69) .. (296.35,432.31) .. controls (291.34,431.92) and (287.59,427.55) .. (287.98,422.54) -- cycle ;
\draw    (264.62,423.22) -- (332.78,423.29) ;
\draw    (348,434) .. controls (373.61,440.9) and (365.26,392.49) .. (345.89,419.7) ;
\draw [shift={(345,421)}, rotate = 303.69] [color={rgb, 255:red, 0; green, 0; blue, 0 }  ][line width=0.75]    (10.93,-3.29) .. controls (6.95,-1.4) and (3.31,-0.3) .. (0,0) .. controls (3.31,0.3) and (6.95,1.4) .. (10.93,3.29)   ;
\draw   (201.98,490.8) .. controls (202.36,485.79) and (206.73,482.04) .. (211.74,482.42) .. controls (216.75,482.81) and (220.5,487.18) .. (220.11,492.19) .. controls (219.73,497.19) and (215.36,500.94) .. (210.35,500.56) .. controls (205.34,500.18) and (201.59,495.81) .. (201.98,490.8) -- cycle ;
\draw [color={rgb, 255:red, 128; green, 128; blue, 128 }  ,draw opacity=1 ]   (212.67,511) -- (235.56,473.37) ;
\draw [color={rgb, 255:red, 128; green, 128; blue, 128 }  ,draw opacity=1 ]   (187.11,508.94) -- (210,471.31) ;
\draw [color={rgb, 255:red, 128; green, 128; blue, 128 }  ,draw opacity=1 ]   (193.53,509.61) -- (216.42,471.98) ;
\draw [color={rgb, 255:red, 128; green, 128; blue, 128 }  ,draw opacity=1 ]   (199.6,510.31) -- (222.49,472.68) ;
\draw [color={rgb, 255:red, 128; green, 128; blue, 128 }  ,draw opacity=1 ]   (205.67,512) -- (228.56,474.37) ;
\draw  [draw opacity=0] (415.58,480.83) .. controls (415.79,481.98) and (415.93,483.17) .. (415.98,484.38) .. controls (416.51,496.98) and (407.82,507.58) .. (396.58,508.05) .. controls (385.33,508.53) and (375.78,498.7) .. (375.25,486.1) .. controls (375.15,483.64) and (375.39,481.26) .. (375.94,479.01) -- (395.61,485.24) -- cycle ; \draw   (415.58,480.83) .. controls (415.79,481.98) and (415.93,483.17) .. (415.98,484.38) .. controls (416.51,496.98) and (407.82,507.58) .. (396.58,508.05) .. controls (385.33,508.53) and (375.78,498.7) .. (375.25,486.1) .. controls (375.15,483.64) and (375.39,481.26) .. (375.94,479.01) ;
\draw [color={rgb, 255:red, 128; green, 128; blue, 128 }  ,draw opacity=1 ]   (378.87,510) -- (401.76,472.37) ;
\draw [color={rgb, 255:red, 128; green, 128; blue, 128 }  ,draw opacity=1 ]   (353.31,507.94) -- (376.2,470.31) ;
\draw [color={rgb, 255:red, 128; green, 128; blue, 128 }  ,draw opacity=1 ]   (359.73,508.61) -- (382.62,470.98) ;
\draw [color={rgb, 255:red, 128; green, 128; blue, 128 }  ,draw opacity=1 ]   (365.8,509.31) -- (388.69,471.68) ;
\draw [color={rgb, 255:red, 128; green, 128; blue, 128 }  ,draw opacity=1 ]   (371.87,511) -- (394.76,473.37) ;
\draw [color={rgb, 255:red, 128; green, 128; blue, 128 }  ,draw opacity=1 ]   (410.87,513) -- (433.76,475.37) ;
\draw [color={rgb, 255:red, 128; green, 128; blue, 128 }  ,draw opacity=1 ]   (385.31,510.94) -- (408.2,473.31) ;
\draw [color={rgb, 255:red, 128; green, 128; blue, 128 }  ,draw opacity=1 ]   (391.73,511.61) -- (414.62,473.98) ;
\draw [color={rgb, 255:red, 128; green, 128; blue, 128 }  ,draw opacity=1 ]   (397.8,512.31) -- (420.69,474.68) ;
\draw [color={rgb, 255:red, 128; green, 128; blue, 128 }  ,draw opacity=1 ]   (403.87,514) -- (426.76,476.37) ;

\end{tikzpicture}

\end{center}
\caption{Dual orbits of nets of conics}\label{dualorbitfig}
\end{table}
\subsection{Dimension of orbits}\label{dimsec}

We will calculate the dimension of most orbits by finding the isotropy group of a 3-dimensional subspace of $R_2$ under the action of $Pgl(3)$ or $\SL_3$. $\SL_3$ is a finite covering space of $Pgl(3)$ so that $$\dim Pgl(3)(V) = \dim \SL_3(V) = \dim \SL_3 -\dim \Stab V= 8-\dim \Stab V,$$ where $\Stab V$ is the isotropy group of V. We start by listing these isotropy groups, orbit by orbit. However, considering the duality, we do not have to perform the calculation for an orbit whose dual already has its known dimension.

\begin{enumerate}[1)]
    \item \underline{Dimension of} \begin{tikzpicture}[baseline={([yshift=-.5ex]current bounding box.center)},scale=0.4]
\draw (0,0) circle (0.3);
\draw    (0.7,0.8) .. controls  (0,-0.3) and (1.2,0.2) .. (0.4,-0.8);
\draw (0.7,-0.4) node [anchor=west]{$\lambda$};
\end{tikzpicture} By Section 4.1, the orbit  \begin{tikzpicture}[baseline={([yshift=-.5ex]current bounding box.center)},scale=0.4]
\draw (0,0) circle (0.3);
\draw    (0.7,0.8) .. controls  (0,-0.3) and (1.2,0.2) .. (0.4,-0.8);
\draw (0.7,-0.4) node [anchor=west]{$\lambda$};
\end{tikzpicture} is in a regular one-to-one correspondence with the orbit of a smooth cubic which is not isomorphic to $(X^3+Y^3+Z^3)$ in the space of cubics under the action of $Pgl(3)$ operating on $P$. This last dimension is 8. Then \begin{tikzpicture}[baseline={([yshift=-.5ex]current bounding box.center)},scale=0.4]
\draw (0,0) circle (0.3);
\draw    (0.7,0.8) .. controls  (0,-0.3) and (1.2,0.2) .. (0.4,-0.8);
\draw (0.7,-0.4) node [anchor=west]{$\lambda$};
\end{tikzpicture} is also of dimension 8. 
    
    \item \underline{Dimension of} \begin{tikzpicture}[baseline={([yshift=-.8ex]current bounding box.center)},scale=0.2]

\draw  [line width=0.75] [line join = round][line cap = round] (0,0) .. controls  (-0.5,0.4) and (-1.3,1.1) .. (-1.9,1.5) .. controls (-2.2,1.7) and (-2.8,1.5) .. (-2.85,1.1) .. controls (-2.9,0.9) and (-2.7,0.55) .. (-2.5,0.5) .. controls (-2,0.4) and (-1,0.8) .. (-0.6,0.9) .. controls (-0.5,0.94) and (0.103,1.0695) .. (0.1,1.15) ;
\end{tikzpicture}.
 We verify that a plane of this class is determined by the partial derivatives of a third degree homogeneous polynomial corresponding to an irreducible double point cubic with distinct tangents, this cubic being unique. The dimension sought is therefore equal to the dimension of the orbit of irreducible cubics with double point with distinct tangents in the space of projective plane cubics under the action of $Pgl(3)$ acting on $P$. This dimension is 8. (See the remarks preceding Theorem~\ref{specialthm} in \S 6.5).
    
    \item[2')] \underline{Dimension of} \begin{tikzpicture}[baseline={([yshift=-.8ex]current bounding box.center)},scale=0.2]
\draw  [line width=0.75] [line join = round][line cap = round] (0,0) .. controls  (-0.5,0.4) and (-1.3,1.1) .. (-1.9,1.5) .. controls (-2.2,1.7) and (-2.8,1.5) .. (-2.85,1.1) .. controls (-2.9,0.9) and (-2.7,0.55) .. (-2.5,0.5) .. controls (-2,0.4) and (-1,0.8) .. (-0.6,0.9) .. controls (-0.5,0.94) and (0.103,1.0695) .. (0.1,1.15) ;
\draw (-0.9,0.8) circle (0.4);
\end{tikzpicture}. Considering the duality, this dimension is also 8. 
        
    \item \underline{Dimension of} \begin{tikzpicture}[baseline={([yshift=-.8ex]current bounding box.center)},scale=0.2]
\draw (0,0) circle (1);
\draw (-1.5,0) -- (1.5,0);
\end{tikzpicture}
$(=\langle X^2+YZ,XY,XZ\rangle).$ $X$ is the unique linear form (up to a multiplication by a non-zero scalar) such that $\dim_k (XR_1\cap V) = 2$; we then have $\langle Y,Z\rangle = V : X$, so that the spaces $\langle X\rangle$ and $\langle Y, Z\rangle$ are determined only by $V$. Let $G_V$ be the isotropy group of $V$ in $SL(3)$. If $g\in G_V$, $g(X)=\alpha X$ and $g(Y),g(Z)\in \langle Y,Z\rangle$. Then $V$ contains $g(X^2+YZ)=\alpha^2X^2+g(Y)g(Z)$. Necessarly, by the conditions on $g(Y)$ and $g(Z)$, $\alpha^2X^2+g(Y)g(Z)$ is proportional to $X^2+YZ$: then either $g(Y)=aY$ or $g(Z)=\frac{\alpha^2}{a}Z$, or $g(Y)=aZ$ and $g(Z)=\frac{\alpha^2}{a}Y$. We have $1=\det g=\pm \alpha^3$ and $G_V\simeq (Z/3Z)^2\times k$ is of dimension 1. Therefore,  \begin{tikzpicture}[baseline={([yshift=-.8ex]current bounding box.center)},scale=0.2]
\draw (0,0) circle (1);
\draw (-1.5,0) -- (1.5,0);
\end{tikzpicture} is of dimension 7. 
    
    \item[3')]  By duality,\,\,\begin{tikzpicture}[baseline={([yshift=-.8ex]current bounding box.center)},scale=0.24]
\draw (0,0) circle (1);
\draw (-1.5,0) -- (1.5,0);
\draw (-1,0) circle (0.3);
\draw (1,0) circle (0.3);
\end{tikzpicture}\,\, is also of dimension 7.

    \item \underline{Dimension of}  \begin{tikzpicture}[baseline={([yshift=-.8ex]current bounding box.center)},scale=0.24]\draw  [color={rgb, 255:red, 0; green, 0; blue, 0 }  ][line width=0.75] [line join = round][line cap = round] (0,1.3) .. controls (0.5,1.3) and (0.95,1) .. (1.2,0.75) .. controls (1.4,0.6) and (1.5,0.4) .. (1.7,0.15) ;
\draw  [color={rgb, 255:red, 0; green, 0; blue, 0 }  ][line width=0.75] [line join = round][line cap = round] (0,1.3) .. controls (0.5,1.3) and (0.95,1.6) .. (1.2,1.85) .. controls (1.4,2) and (1.5,2.2) .. (1.7,2.55) ;
\draw (0,1.3) circle (0.4);
\end{tikzpicture}\quad $(=\langle XY,X^2+YZ,Z^2\rangle).$ The space $\langle Z\rangle$ is completely determined by $V$. Let $g\in G_V$, the isotropy group, satisfy $g(Z)= aZ$, $g(X) = bX + cY + dZ,$ $g(Y)=b'X + c'Y + d'Z$. Then, $$g(XY) = g(X)g(Y) = (bX + cY + dZ) (b'X + c'Y + d'Z) \in V$$ so that $$O\equiv (-bb' + c'd)YZ + (bd' + b'd)XZ + cc'Y^2 + (bd '+ b'd)XZ\ (mod.V)$$ and $g(X^2+YZ)=(bX+cY+dZ)^2+aZ(b'X+c'Y+d'Z)\in V$ so that  $$O\equiv (-b^2 + 2cd+ac')YZ + c^2Y^2+(2bd + ab')XZ \ (mod.V)$$ 
    Then $c=0$ and $1=\det g=abc'$. Finally, $g(Z)=aZ, g(X)=bX,$ ${g(Y)=c'Y,}$ with $b^3=1$, and $a=b^2/c'$, $c'\in k^*$. $G_V\simeq Z_3\times k^*$ is of dimension 1 and \begin{tikzpicture}[baseline={([yshift=-.8ex]current bounding box.center)},scale=0.24]\draw  [color={rgb, 255:red, 0; green, 0; blue, 0 }  ][line width=0.75] [line join = round][line cap = round] (0,1.3) .. controls (0.5,1.3) and (0.95,1) .. (1.2,0.75) .. controls (1.4,0.6) and (1.5,0.4) .. (1.7,0.15) ;
\draw  [color={rgb, 255:red, 0; green, 0; blue, 0 }  ][line width=0.75] [line join = round][line cap = round] (0,1.3) .. controls (0.5,1.3) and (0.95,1.6) .. (1.2,1.85) .. controls (1.4,2) and (1.5,2.2) .. (1.7,2.55) ;
\draw (0,1.3) circle (0.4);
\end{tikzpicture} is of dimension 7. 
    
    \item \begin{tikzpicture}[baseline={([yshift=-.8ex]current bounding box.center)},scale=0.17]
\draw (-0.4,0) -- (3.4,0);
\draw (-0.2,-0.4) -- (1.7,3.4);
\draw (3.2,-0.4) -- (1.3,3.4);
\draw (0,0) circle (0.4);
\draw (3,0) circle (0.4);
\draw (1.5,3) circle (0.4);
\end{tikzpicture} $(=\langle X^2,Y^2,Z^2\rangle)$.\newline
The set of 3 spaces $\langle X\rangle, \langle Y\rangle, \langle Z\rangle$ is determined by $V$ and $V$ determines them in turn. If $g\in G_V$, it is the product of a substitution of variables and a transformation of the form $(X,Y,Z)\rightarrow (aX,bY,cZ)$ with $abc=1$. $\dim G_V =2 $ and $\dim \begin{tikzpicture}[baseline={([yshift=-.8ex]current bounding box.center)},scale=0.17]
\draw (-0.4,0) -- (3.4,0);
\draw (-0.2,-0.4) -- (1.7,3.4);
\draw (3.2,-0.4) -- (1.3,3.4);
\draw (0,0) circle (0.4);
\draw (3,0) circle (0.4);
\draw (1.5,3) circle (0.4);
\end{tikzpicture} =6$. 
    
    \item[5')] By duality, $\dim \begin{tikzpicture}[baseline={([yshift=-.8ex]current bounding box.center)},scale=0.17]
\draw (-0.4,0) -- (3.4,0);
\draw (-0.2,-0.4) -- (1.7,3.4);
\draw (3.2,-0.4) -- (1.3,3.4);
\end{tikzpicture} =6$. 
    
    \item  \underline{Dimension of}  \begin{tikzpicture}[baseline={([yshift=-.8ex]current bounding box.center)},scale=0.2]
\draw (0,0) circle (1);
\draw (-1.5,-1) -- (1.5,-1);
\draw (0,-1) circle(0.3);
\end{tikzpicture} \quad $(=\langle XY,X^2,Z^2+YZ\rangle).$ \newline 
$\langle X\rangle$ and $\langle X,Y\rangle$ are uniquely determined by $\langle V\rangle$. If $g\in G_V$, $g(X)=aX$, $g(Y)=bX+cY$, $g(Z)=dX+eY+fZ$, and $1=\det g=acf$ and $g(Z^2+YZ)=(dX+eY+fZ)[(d+b)X+(e+c)Y+fZ]\equiv 0\  (\mod \, V),$ therefore $$0\equiv e(fc)Y^2+f(2d+b)XZ+f(2e+c)YZ+f^2Z^2\ \ (\mod\, V),$$ 
    hence, ($c=0$ or $e+c=0$) and $b=-2d$, $2e+c=f$, $a=\frac{1}{cf}$. Finally, the isotropy group $G_V=Z/2Z\times k^*\times k$ is of dimension 2, and \begin{tikzpicture}[baseline={([yshift=-.8ex]current bounding box.center)},scale=0.2]
\draw (0,0) circle (1);
\draw (-1.5,-1) -- (1.5,-1);
\draw (0,-1) circle(0.3);
\end{tikzpicture} is of dimension 6. 
    
    \item[6')] By duality, \begin{tikzpicture}[baseline={([yshift=-.8ex]current bounding box.center)},scale=0.2]
\draw (0,0) circle (0.5);
\draw (1,2) -- (1,-2);
\draw (2,0) circle(0.5);
\draw (-1,0.2) -- (3,0.2);
\draw (-1,-0.2) -- (3,-0.2);
\end{tikzpicture}\,\, is also of dimension 6. 
    
    \item \begin{tikzpicture}[baseline={([yshift=-.8ex]current bounding box.center)},scale=0.2]
\draw (0,0) circle (0.5);
\draw (0,2) -- (0,-2);
\draw (2,0) circle(0.5);
\draw (-1,0) -- (3,0);
\end{tikzpicture}\,\, $(=\langle Y^2,XY,Z^2\rangle)$. Again, $\langle Y\rangle$, $\langle X,Y\rangle$ and $\langle Z\rangle$ are uniquely determined by $V$. If $g\in G_V$, $g(Y)=aY$, $g(Z)=bZ$ and $g(X)=cX+dY,$ $\det g=abc$. Conversely, such an automorphism is in $G_V$. Therefore $G_V\simeq (k^{*})^2\times k$ has dimension 3 and \begin{tikzpicture}[baseline={([yshift=-.8ex]current bounding box.center)},scale=0.2]
\draw (0,0) circle (0.5);
\draw (0,2) -- (0,-2);
\draw (2,0) circle(0.5);
\draw (-1,0) -- (3,0);
\end{tikzpicture}\,\, has dimension 5. 
    
    \item[7')] By duality  \begin{tikzpicture} [baseline={([yshift=-.8ex]current bounding box.center)},scale=0.2]
\draw (1,2) -- (1,-2);
\draw (2,0) circle(0.5);
\draw (-1,0.2) -- (3,0.2);
\draw (-1,-0.2) -- (3,-0.2);
\end{tikzpicture} has dimension 5. . 
    
    \item \begin{tikzpicture}[scale=0.3]
\draw (0,0) -- (2,0);
\draw (0,-0.2) -- (2,-0.2);
\draw (0,0.2) -- (2,0.2);
\draw (1,0) circle (0.5);
\end{tikzpicture}\quad $(=\langle Y^2,XY,YZ-X^2\rangle)$.\newline 
The vector space determines $\langle Y\rangle$ and $\langle X-Y\rangle.$ If $g\in G_V$, $g(Y)=aY$, $g(X)=bX+cY$, and $g(Z)=dX+eY+fZ$, $1=\det g=abf$ and $g(YZ-X^2)\in V\Rightarrow 0\equiv aY(dX+eY+fZ)-(bX+cY)^2 \ (\mod V) \Rightarrow afYZ-b^2X^2\in V\Rightarrow af=b^2.$ Finally $G_V\simeq Z_3\times k^*\times k^3$ is of dimension 4 and \begin{tikzpicture}[scale=0.3]
\draw (0,0) -- (2,0);
\draw (0,-0.2) -- (2,-0.2);
\draw (0,0.2) -- (2,0.2);
\draw (1,0) circle (0.5);
\end{tikzpicture} has dimension 4. 
    
    \item It is easy to verify that \begin{tikzpicture}[baseline={([yshift=-.8ex]current bounding box.center)},scale=0.3]
\draw (0,0) circle (0.5);
\draw[color=gray] (-0.5,-1) -- (0.5,1);
\draw[color=gray] (-0.3,-1.1) -- (0.7,0.9);
\draw[color=gray] (-0.1,-1.2) -- (0.9,0.8);
\draw[color=gray] (0.1,-1.3) -- (1.1,0.7);
\draw[color=gray] (-0.7,-0.9) -- (0.3,1.1);
\draw[color=gray] (-0.9,-0.8) -- (0.1,1.2);
\draw[color=gray] (-1.1,-0.7) -- (-0.1,1.3);
\end{tikzpicture}
        \,\,$(=\langle X^2,XY,XZ\rangle)$ and \begin{tikzpicture}[baseline={([yshift=-.8ex]current bounding box.center)},scale=0.25]
\draw (-0.8,1) .. controls (-0.3,0) and (0.7,0) .. (1.2,1);
\draw[color=gray] (0.5,-0.75) -- (1.5,1.25);
\draw[color=gray] (0.3,-0.7) -- (1.3,1.3);
\draw[color=gray] (0.1,-0.65) -- (1.1,1.35);
\draw[color=gray] (-0.1,-0.6) -- (0.9,1.4);
\draw[color=gray] (-0.3,-0.55) -- (0.7,1.45);
\draw[color=gray] (-0.5,-0.5) -- (0.5,1.5);
\draw[color=gray] (-0.7,-0.45) -- (0.3,1.55);
\draw[color=gray] (-0.9,-0.4) -- (0.1,1.6);
\draw[color=gray] (-1.1,-0.35) -- (-0.1,1.65);
\draw[color=gray] (-1.3,-0.3) -- (-0.3,1.7);
\draw[color=gray] (-1.5,-0.25) -- (-0.5,1.76);
\end{tikzpicture} $(=\langle X^2,Y^2,(X+Y)^2\rangle)$ both have dimension 2. 
\end{enumerate}

\subsection{Scheme length as an invariant of special orbits}\label{lengthsec}
\begin{definition}
Let $u\in \Grass(3,R_2)$. Let $P,Q,R\in R_2$ such that $u=\langle P,Q,R\rangle.$ We call  ``projective scheme associated with $u"$ the scheme $$\mathrm{Proj}\ {\sf k} [X, Y, Z] / (P, Q, R).$$ In general, this scheme is empty, and if $(P, Q, R)$ does not belong to the orbit of $\langle X^2, XY, XZ\rangle$, nor to the orbit of $\langle X^2,XY,Y^2\rangle$, then it is finite. 
\end{definition}
\begin{definition}
Let $S$ be a finite scheme. We call \emph{scheme length} of this scheme the integer $$\sum_{s\in S} \ell(O_{S,s}) \text{ where for } s\in S,\ O_{S,s} \text{ is the local ring of } S \text{ in } s.$$
\end{definition}

In what follows, we give the length of the projective schemes associated to different orbits of $\Grass(3,R_2)$, an integer in the interval $[0,3]$. Indeed, two elements of $\Grass(3,R_2)$ that belong to the same orbit have isomorphic associated schemes, so have the same scheme length. The calculation of this length is easily done by carrying out a generic dehomogenization to obtain $\mathrm{Proj}_k [X, Y, Z] / (P, Q, R)$ in the form of the spectrum of an Artinian algebra. The length in question is the height $\ell$ of the infinite string of the Hilbert function $H(k [X, Y, Z] / (P, Q, R))=(1,3,3,a,\overline{\ell})$. When $(P,Q,R)$ is a complete intersection this is $0$.

The Table \ref{lengthfig} shows this scheme length $\ell$ immediately after the symbol associated with each orbit. It also summarizes the two previous discussions: each horizontal line of the table corresponds to a dimension of an orbit, that is noted at its right end. The axial symmetry corresponds to duality.\footnote{Note of Translator: the original table did not have the dimension two entries,
J.~Emsalem and A. Iarrobino perhaps considering them as too special.}\vskip 0.2cm\noindent
\underline{Notation}: If $U$ is the class of an element in $\Grass(3,R_2),$ in the sequel we let $\ell(U)$ be the length of the projective schemes associated to its various elements.

\begin{table}[h]
\begin{center}

\tikzset{every picture/.style={line width=0.75pt}} 

\tikzset{every picture/.style={line width=0.75pt}} 

\begin{tikzpicture}[x=0.75pt,y=0.75pt,yscale=-1,xscale=1]

\draw  [color={rgb, 255:red, 0; green, 0; blue, 0 }  ][line width=0.75] [line join = round][line cap = round] (101.07,41.16) .. controls (92.33,49.13) and (79.2,63.2) .. (69.7,69.22) .. controls (63.36,73.24) and (53.96,70.82) .. (52.74,62.45) .. controls (52.09,58.01) and (55.63,51.64) .. (59.52,50.83) .. controls (66.17,49.46) and (83.9,55.92) .. (90.89,58.58) .. controls (93.06,59.4) and (102.76,61.86) .. (102.76,63.42) ;
\draw   (186.23,58.6) .. controls (186.23,51.98) and (191.55,46.61) .. (198.1,46.61) .. controls (204.66,46.61) and (209.97,51.98) .. (209.97,58.6) .. controls (209.97,65.23) and (204.66,70.6) .. (198.1,70.6) .. controls (191.55,70.6) and (186.23,65.23) .. (186.23,58.6) -- cycle ;
\draw    (215.13,91) .. controls (245.02,68.34) and (191.96,46.43) .. (221.86,23.77) ;
\draw   (320.23,55.6) .. controls (320.23,48.98) and (325.55,43.61) .. (332.1,43.61) .. controls (338.66,43.61) and (343.97,48.98) .. (343.97,55.6) .. controls (343.97,62.23) and (338.66,67.6) .. (332.1,67.6) .. controls (325.55,67.6) and (320.23,62.23) .. (320.23,55.6) -- cycle ;
\draw    (349.13,88) .. controls (379.02,65.34) and (325.96,43.43) .. (355.86,20.77) ;
\draw  [color={rgb, 255:red, 0; green, 0; blue, 0 }  ][line width=0.75] [line join = round][line cap = round] (520.22,35.32) .. controls (511.05,43.49) and (497.27,57.89) .. (487.29,64.06) .. controls (480.64,68.17) and (470.77,65.69) .. (469.49,57.12) .. controls (468.81,52.58) and (472.52,46.06) .. (476.61,45.23) .. controls (483.59,43.81) and (502.2,50.43) .. (509.54,53.16) .. controls (511.81,54) and (522,56.52) .. (522,58.11) ;
\draw   (498.62,51.24) .. controls (498.62,47.75) and (501.17,44.91) .. (504.3,44.91) .. controls (507.44,44.91) and (509.99,47.75) .. (509.99,51.24) .. controls (509.99,54.73) and (507.44,57.57) .. (504.3,57.57) .. controls (501.17,57.57) and (498.62,54.73) .. (498.62,51.24) -- cycle ;
\draw   (134.67,132.5) .. controls (134.67,123.94) and (141.6,117) .. (150.15,117) .. controls (158.7,117) and (165.64,123.94) .. (165.64,132.5) .. controls (165.64,141.06) and (158.7,148) .. (150.15,148) .. controls (141.6,148) and (134.67,141.06) .. (134.67,132.5) -- cycle ;
\draw    (113,134.51) -- (187,133.82) ;
\draw  [color={rgb, 255:red, 0; green, 0; blue, 0 }  ][line width=0.75] [line join = round][line cap = round] (279.42,134.6) .. controls (287.93,134.6) and (296.4,128.02) .. (300.95,123.59) .. controls (303.46,121.14) and (309.7,115.45) .. (309.7,113.62) ;
\draw   (273.82,133.25) .. controls (274.1,129.02) and (277.27,125.84) .. (280.91,126.17) .. controls (284.55,126.49) and (287.28,130.19) .. (287,134.43) .. controls (286.72,138.67) and (283.54,141.84) .. (279.9,141.51) .. controls (276.27,141.19) and (273.54,137.49) .. (273.82,133.25) -- cycle ;
\draw  [color={rgb, 255:red, 0; green, 0; blue, 0 }  ][line width=0.75] [line join = round][line cap = round] (310,153) .. controls (305.35,145.45) and (296.4,140.81) .. (291,138.7) .. controls (288.02,137.54) and (280.88,134.49) .. (279.67,135.29) ;
\draw   (425.89,130.5) .. controls (425.89,118.63) and (435.49,109) .. (447.34,109) .. controls (459.19,109) and (468.79,118.63) .. (468.79,130.5) .. controls (468.79,142.37) and (459.19,152) .. (447.34,152) .. controls (435.49,152) and (425.89,142.37) .. (425.89,130.5) -- cycle ;
\draw    (395.05,132.65) -- (498,131.79) ;
\draw   (421.27,132.22) .. controls (421.27,129.19) and (423.73,126.73) .. (426.75,126.73) .. controls (429.77,126.73) and (432.23,129.19) .. (432.23,132.22) .. controls (432.23,135.25) and (429.77,137.71) .. (426.75,137.71) .. controls (423.73,137.71) and (421.27,135.25) .. (421.27,132.22) -- cycle ;
\draw   (462.45,132.22) .. controls (462.45,129.19) and (464.91,126.73) .. (467.93,126.73) .. controls (470.95,126.73) and (473.41,129.19) .. (473.41,132.22) .. controls (473.41,135.25) and (470.95,137.71) .. (467.93,137.71) .. controls (464.91,137.71) and (462.45,135.25) .. (462.45,132.22) -- cycle ;
\draw    (75.16,179) -- (106.33,227) ;
\draw    (54.37,220.83) -- (112,220.83) ;
\draw    (57.2,227) -- (84.12,179) ;
\draw   (212.51,222.06) .. controls (212.51,219.34) and (214.75,217.13) .. (217.51,217.13) .. controls (220.27,217.13) and (222.51,219.34) .. (222.51,222.06) .. controls (222.51,224.79) and (220.27,227) .. (217.51,227) .. controls (214.75,227) and (212.51,224.79) .. (212.51,222.06) -- cycle ;
\draw   (198.2,201.95) .. controls (198.2,191.27) and (206.96,182.62) .. (217.77,182.62) .. controls (228.58,182.62) and (237.34,191.27) .. (237.34,201.95) .. controls (237.34,212.63) and (228.58,221.29) .. (217.77,221.29) .. controls (206.96,221.29) and (198.2,212.63) .. (198.2,201.95) -- cycle ;
\draw    (170.05,222.45) -- (264,221.68) ;
\draw    (363.65,189) -- (363.65,248) ;
\draw    (327.62,210) -- (401,210.06) ;
\draw   (332.38,213.64) .. controls (332.71,209.42) and (336.48,206.26) .. (340.79,206.59) .. controls (345.11,206.91) and (348.34,210.59) .. (348.01,214.81) .. controls (347.67,219.03) and (343.91,222.19) .. (339.59,221.87) .. controls (335.28,221.55) and (332.05,217.87) .. (332.38,213.64) -- cycle ;
\draw   (378.76,213.64) .. controls (379.09,209.42) and (382.86,206.26) .. (387.18,206.59) .. controls (391.49,206.91) and (394.72,210.59) .. (394.39,214.81) .. controls (394.06,219.03) and (390.29,222.19) .. (385.98,221.87) .. controls (381.66,221.55) and (378.43,217.87) .. (378.76,213.64) -- cycle ;
\draw    (327.62,219.27) -- (401,219.33) ;
\draw    (481.09,176) -- (515.71,230) ;
\draw    (458,223.06) -- (522,223.06) ;
\draw    (461.15,230) -- (491.04,176) ;
\draw   (481.8,184.19) .. controls (481.8,181.47) and (484.05,179.27) .. (486.82,179.27) .. controls (489.59,179.27) and (491.84,181.47) .. (491.84,184.19) .. controls (491.84,186.91) and (489.59,189.11) .. (486.82,189.11) .. controls (484.05,189.11) and (481.8,186.91) .. (481.8,184.19) -- cycle ;
\draw   (506.97,222.76) .. controls (506.97,220.04) and (509.22,217.84) .. (512,217.84) .. controls (514.77,217.84) and (517.02,220.04) .. (517.02,222.76) .. controls (517.02,225.48) and (514.77,227.69) .. (512,227.69) .. controls (509.22,227.69) and (506.97,225.48) .. (506.97,222.76) -- cycle ;
\draw   (460.56,222.76) .. controls (460.56,220.04) and (462.81,217.84) .. (465.58,217.84) .. controls (468.35,217.84) and (470.6,220.04) .. (470.6,222.76) .. controls (470.6,225.48) and (468.35,227.69) .. (465.58,227.69) .. controls (462.81,227.69) and (460.56,225.48) .. (460.56,222.76) -- cycle ;
\draw    (197.13,265) -- (197.13,326) ;
\draw    (178.5,286.71) -- (239.49,286.77) ;
\draw    (177.52,295.5) -- (240,295.5) ;
\draw   (218.19,290.48) .. controls (218.54,286.12) and (222.45,282.85) .. (226.93,283.18) .. controls (231.41,283.52) and (234.76,287.33) .. (234.42,291.69) .. controls (234.08,296.05) and (230.16,299.32) .. (225.68,298.99) .. controls (221.2,298.65) and (217.85,294.84) .. (218.19,290.48) -- cycle ;
\draw    (374.13,268) -- (374.13,329) ;
\draw    (355.5,289.71) -- (416.49,289.77) ;
\draw    (354.52,298.5) -- (417,298.5) ;
\draw   (395.19,293.48) .. controls (395.54,289.12) and (399.45,285.85) .. (403.93,286.18) .. controls (408.41,286.52) and (411.76,290.33) .. (411.42,294.69) .. controls (411.08,299.05) and (407.16,302.32) .. (402.68,301.99) .. controls (398.2,301.65) and (394.85,297.84) .. (395.19,293.48) -- cycle ;
\draw   (366.64,292.99) .. controls (366.98,288.63) and (370.89,285.36) .. (375.37,285.7) .. controls (379.85,286.03) and (383.21,289.84) .. (382.86,294.2) .. controls (382.52,298.57) and (378.61,301.83) .. (374.13,301.5) .. controls (369.65,301.17) and (366.29,297.36) .. (366.64,292.99) -- cycle ;
\draw    (264.62,367.22) -- (332.78,367.29) ;
\draw    (264.52,377.31) -- (333.35,377.31) ;
\draw   (287.98,371.54) .. controls (288.36,366.54) and (292.73,362.79) .. (297.74,363.17) .. controls (302.75,363.55) and (306.5,367.93) .. (306.11,372.93) .. controls (305.73,377.94) and (301.36,381.69) .. (296.35,381.31) .. controls (291.34,380.92) and (287.59,376.55) .. (287.98,371.54) -- cycle ;
\draw    (264.62,372.22) -- (332.78,372.29) ;
\draw   (191.98,435.8) .. controls (192.36,430.79) and (196.73,427.04) .. (201.74,427.42) .. controls (206.75,427.81) and (210.5,432.18) .. (210.11,437.19) .. controls (209.73,442.19) and (205.36,445.94) .. (200.35,445.56) .. controls (195.34,445.18) and (191.59,440.81) .. (191.98,435.8) -- cycle ;
\draw [color={rgb, 255:red, 128; green, 128; blue, 128 }  ,draw opacity=1 ]   (202.67,456) -- (225.56,418.37) ;
\draw [color={rgb, 255:red, 128; green, 128; blue, 128 }  ,draw opacity=1 ]   (177.11,453.94) -- (200,416.31) ;
\draw [color={rgb, 255:red, 128; green, 128; blue, 128 }  ,draw opacity=1 ]   (183.53,454.61) -- (206.42,416.98) ;
\draw [color={rgb, 255:red, 128; green, 128; blue, 128 }  ,draw opacity=1 ]   (189.6,455.31) -- (212.49,417.68) ;
\draw [color={rgb, 255:red, 128; green, 128; blue, 128 }  ,draw opacity=1 ]   (195.67,457) -- (218.56,419.37) ;
\draw  [draw opacity=0] (409.69,424.11) .. controls (409.86,425.14) and (409.96,426.19) .. (410.01,427.27) .. controls (410.48,439.25) and (402.78,449.33) .. (392.82,449.78) .. controls (382.86,450.23) and (374.4,440.88) .. (373.93,428.9) .. controls (373.85,426.72) and (374.03,424.6) .. (374.46,422.58) -- (391.97,428.09) -- cycle ; \draw   (409.69,424.11) .. controls (409.86,425.14) and (409.96,426.19) .. (410.01,427.27) .. controls (410.48,439.25) and (402.78,449.33) .. (392.82,449.78) .. controls (382.86,450.23) and (374.4,440.88) .. (373.93,428.9) .. controls (373.85,426.72) and (374.03,424.6) .. (374.46,422.58) ;
\draw [color={rgb, 255:red, 128; green, 128; blue, 128 }  ,draw opacity=1 ]   (377.14,451.63) -- (397.41,415.84) ;
\draw [color={rgb, 255:red, 128; green, 128; blue, 128 }  ,draw opacity=1 ]   (354.5,449.67) -- (374.77,413.88) ;
\draw [color={rgb, 255:red, 128; green, 128; blue, 128 }  ,draw opacity=1 ]   (360.19,450.31) -- (380.46,414.52) ;
\draw [color={rgb, 255:red, 128; green, 128; blue, 128 }  ,draw opacity=1 ]   (365.56,450.97) -- (385.84,415.18) ;
\draw [color={rgb, 255:red, 128; green, 128; blue, 128 }  ,draw opacity=1 ]   (370.94,452.58) -- (391.21,416.8) ;
\draw [color={rgb, 255:red, 128; green, 128; blue, 128 }  ,draw opacity=1 ]   (405.48,454.48) -- (425.76,418.7) ;
\draw [color={rgb, 255:red, 128; green, 128; blue, 128 }  ,draw opacity=1 ]   (382.84,452.52) -- (403.12,416.73) ;
\draw [color={rgb, 255:red, 128; green, 128; blue, 128 }  ,draw opacity=1 ]   (388.53,453.16) -- (408.81,417.38) ;
\draw [color={rgb, 255:red, 128; green, 128; blue, 128 }  ,draw opacity=1 ]   (393.91,453.82) -- (414.18,418.04) ;
\draw [color={rgb, 255:red, 128; green, 128; blue, 128 }  ,draw opacity=1 ]   (399.28,455.43) -- (419.56,419.65) ;

\draw (231.94,54.75) node [anchor=north west][inner sep=0.75pt]    {$j( \lambda )$};
\draw (363,34.4) node [anchor=north west][inner sep=0.75pt]    {$j\left( -\frac{2}{\lambda }\right)$};
\draw (118,46.4) node [anchor=north west][inner sep=0.75pt]    {$1$};
\draw (270,46.4) node [anchor=north west][inner sep=0.75pt]    {$0$};
\draw (418,45.4) node [anchor=north west][inner sep=0.75pt]    {$0$};
\draw (548,44.4) node [anchor=north west][inner sep=0.75pt]    {$0$};
\draw (611,45.4) node [anchor=north west][inner sep=0.75pt]    {$8$};
\draw (201,123.4) node [anchor=north west][inner sep=0.75pt]    {$2$};
\draw (330,121.4) node [anchor=north west][inner sep=0.75pt]    {$1$};
\draw (522,121.4) node [anchor=north west][inner sep=0.75pt]    {$0$};
\draw (612,122.4) node [anchor=north west][inner sep=0.75pt]    {$7$};
\draw (120,205.4) node [anchor=north west][inner sep=0.75pt]    {$3$};
\draw (273,207.4) node [anchor=north west][inner sep=0.75pt]    {$2$};
\draw (418,206.4) node [anchor=north west][inner sep=0.75pt]    {$2$};
\draw (549,204.4) node [anchor=north west][inner sep=0.75pt]    {$0$};
\draw (612,204.4) node [anchor=north west][inner sep=0.75pt]    {$6$};
\draw (253,281.4) node [anchor=north west][inner sep=0.75pt]    {$3$};
\draw (430,284.4) node [anchor=north west][inner sep=0.75pt]    {$2$};
\draw (358,364.4) node [anchor=north west][inner sep=0.75pt]    {$3$};
\draw (611,285) node [anchor=north west][inner sep=0.75pt]   [align=left] {5};
\draw (611,368) node [anchor=north west][inner sep=0.75pt]   [align=left] {4};
\draw (611,429.4) node [anchor=north west][inner sep=0.75pt]    {$2$};
\draw (435,425.4) node [anchor=north west][inner sep=0.75pt]    {$3$};
\draw (227.56,427.77) node [anchor=north west][inner sep=0.75pt]    {$\infty $};

\end{tikzpicture}

\end{center}
\caption{Length of the associated projective scheme to a net of conics.}\label{lengthfig}
\end{table}

\subsection{Impossibility of some specializations of nets of conics}
We call from now on $\mathcal C$ the set of classes of projective planes of conics modulo the natural action of $\mathrm{GL}(3)$.
\vskip 0.2cm\noindent
\underline{Definition}\label{elementarydef}
Suppose $A,B\in \mathcal C.$ We say that $(A,B)$ is an {\it elementary specialization} if $B$ is a specialization of $A$, (if $A\neq B$) and if, whenever $C$ is a specialization of $A$ and specializes to $B$, then $C$ is equal to $A$ or $B$. It will turn out that all elementary specializations of orbits in the nets of conics will be from one of dimension $c$ to another of dimension $c-1$.
\par
We have the following proposition, whose proof is trivial.\footnote{{\it Note of translators}: for (b) above, the scheme length $\ell(A)$ is $\dim_{\sf k} A^\prime, A^\prime= A/({\sf l})$ where $\sf l$ is a general enough linear form, and this dimension is upper semicontinuous. The \emph{determinantal} nets of conics, namely, those arising as the $2\times 2$ minors of a $2\times 3$ matrix of linear forms, are the four nets of maximum finite scheme length $3$. These are seen in the literature (see Appendix B).}
\begin{proposition}
Let $(A,B)$ be a specialization with $A\neq B.$  Then
\begin{enumerate}[a)]
\item $\dim B<\dim A$;
\item $\ell(B)\ge \ell(A)$; 
\item if a representative of $A$ has a non-empty intersection with $D_2$, all representatives of $B$ have a non-empty intersection with $D_2.$
\item $(A^\perp,B^\perp)$ is a specialization with $A^\perp\neq B^\perp$.
\end{enumerate}
\end{proposition}
Taking into account a) of the Proposition, the description of the order of genericity reduces to the establishing of an exhaustive list of elementary specializations.\par
In order to establish this list, the Proposition permits us to eliminate the possibility of the existence of elementary specializations $(A,B)$ that do not figure in Table \ref{table1}, with the exception of the following potential families of specializations:\par
\begin{center}
    \begin{tikzpicture}[baseline={([yshift=-.8ex]current bounding box.center)},scale=0.6]
\draw (0,0) circle (0.3);
\draw    (0.7,0.8) .. controls  (0,-0.3) and (1.2,0.2) .. (0.4,-0.8);
\draw (0.7,-0.4) node [anchor=west]{$\lambda$};
\end{tikzpicture} $\rightarrow$\quad \begin{tikzpicture}[baseline={([yshift=-.8ex]current bounding box.center)},scale=0.4]
\draw (0,0) circle (1);
\draw (-1.5,0) -- (1.5,0);
\end{tikzpicture}\quad and \quad  \begin{tikzpicture}[baseline={([yshift=-.8ex]current bounding box.center)},scale=0.6]
\draw (0,0) circle (0.3);
\draw    (0.7,0.8) .. controls  (0,-0.3) and (1.2,0.2) .. (0.4,-0.8);
\draw (0.7,-0.4) node [anchor=west]{$\lambda$};
\end{tikzpicture} $\rightarrow$\quad \begin{tikzpicture}[baseline={([yshift=-.8ex]current bounding box.center)},scale=0.4]
\draw (0,0) circle (1);
\draw (-1.5,0) -- (1.5,0);
\draw (-1,0) circle (0.3);
\draw (1,0) circle (0.3);
\end{tikzpicture}
\end{center}

\begin{center}
\begin{tikzpicture}[baseline={([yshift=-.8ex]current bounding box.center)},scale=0.6]
\draw (0,0) circle (0.3);
\draw    (0.7,0.8) .. controls  (0,-0.3) and (1.2,0.2) .. (0.4,-0.8);
\draw (0.7,-0.4) node [anchor=west]{$\lambda$};
\end{tikzpicture} $\rightarrow$\quad \begin{tikzpicture}[baseline={([yshift=-.8ex]current bounding box.center)},scale=0.2]
\draw (-0.4,0) -- (3.4,0);
\draw (-0.2,-0.4) -- (1.7,3.4);
\draw (3.2,-0.4) -- (1.3,3.4);
\end{tikzpicture}\quad and \quad  \begin{tikzpicture}[baseline={([yshift=-.8ex]current bounding box.center)},scale=0.6]
\draw (0,0) circle (0.3);
\draw    (0.7,0.8) .. controls  (0,-0.3) and (1.2,0.2) .. (0.4,-0.8);
\draw (0.7,-0.4) node [anchor=west]{$\lambda$};
\end{tikzpicture} $\rightarrow$\quad \begin{tikzpicture}[baseline={([yshift=-.8ex]current bounding box.center)},scale=0.2]
\draw (-0.4,0) -- (3.4,0);
\draw (-0.2,-0.4) -- (1.7,3.4);
\draw (3.2,-0.4) -- (1.3,3.4);
\draw (0,0) circle (0.3);
\draw (3,0) circle (0.3);
\draw (1.5,3) circle (0.3);
\end{tikzpicture}
\end{center} 
\noindent which are duals in pairs, and whose elimination is the goal of the following paragraph.
\subsection{Specialization of plane cubics and related nets of conics}\label{specsec}
\subsubsection{Specializations of Jacobian planes of a family of cubics $f$ having constant ${\sf j}$-invariant.}\label{specjsec}
In what follows, we will consider the specializations of an orbit  \begin{tikzpicture}[baseline={([yshift=-.8ex]current bounding box.center)},scale=0.4]
\draw (0,0) circle (0.3);
\draw    (0.7,0.8) .. controls  (0,-0.3) and (1.2,0.2) .. (0.4,-0.8);
\draw (0.7,-0.4) node [anchor=west]{$\lambda$};
\end{tikzpicture}, $\lambda$ being fixed, which comes down to fixing ${\sf j}(\lambda)$.\par
1) The potential specializations \begin{tikzpicture}[baseline={([yshift=-.8ex]current bounding box.center)},scale=0.4]
\draw (0,0) circle (0.3);
\draw    (0.7,0.8) .. controls  (0,-0.3) and (1.2,0.2) .. (0.4,-0.8);
\draw (0.7,-0.4) node [anchor=west]{$\lambda$};
\end{tikzpicture} $\rightarrow$ \begin{tikzpicture}[baseline={([yshift=-.8ex]current bounding box.center)},scale=0.2]
\draw (0,0) circle (1);
\draw (-1.5,0) -- (1.5,0);
\end{tikzpicture} and \begin{tikzpicture}[baseline={([yshift=-.8ex]current bounding box.center)},scale=0.4]
\draw (0,0) circle (0.3);
\draw    (0.7,0.8) .. controls  (0,-0.3) and (1.2,0.2) .. (0.4,-0.8);
\draw (0.7,-0.4) node [anchor=west]{$\lambda$};
\end{tikzpicture} $\rightarrow$ \begin{tikzpicture}[baseline={([yshift=-.8ex]current bounding box.center)},scale=0.2]
\draw (0,0) circle (1);
\draw (-1.5,0) -- (1.5,0);
\draw (-1,0) circle (0.3);
\draw (1,0) circle (0.3);
\end{tikzpicture} 
do not exist. It suffices, in considering the duality, to show that the first specialization does not exist. In considering the map which associates to $\pi$ the cubic $\Gamma_\pi$, and which is rational, this will imply that, in the space of cubics, the cubic that decomposes into a conic with a secant line will belong to the closure of the orbit of $\Gamma_{\pi_\lambda},$  where $\pi_\lambda$ is the plane associated to ${\langle X^2+\lambda YZ,Y^2+\lambda XZ,Z^2+\lambda XY\rangle.}$ However, as we have seen in Section~4.1.2, $ \Gamma_{\pi_\lambda}$ is smooth, which excludes the preceding possibility, according to the Proposition \ref{3prop} below.\par
In this way, in the space of projective plane cubics, each cubic \begin{tikzpicture}[baseline={([yshift=-.8ex]current bounding box.center)},scale=0.2]
\draw (0,0) circle (1);
\draw (-1.5,0) -- (1.5,0);
\end{tikzpicture} would belong to  the closure (in the sense of Zariski topology) of at least one orbit of a smooth cubic under the natural action of $Pgl(3)$. Thus, the closure of the orbits of these smooth cubics as well as the closure of the orbit of irreducible cubics having a double point with distinct tangents are hypersurfaces (thus, of dimension 8), in the space of cubics.\par
More precisely \cite[Exercise 14-31,\S 28]{Gur}\footnote{Exercise 15 of \cite[\S 28]{Gur} defines $S=S_4,T=S_6$ as invariants, and Exercise 29 is to show if $T^2-4S^3\not=0$ that each invariant of a canonical cubic form is a rational integral function of $S,T$.  {\it Note of translators}: G. Salmon \cite[Chapter V, \S V]{Sal} has a detailed study of invariants of the cubic $f$: \S 220, 221 give formulas for $S,T$ in terms of the coefficients of $f$. See also Aronhold invariants \cite[\S 4.6]{Ott2}, \cite[\S 10.3]{Dol1} (which classifies cubics over arbitrary algebraically closed fields, and specifies $S,T$), and \cite[\S 3.4]{Dol}.}
there are two fundamental relative homogeneous invariants $S_4$ and $S_6$ of degrees respectively $4$ and $6$ in the coefficients of a general cubic, such that a fundamental rational invariant (rationally equivalent to the classic $``{\sf j}''$) can be written $\frac{S_4^3}{S_6^2}.$  Here the closures of the orbits of smooth cubics with a fixed ${\sf j}$-invariant, and also the closure of the cubic with a double point having two distinct tangents constitute in each case a pencil of hypersurfaces of degree 12, whose general equation is written $\beta S_4^3+\alpha S_6^2=0, \beta, \alpha \in {\sf k}$. 

\par Since in a pencil of hypersurfaces, if a point belongs to two different hypersurfaces, then it belongs to all the hypersurfaces of the pencil,
we thus have the following result.
\begin{theorem}\label{specialthm} For a cubic whose orbit is of dimension strictly smaller than $8$ one of the two following occurs concerning its presence among the closures of the dimension eight orbits: \par
i. either it belongs to the closure of each of the $\sf j$-constant  orbits of the smooth cubics and also to the closure of the orbit of singular cubics having a double point with distinct tangents, or\par
 ii. it belongs to exactly one of them.
\end{theorem}

Also, the existence of the specialization of cubics \begin{tikzpicture}[baseline={([yshift=-.8ex]current bounding box.center)},scale=0.2]
\draw  [line width=0.75] [line join = round][line cap = round] (0,0) .. controls  (-0.5,0.4) and (-1.3,1.1) .. (-1.9,1.5) .. controls (-2.2,1.7) and (-2.8,1.5) .. (-2.85,1.1) .. controls (-2.9,0.9) and (-2.7,0.55) .. (-2.5,0.5) .. controls (-2,0.4) and (-1,0.8) .. (-0.6,0.9) .. controls (-0.5,0.94) and (0.103,1.0695) .. (0.1,1.15) ;
\end{tikzpicture} $\rightarrow$ \begin{tikzpicture}[baseline={([yshift=-.8ex]current bounding box.center)},scale=0.2]
\draw (0,0) circle (1);
\draw (-1.5,0) -- (1.5,0);
\end{tikzpicture}, which is shown below, proves
\begin{proposition}\label{2prop}
Every cubic that is comprised of a conic and a secant line belongs to the closure of the orbit of irreducible cubics having a singular point with two distinct tangents.
\end{proposition}
Thus, we will show that \begin{tikzpicture}[baseline={([yshift=-.8ex]current bounding box.center)},scale=0.2]
\draw (0,0) circle (1);
\draw (-1.5,0) -- (1.5,0);
\end{tikzpicture}  does not belong to the closure of the cubic $(X^3+Y^3+Z^3)$. The conjunction of the two propositions implies also
\begin{proposition}\label{3prop}
 Each cubic comprised of a conic and a secant line does not belong to the closure of the orbit of any smooth cubic.
 \end{proposition}
 It remains to show
 \begin{lemma}\label{7.13lem}
 A cubic comprised of a conic and a secant line does not belong to the closure of the orbit of $(X^3+Y^3+Z^3).$
 \end{lemma}
 To prove the Lemma, we will show
 \begin{proposition}\label{4prop}
Let $F$ be an element of $R_3.$  For the curve $(F)$ (${F=0}$) to belong to the closure of the orbit of 
$X^3+Y^3+Z^3,$ it is necessary and sufficient that there exists $G\in R_3$ such that $(G)\neq (F)$ and we have $${\langle G'_X,G'_Y,G'_Z\rangle =\langle F'_X,F'_Y,F'_Z\rangle.}$$
\end{proposition}
\begin{proof}[Proof of Proposition]
In effect, if $(F)=(X^3+Y^3+Z^3),$ there exists $G=2X^3+Y^3+Z^3\neq F$ such that $\langle G'_X,G'_Y,G'_Z\rangle =\langle F'_X,F'_Y,F'_Z\rangle.$ In the reverse direction the Proposition~\ref{4.2prop} shows that if $(F)$ is smooth but not isomorphic to $(X^3+Y^3+Z^3),$ then each cubic $(G)$ such that $\langle G'_X,G'_Y,G'_Z\rangle =\langle F'_X,F'_Y,F'_Z\rangle$ is equal to $(F)$. 
\end{proof}
One verifies a similar result if $(F)$ is an irreducible cubic having a singular point with two distinct tangents.\par
Finally, the set of $(F)$ such that there is $(G)\neq (F)$ for which 
$\langle G'_X,G'_Y,G'_Z\rangle =\langle F'_X,F'_Y,F'_Z\rangle
$ is a closed hypersurface of the space of cubics. In effect, that condition translates itself to the annihilation of a homogeneous polynomial (hypersurface) which is the determinant associated to a system of linear conditions that generalizes Equation \eqref{partialeq} considered in \S 4.1.2. \par
It results from the last property that this hypersurface is exactly the closure of the orbit of $(X^3+Y^3+Z^3)$.\par
In using the preceding Proposition \ref{4prop} to show the Lemma \ref{7.13lem}, it suffices to exhibit an element $F\in R_3$, a cubic comprised of a conic and a secant line, and such that for all $G\in R_3$ the equality
$\langle G'_X,G'_Y,G'_Z\rangle =\langle F'_X,F'_Y,F'_Z\rangle$ implies $\langle G\rangle =\langle F\rangle.$ This is the case for $F=X(X^2-YZ)$.
            $\qquad\qquad\qquad $ Q.E.D.\vskip 0.2cm\par\noindent
            {\underline{Remark}} The ideal  of polynomials in the coefficients of  a general cubic that vanish on the orbit of a cubic comprised of a conic and a secant line (as $F=X(X^2-YZ)$ above), furnishes the example of an invariant ideal under $Pgl(3)$ that does not have a system of generators that are individually invariant (because they cannot be expressed as polynomials in $S_4$ and  $S_6$).\vskip 0.3cm\par\noindent
            2)  {\it A key specialization}:\label{keyspecial} The specialization of nets of conics \begin{tikzpicture}[baseline={([yshift=-.8ex]current bounding box.center)},scale=0.4]
\draw (0,0) circle (0.3);
\draw    (0.7,0.8) .. controls  (0,-0.3) and (1.2,0.2) .. (0.4,-0.8);
\draw (0.7,-0.4) node [anchor=west]{$\lambda$};
\end{tikzpicture} $\rightarrow$ \begin{tikzpicture}[baseline={([yshift=-.8ex]current bounding box.center)},scale=0.2]
\draw  [color={rgb, 255:red, 0; green, 0; blue, 0 }  ][line width=0.75] [line join = round][line cap = round] (0,1.3) .. controls (0.5,1.3) and (0.95,1) .. (1.2,0.75) .. controls (1.4,0.6) and (1.5,0.4) .. (1.7,0.15) ;
\draw  [color={rgb, 255:red, 0; green, 0; blue, 0 }  ][line width=0.75] [line join = round][line cap = round] (0,1.3) .. controls (0.5,1.3) and (0.95,1.6) .. (1.2,1.85) .. controls (1.4,2) and (1.5,2.2) .. (1.7,2.55) ;
\draw (0,1.3) circle (0.4);
\end{tikzpicture}  exists for any $\lambda$. In effect, given $\lambda$ there exists $\mu$ such that for all $b\neq 0$, the plane associated to the space $\langle ZX,YZ-X^2,(Y-bZ)(Y-\mu^2bZ)\rangle$ belongs to the orbit \begin{tikzpicture}[baseline={([yshift=-.8ex]current bounding box.center)},scale=0.4]
\draw (0,0) circle (0.3);
\draw    (0.7,0.8) .. controls  (0,-0.3) and (1.2,0.2) .. (0.4,-0.8);
\draw (0.7,-0.4) node [anchor=west]{$\lambda$};
\end{tikzpicture}. (Compare the conclusions of \S 4.1.1 and of \S 4.1.2 and make for $b\neq 0$ a change of variables transforming the expression given at the end of \S 4.1.1 to the expression given here.)\par
            For $b=0$ the space $\langle ZX,YZ-X^2,(Y-bZ)(X-\mu^2bZ)\rangle$ belongs to the orbit \begin{tikzpicture}[baseline={([yshift=-.8ex]current bounding box.center)},scale=0.2]
\draw  [color={rgb, 255:red, 0; green, 0; blue, 0 }  ][line width=0.75] [line join = round][line cap = round] (0,1.3) .. controls (0.5,1.3) and (0.95,1) .. (1.2,0.75) .. controls (1.4,0.6) and (1.5,0.4) .. (1.7,0.15) ;
\draw  [color={rgb, 255:red, 0; green, 0; blue, 0 }  ][line width=0.75] [line join = round][line cap = round] (0,1.3) .. controls (0.5,1.3) and (0.95,1.6) .. (1.2,1.85) .. controls (1.4,2) and (1.5,2.2) .. (1.7,2.55) ;
\draw (0,1.3) circle (0.4);
\end{tikzpicture} .\vskip 0.3cm \par\noindent
            
3) The specializations of nets of conics \begin{tikzpicture}[baseline={([yshift=-.8ex]current bounding box.center)},scale=0.4]
\draw (0,0) circle (0.3);
\draw    (0.7,0.8) .. controls  (0,-0.3) and (1.2,0.2) .. (0.4,-0.8);
\draw (0.7,-0.4) node [anchor=west]{$\lambda$};
\end{tikzpicture} $\rightarrow$ \begin{tikzpicture}[baseline={([yshift=-.8ex]current bounding box.center)},scale=0.15]
\draw (-0.4,0) -- (3.4,0);
\draw (-0.2,-0.4) -- (1.7,3.4);
\draw (3.2,-0.4) -- (1.3,3.4);
\end{tikzpicture} and \begin{tikzpicture}[baseline={([yshift=-.8ex]current bounding box.center)},scale=0.4]
\draw (0,0) circle (0.3);
\draw    (0.7,0.8) .. controls  (0,-0.3) and (1.2,0.2) .. (0.4,-0.8);
\draw (0.7,-0.4) node [anchor=west]{$\lambda$};
\end{tikzpicture} $\rightarrow$ \begin{tikzpicture}[baseline={([yshift=-.8ex]current bounding box.center)},scale=0.15]
\draw (-0.4,0) -- (3.4,0);
\draw (-0.2,-0.4) -- (1.7,3.4);
\draw (3.2,-0.4) -- (1.3,3.4);
\draw (0,0) circle (0.4);
\draw (3,0) circle (0.4);
\draw (1.5,3) circle (0.4);
\end{tikzpicture}\,\,do not exist. In considering the duality, it suffices to show that the first specialization does not exist. For, if it existed, in the space of cubics $(XYZ)$ would belong to the closure of the smooth cubic $\Gamma_{\pi_\lambda}.$
            However we have
            \begin{proposition}\label{6.16prop}
           In the space of cubics,  $(XYZ)$ belongs to the closure of the orbit of irreducible cubics having a double point with two distinct tangents, and it does not belong to the closure of the orbit of any smooth cubic. \end{proposition}
           \begin{enumerate}[a)]
               \item 
               Considering the family $XYZ+\mu(X^3+Y^3)$ whose elements for $\mu\neq 0$ belong to the closure of the orbit of the irreducible cubic having a double point with distinct tangents and whose element corresponding to $\mu=0$ is $(XYZ)$, one obtains the first part of the Proposition.
               \item  By Proposition \ref{3prop} the second part of the Proposition \ref{6.16prop} results from a) and the fact $(XYZ)$ does not belong to the closure of the orbits of $(X^3+Y^3+Z^3)$ (Lemma \ref{7.13lem}). In effect, up to multiplication by a scalar $XYZ$ is the unique element of $R_3$ whose partial derivatives generate the vector space $\langle XY,XZ,YZ\rangle.$
               \end{enumerate}
               \subsubsection{Specializations of singular cubics $\Gamma_\pi$}\label{singsec}
          \begin{transl} The diagrams of Table \ref{table1} (except for the last, dimension-two orbits), are those for the cubics  that are discriminants $\Gamma_\pi$ of the nets $\pi$ of conics: in each case this cubic $\Gamma_\pi$ is the locus of conics in the net $\pi$ that have a linear factor: the figure for the cubic is garnished by circles around points corresponding to a conic that is a  double line. A specialization of nets $A$ to $B$ requires the specialization of the corresponding discriminant cubics; but a specialization of cubics that are discriminants would not necessarily come from a specialization of nets.\footnote{For example \#6a and \#6d have the same discriminant cubic $ABC$ before decoration with double points, but there is no specialization of the corresponding nets.}  We specify here in Table \ref{discrimtable} the cubic discriminants of each of the nets of dimension at least two.  We give an example of the calculation for \#8a \footnote{Recall the usual notation:  \#8a refers to the orbit of dimension 8, and at the left of the table \ref{table1}.}.  We then provide a table of specializations of cubics, Table \ref{cubicfig}. \end{transl}
\vskip 0.2cm\noindent
{\it The discriminant cubic $\Gamma_\pi$ in the plane $\pi$ of the net.}          
If the net is generated by $[f,g,h]$, we let $$G=A f + B g + Ch$$ 
To get a singular curve, we should find the cubic curve $G=\Gamma_\pi$ in the variables $A,B,C$ such that $G_X=G_Y=G_Z=0$ has a nontrivial solution. 
 $ G_X=Af_X+Bg_X+C h_X=0$\\
$G_Y=Af_Y+B g_Y+C h_Y=0$\\
 $G_Z=A f_Z+B g_Z+C h_Z=0$.\\
\vskip -0.4cm
\begin{example}[Discriminant $\Gamma_\pi$ for \#8a]\label{8adiscex}
Here $\pi=\langle f,g,h\rangle =\langle XY,X^2+YZ,Y^2+XZ\rangle.$\\
$G=A XY+B (X^2+YZ) + C (Y^2+XZ)$\\
$G_X=A Y+2B X+C Z=0$\\
$G_Y=A X+B Z+2C Y =0$\\
$G_Z=B Y+C X=0$

$$\det \begin{pmatrix}
2B & A &C\\
A & 2C &B\\
C&B&0
\end{pmatrix}=C(AB-2C^2)-B(2B^2-AC)=2ABC -2(C^3+B^3)$$ which is the equation of a node. 
In Table \ref{discrimtable} we give the equations of the discriminant locus, up to a non-zero constant multiple; the diagram is that of the cubic $\Gamma_\pi$ in Table \ref{table1} with double points indicated.
\end{example}

  \begin{table}[h]\caption{The equations of the discriminant cubics of Table \ref{table1}}\label{discrimtable}
  \begin{center}
\begin{tabular}{|c|c|c|}
$\pi$&cubic&$\Gamma_\pi$\\\hline
\hline
  8a& \begin{tikzpicture}[baseline={([yshift=-.8ex]current bounding box.center)},scale=0.2]

\draw  [line width=0.75] [line join = round][line cap = round] (0,0) .. controls  (-0.5,0.4) and (-1.3,1.1) .. (-1.9,1.5) .. controls (-2.2,1.7) and (-2.8,1.5) .. (-2.85,1.1) .. controls (-2.9,0.9) and (-2.7,0.55) .. (-2.5,0.5) .. controls (-2,0.4) and (-1,0.8) .. (-0.6,0.9) .. controls (-0.5,0.94) and (0.103,1.0695) .. (0.1,1.15) ;
\end{tikzpicture}  & $B^3+C^3-ABC$ \\
\hline 
  8b& \begin{tikzpicture}[baseline={([yshift=-.8ex]current bounding box.center)},scale=0.4]
\draw (0,0) circle (0.3);
\draw    (0.7,0.8) .. controls  (0,-0.3) and (1.2,0.2) .. (0.4,-0.8);
\end{tikzpicture}   & $\lambda^2(A^3+B^3+C^3)-(\lambda^3+4)ABC$\\
  \hline 
  8c& \begin{tikzpicture}[baseline={([yshift=-.8ex]current bounding box.center)},scale=0.3]
\draw  [line width=0.75] [line join = round][line cap = round] (0,0) .. controls  (-0.5,0.4) and (-1.3,1.1) .. (-1.9,1.5) .. controls (-2.2,1.7) and (-2.8,1.5) .. (-2.85,1.1) .. controls (-2.9,0.9) and (-2.7,0.55) .. (-2.5,0.5) .. controls (-2,0.4) and (-1,0.8) .. (-0.6,0.9) .. controls (-0.5,0.94) and (0.103,1.0695) .. (0.1,1.15) ;
\draw (-0.9,0.8) circle (0.3);
\end{tikzpicture} & $B^3+C^3-ABC$\\
\hline
  7a& \begin{tikzpicture}[baseline={([yshift=-.8ex]current bounding box.center)},scale=0.2]
\draw (0,0) circle (1);
\draw (-1.5,0) -- (1.5,0);
\end{tikzpicture}  & $A(A^2-BC)$\\
\hline 
  7b& \begin{tikzpicture}[baseline={([yshift=-.8ex]current bounding box.center)},scale=0.3]
\draw  [color={rgb, 255:red, 0; green, 0; blue, 0 }  ][line width=0.75] [line join = round][line cap = round] (0,1.3) .. controls (0.5,1.3) and (0.95,1) .. (1.2,0.75) .. controls (1.4,0.6) and (1.5,0.4) .. (1.7,0.15) ;
\draw  [color={rgb, 255:red, 0; green, 0; blue, 0 }  ][line width=0.75] [line join = round][line cap = round] (0,1.3) .. controls (0.5,1.3) and (0.95,1.6) .. (1.2,1.85) .. controls (1.4,2) and (1.5,2.2) .. (1.7,2.55) ;
\draw (0,1.3) circle (0.3);
\end{tikzpicture}  & $A^3+B^2C$\\
\hline 
  7c& \begin{tikzpicture}[baseline={([yshift=-.8ex]current bounding box.center)},scale=0.2]
\draw (0,0) circle (1);
\draw (-1.5,0) -- (1.5,0);
\draw (-1,0) circle (0.3);
\draw (1,0) circle (0.3);
\end{tikzpicture}  & $A(A^2-BC)$\\
\hline 
  6a&  \begin{tikzpicture}[baseline={([yshift=-.8ex]current bounding box.center)},scale=0.2]
\draw (-0.4,0) -- (3.4,0);
\draw (-0.2,-0.4) -- (1.7,3.4);
\draw (3.2,-0.4) -- (1.3,3.4);
\end{tikzpicture} & $ABC$\\
\hline 
  6b&  \begin{tikzpicture}[baseline={([yshift=-.8ex]current bounding box.center)},scale=0.2]
\draw (0,0) circle (1);
\draw (-1.5,-1) -- (1.5,-1);
\draw (0,-1) circle(0.3);
\end{tikzpicture} & $B(AB+C^2)$\\
\hline 
  6c& \begin{tikzpicture}[baseline={([yshift=-.8ex]current bounding box.center)},scale=0.2]
\draw (0,0) circle (0.5);
\draw (1,2) -- (1,-2);
\draw (2,0) circle(0.5);
\draw (-1,0.2) -- (3,0.2);
\draw (-1,-0.2) -- (3,-0.2);
\end{tikzpicture}  & $BC^2$\\
\hline 
  6d& \begin{tikzpicture}[baseline={([yshift=-.8ex]current bounding box.center)},scale=0.2]
\draw (-0.4,0) -- (3.4,0);
\draw (-0.2,-0.4) -- (1.7,3.4);
\draw (3.2,-0.4) -- (1.3,3.4);
\draw (0,0) circle (0.3);
\draw (3,0) circle (0.3);
\draw (1.5,3) circle (0.3);
\end{tikzpicture}  & $ABC$\\
\hline 
  5a& \begin{tikzpicture}[baseline={([yshift=-.8ex]current bounding box.center)},scale=0.2]
\draw (1,2) -- (1,-2);
\draw (2,0) circle(0.5);
\draw (-1,0.2) -- (3,0.2);
\draw (-1,-0.2) -- (3,-0.2);
\end{tikzpicture}  & $A^2(A+C)$\\
\hline 
  5b& \begin{tikzpicture}[baseline={([yshift=-.8ex]current bounding box.center)},scale=0.2]
\draw (0,0) circle (0.5);
\draw (0,2) -- (0,-2);
\draw (2,0) circle(0.5);
\draw (-1,0) -- (3,0);
\end{tikzpicture}  & $BC^2$\\
\hline 
  4 & \begin{tikzpicture}[baseline={([yshift=-.8ex]current bounding box.center)},scale=0.3]
\draw (0,0) -- (2,0);
\draw (0,-0.2) -- (2,-0.2);
\draw (0,0.2) -- (2,0.2);
\draw (1,0) circle (0.3);
\end{tikzpicture}  & $A^3$\\
\hline 
\end{tabular}
\end{center}
  \end{table}
              
      \par\noindent{\underline{Remark}} Considerations we have not set out in this section and some easy calculations of dimensions of orbits permits us to present the following Table~\ref{cubicfig} of elementary specializations in the space of cubics. We give some of the specializations\footnote{Added in translation}.
\par
(a) Smooth cubic  ${\sf j}$ fixed, to the cusp.
We use a Weierstrass presentation 
\begin{equation}\label{W1eqn} Y^2Z =X (X - Z) (X - a Z),
\end{equation} 
and introduce new variables $H$ and $K$ in such a way  that :
$Z = b^2 H  , Y =b^{-1}K $. Then Equation \eqref{W1eqn} becomes  $ K^2 H =X (X- b^2H)(X- ab^2 H)$;
When $b =0$ , we get  $K^2 H =X^3$  which is an equation of the cubic with cusp.
\par
(b)  Singular cubic with a node \begin{tikzpicture}[baseline={([yshift=-.8ex]current bounding box.center)},scale=0.2]
\draw  [line width=0.75] [line join = round][line cap = round] (0,0) .. controls  (-0.5,0.4) and (-1.3,1.1) .. (-1.9,1.5) .. controls (-2.2,1.7) and (-2.8,1.5) .. (-2.85,1.1) .. controls (-2.9,0.9) and (-2.7,0.55) .. (-2.5,0.5) .. controls (-2,0.4) and (-1,0.8) .. (-0.6,0.9) .. controls (-0.5,0.94) and (0.103,1.0695) .. (0.1,1.15) ;
\end{tikzpicture}  to the cusp \begin{tikzpicture}[baseline={([yshift=-.8ex]current bounding box.center)},scale=0.2]
\draw  [color={rgb, 255:red, 0; green, 0; blue, 0 }  ][line width=0.75] [line join = round][line cap = round] (0,1.3) .. controls (0.5,1.3) and (0.95,1) .. (1.2,0.75) .. controls (1.4,0.6) and (1.5,0.4) .. (1.7,0.15) ;
\draw  [color={rgb, 255:red, 0; green, 0; blue, 0 }  ][line width=0.75] [line join = round][line cap = round] (0,1.3) .. controls (0.5,1.3) and (0.95,1.6) .. (1.2,1.85) .. controls (1.4,2) and (1.5,2.2) .. (1.7,2.55) ;
\end{tikzpicture}.
We have $X^3+Y(Y+aX)Z$ as an equation for the node; letting $a\rightarrow 0$, one gets $X^3+Y^2Z$, which is an equation of a cusp.
\par
(c) Singular cubic with node to an union of a smooth conic and a transverse line, then to a smooth conic and tangent line.
$ZXY+ aX^3+X^2Y + Y^3$ gives for $a=0$,
$Y( ZX+X^2+Y^2)$. 
Then $X^2+Y^2-Z^2(Y -b Z)$ gives for  $b=1,
(X^2+Y^2-Z^2)(Y - Z) $. 
\par
(d)  Cusp to a union of a smooth conic and a tangent line, then to a triangle.
$aX^3 + YX^2+Y^2Z$  gives when $a=0$,  $Y(X^2+YZ)$. Then
$Y(X^2 +aY^2-Z^2)$ gives for $a=0$, $Y(X^2-Z^2)$.\par
(e) The union of a smooth conic and a tangent line to three concurrent lines:
$(X^2-Y^2-aYZ)Y$ when $a=0$ gives $(X^2-Y^2)Y$.\par
(f) The triangle to the  union of three concurrent lines to the union of a double line and a transverse line, to a triple line are clear.\vskip 0.2cm
      \par The arguments above show there is no ${\sf j}$ constant specialization from a family of smooth conics to the cubic $XYZ$: this and the $\sf j$-constant specialization from smooth cubics to a cusp shows the following. 
 \begin{lemma}\label{nospecialXYZ}
  An algebraic family of cubics isomorphic to the cusp $X^3+Y^2Z$, or to the smooth Fermat cubic $X^3+Y^3+Z^3$ cannot have as limit (specialization) the cubic $XYZ$.
  \end{lemma}
  \begin{proof}[Second proof]\footnote{Added in translation.}\label{contraction} In the Macaulay duality (see \cite{Em}), using the contraction pairing from $S={\sf k}[x,y,z]$ to $R={\sf k}[X,Y,Z]$ a graded Gorenstein algebra quotient $A=S/I_F, I_F=\mathrm{Ann} F$ of Hilbert function $H(A)=(1,3,3,1)$ is determined by a degree-three form $F\in R$. For the Fermat cubic and the cusp one has, respectively $I_F=(xy,xy,yz,x^3-y^3, x^3-z^3)$, or $I_F=(xy,xz,z^2,y^3,x^3-y^2z)$, each with five generators. 
  The annihilating ideal for $XYZ$ is the complete intersection $(x^2,y^2,z^2)$.  But the number of generators of an 
  ideal in a finite length $n$ (here $n=8$) Artinian quotient of $S$ is semicontinuous: it is not possible to specialize from an ideal with five generators to a complete intersection. This completes the proof.
  \end{proof}
  \par    
           \begin{table}[hp]
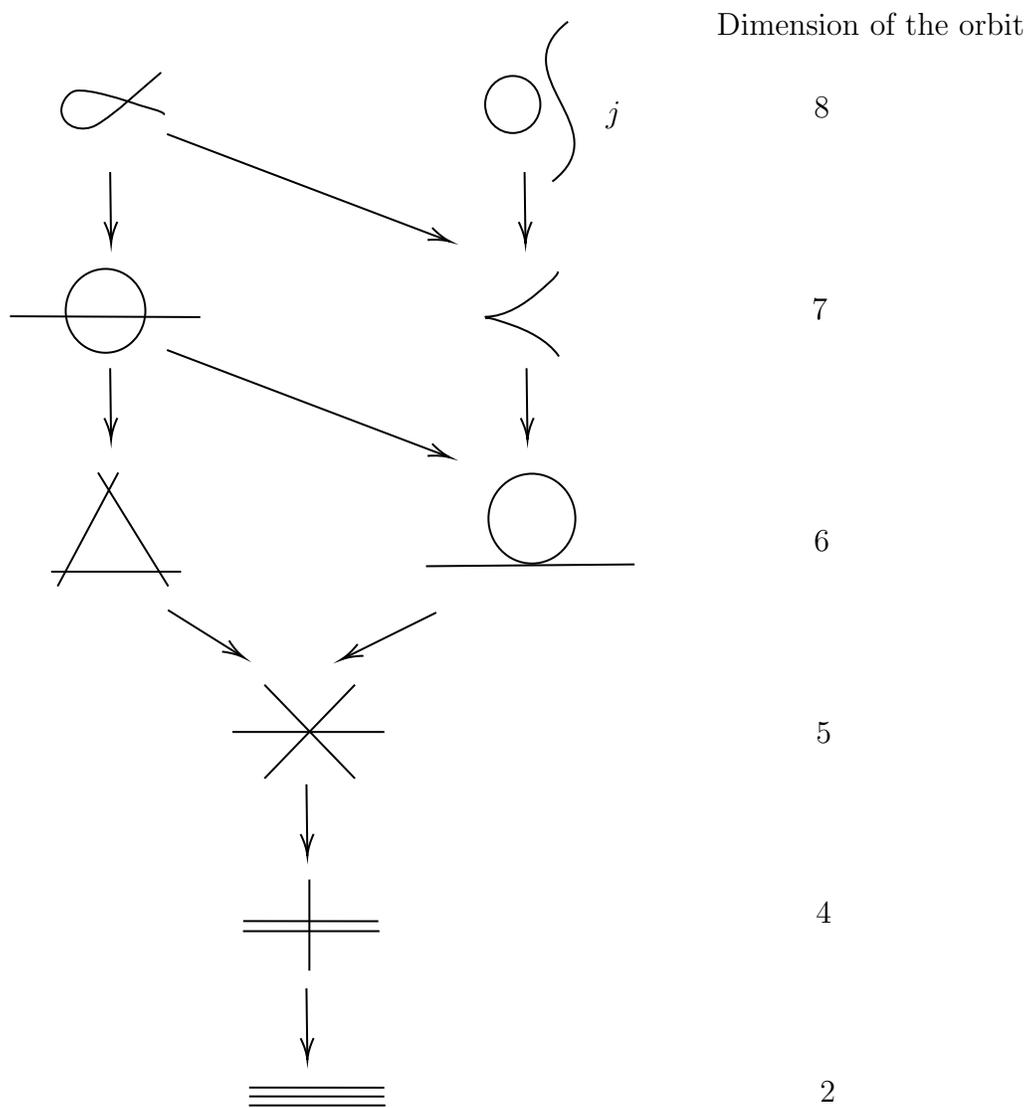
           \begin{center}
           \caption{Specializations of planar cubics}\label{cubicfig}
\vskip 0.3cm
\tikzset{every picture/.style={line width=0.75pt}} 

\begin{tikzpicture}[x=0.75pt,y=0.75pt,yscale=-1,xscale=1]

\draw  [color={rgb, 255:red, 0; green, 0; blue, 0 }  ][line width=0.75] [line join = round][line cap = round] (125.95,53.89) .. controls (116.92,61.36) and (103.34,74.53) .. (93.52,80.18) .. controls (86.96,83.95) and (77.25,81.67) .. (75.98,73.83) .. controls (75.31,69.68) and (78.97,63.71) .. (83,62.95) .. controls (89.87,61.66) and (108.2,67.71) .. (115.43,70.21) .. controls (117.67,70.98) and (127.7,73.29) .. (127.7,74.74) ;
\draw    (100.53,103.92) -- (100.97,138) ;
\draw [shift={(101,140)}, rotate = 269.26] [color={rgb, 255:red, 0; green, 0; blue, 0 }  ][line width=0.75]    (10.93,-3.29) .. controls (6.95,-1.4) and (3.31,-0.3) .. (0,0) .. controls (3.31,0.3) and (6.95,1.4) .. (10.93,3.29)   ;
\draw   (78.11,173.95) .. controls (78.11,162.26) and (87.11,152.79) .. (98.2,152.79) .. controls (109.29,152.79) and (118.28,162.26) .. (118.28,173.95) .. controls (118.28,185.64) and (109.29,195.12) .. (98.2,195.12) .. controls (87.11,195.12) and (78.11,185.64) .. (78.11,173.95) -- cycle ;
\draw    (50,176.7) -- (145.99,177.13) ;
\draw    (94.42,255.61) -- (129.85,313) ;
\draw    (70.79,305.62) -- (136.29,305.62) ;
\draw    (74.01,313) -- (104.61,255.61) ;
\draw   (289.65,69.81) .. controls (289.65,61.86) and (295.88,55.41) .. (303.57,55.41) .. controls (311.26,55.41) and (317.49,61.86) .. (317.49,69.81) .. controls (317.49,77.77) and (311.26,84.21) .. (303.57,84.21) .. controls (295.88,84.21) and (289.65,77.77) .. (289.65,69.81) -- cycle ;
\draw    (323.54,108.7) .. controls (358.61,81.49) and (296.37,55.2) .. (331.43,28) ;
\draw  [color={rgb, 255:red, 0; green, 0; blue, 0 }  ][line width=0.75] [line join = round][line cap = round] (289.78,176.97) .. controls (300.05,176.97) and (310.26,169.91) .. (315.75,165.16) .. controls (318.78,162.54) and (326.3,156.44) .. (326.3,154.48) ;
\draw   (291.32,278.83) .. controls (291.32,266.31) and (301.13,256.17) .. (313.24,256.17) .. controls (325.34,256.17) and (335.15,266.31) .. (335.15,278.83) .. controls (335.15,291.35) and (325.34,301.5) .. (313.24,301.5) .. controls (301.13,301.5) and (291.32,291.35) .. (291.32,278.83) -- cycle ;
\draw    (259.8,302.86) -- (365,301.95) ;
\draw    (129.17,193.72) -- (270.13,247.29) ;
\draw [shift={(272,248)}, rotate = 200.81] [color={rgb, 255:red, 0; green, 0; blue, 0 }  ][line width=0.75]    (10.93,-3.29) .. controls (6.95,-1.4) and (3.31,-0.3) .. (0,0) .. controls (3.31,0.3) and (6.95,1.4) .. (10.93,3.29)   ;
\draw    (201,461) -- (201,507) ;
\draw    (167.52,487.07) -- (236.35,487.07) ;
\draw    (167.62,481.98) -- (235.78,482.05) ;
\draw  [color={rgb, 255:red, 0; green, 0; blue, 0 }  ][line width=0.75] [line join = round][line cap = round] (326.67,196.68) .. controls (321.06,188.59) and (310.26,183.61) .. (303.75,181.36) .. controls (300.16,180.12) and (291.54,176.85) .. (290.09,177.7) ;
\draw    (162.17,386.52) -- (238.81,386.52) ;
\draw    (178.36,362.7) -- (223.98,410) ;
\draw    (223.98,362.7) -- (178.36,410) ;
\draw    (129.7,325) -- (166.31,347.94) ;
\draw [shift={(168,349)}, rotate = 212.07] [color={rgb, 255:red, 0; green, 0; blue, 0 }  ][line width=0.75]    (10.93,-3.29) .. controls (6.95,-1.4) and (3.31,-0.3) .. (0,0) .. controls (3.31,0.3) and (6.95,1.4) .. (10.93,3.29)   ;
\draw    (265,326.25) -- (218.79,349.11) ;
\draw [shift={(217,350)}, rotate = 333.66999999999996] [color={rgb, 255:red, 0; green, 0; blue, 0 }  ][line width=0.75]    (10.93,-3.29) .. controls (6.95,-1.4) and (3.31,-0.3) .. (0,0) .. controls (3.31,0.3) and (6.95,1.4) .. (10.93,3.29)   ;
\draw    (170.62,566) -- (238.78,566.06) ;
\draw    (170.52,575) -- (239.35,575) ;
\draw    (170.62,570.46) -- (238.78,570.52) ;
\draw    (100.53,202.92) -- (100.97,237) ;
\draw [shift={(101,239)}, rotate = 269.26] [color={rgb, 255:red, 0; green, 0; blue, 0 }  ][line width=0.75]    (10.93,-3.29) .. controls (6.95,-1.4) and (3.31,-0.3) .. (0,0) .. controls (3.31,0.3) and (6.95,1.4) .. (10.93,3.29)   ;
\draw    (310.53,202.92) -- (310.97,237) ;
\draw [shift={(311,239)}, rotate = 269.26] [color={rgb, 255:red, 0; green, 0; blue, 0 }  ][line width=0.75]    (10.93,-3.29) .. controls (6.95,-1.4) and (3.31,-0.3) .. (0,0) .. controls (3.31,0.3) and (6.95,1.4) .. (10.93,3.29)   ;
\draw    (309.53,103.92) -- (309.97,138) ;
\draw [shift={(310,140)}, rotate = 269.26] [color={rgb, 255:red, 0; green, 0; blue, 0 }  ][line width=0.75]    (10.93,-3.29) .. controls (6.95,-1.4) and (3.31,-0.3) .. (0,0) .. controls (3.31,0.3) and (6.95,1.4) .. (10.93,3.29)   ;
\draw    (199.53,412.92) -- (199.97,447) ;
\draw [shift={(200,449)}, rotate = 269.26] [color={rgb, 255:red, 0; green, 0; blue, 0 }  ][line width=0.75]    (10.93,-3.29) .. controls (6.95,-1.4) and (3.31,-0.3) .. (0,0) .. controls (3.31,0.3) and (6.95,1.4) .. (10.93,3.29)   ;
\draw    (199.53,515.92) -- (199.97,550) ;
\draw [shift={(200,552)}, rotate = 269.26] [color={rgb, 255:red, 0; green, 0; blue, 0 }  ][line width=0.75]    (10.93,-3.29) .. controls (6.95,-1.4) and (3.31,-0.3) .. (0,0) .. controls (3.31,0.3) and (6.95,1.4) .. (10.93,3.29)   ;
\draw    (129.17,84.72) -- (174.52,101.96) -- (270.13,138.29) ;
\draw [shift={(272,139)}, rotate = 200.81] [color={rgb, 255:red, 0; green, 0; blue, 0 }  ][line width=0.75]    (10.93,-3.29) .. controls (6.95,-1.4) and (3.31,-0.3) .. (0,0) .. controls (3.31,0.3) and (6.95,1.4) .. (10.93,3.29)   ;

\draw (349.19,66.85) node [anchor=north west][inner sep=0.75pt]    {$j$};
\draw (454,64.4) node [anchor=north west][inner sep=0.75pt]    {$8$};
\draw (453,166.4) node [anchor=north west][inner sep=0.75pt]    {$7$};
\draw (454,283.4) node [anchor=north west][inner sep=0.75pt]    {$6$};
\draw (455,380) node [anchor=north west][inner sep=0.75pt]   [align=left] {5};
\draw (455,471) node [anchor=north west][inner sep=0.75pt]   [align=left] {4};
\draw (457,561.4) node [anchor=north west][inner sep=0.75pt]    {$2$};
\draw (405,22) node [anchor=north west][inner sep=0.75pt]   [align=left] {Dimension of the orbit};

\end{tikzpicture}
           \end{center}
           \end{table}
           \vskip 0.3cm
           \newpage
           \subsubsection{Polar nets $J_\phi$ of singular cubics $\phi$}\label{polarsec}
\begin{transl}\label{4.5trlnote} (See \cite{Co}\footnote{{\it Note of translators}: we have added this section as a clarification.}).We give the list of those nets that are polars - Jacobian planes - of cubics. We use the notation for classes of nets in Table \ref{table1} where the number is the dimension of the orbit, and the letter is position from the left of the Table. The first class is the polars studied above of the family of smooth cubics \#8b not isomorphic to $X^3+Y^3+Z^3$.\par
There are six discrete classes, the polar nets $\pi=J_\phi$ of
\begin{enumerate}
\item A cubic having a double point with two distinct tangents,  \#8a, $X^3+Y^3+Z^3+3XYZ$.
\item A cubic comprised of a smooth conic and a transversal line,  \#7a, $X^3+3XYZ$.
\item  A cubic decomposed into three non-concurrent lines,  \#6a,
$XYZ$.
\item The cubic for  \#6d, $X^3+Y^3+Z^3$.
\item The cubic for \#5b, $XY^2+Z^3$, a cusp.
\item  A smooth conic union a tangent line   \#4, $X^2Y+Y^2Z$.
\end{enumerate}
Of the nine discrete classes of cubics (Table \ref{cubicfig}) two have polars that are pencils, and one a unique conic. See also Appendix B, and \cite[p.1568]{Co}.  Note that a specialization of a cubic  $F_A$ to $F_B$ here does imply a specialization $A$ to $B$ of the related polar nets. 
\end{transl}
\begin{landscape}
\begin{table}[h]\caption{Polar nets}\label{polartable}
\begin{center}
\begin{tabular}{|c|c|c|c|c|}
\hline
Case      & $F$ & Jacobian $(F_X,F_Y,F_Z)$& $\det Hess(F)$ & $\Gamma_\pi$\\
\hline
8a      & $X^3+Y^3+3XYZ$ &$(X^2+YZ,Y^2+XZ,XY)$ &$X^3+Y^3-XYZ$ & $B^3+C^3-ABC$ \\
\hline
8b& $X^3+Y^3+Z^3+$&$(X^2+\lambda YZ,Y^2+\lambda XZ,$&$X^3+Y^3+Z^3$&$\lambda^2(A^3+B^3+C^3)$\\
& $3\lambda XYZ$&$Z^2+\lambda XY)$&$-5\lambda XYZ$&$-(\lambda^3+4) ABC$\\

\hline
7a & $X^3+3XYZ$&$(X^2+YZ,XZ,YZ)$&$3X^3-XYZ$ & $A^3-ABC$\\
\hline 
6a&$XYZ$&$(YZ,XZ,XY)$&$XYZ$&$ABC$\\
\hline 
6d&$X^3+Y^3+Z^3$&$(X^2,Y^2,Z^2)$&$XYZ $&$ABC$\\
\hline 
5b&$XY^2+Z^3$&$(Y^2,XY,Z^2)$&$Y^2Z$&$BC^2$ \\
\hline 
4&$X^2Y+Y^2Z$&$(XY,X^2+2YZ,Y^2)$&$Y^3$&$A^3$ \\
\hline 
\end{tabular}
\end{center}
\end{table}
\vskip 0.4cm
In Table \ref{polartable} we list $F$, the Jacobian or polar net. The Hessian of $F$ is related to but not the same as the $\Gamma_\pi$ given in the right column. See also A. Conca \cite[p. 1568]{Co}.
\end{landscape}

           \clearpage
           \subsection{Existence of specializations of nets of conics}\label{specialnetsec}\par          
          An elementary specialization will be one from an orbit of dimension $c$ to one of dimension $c-1$. In this paragraph, we will show 
           \begin{proposition}\label{elemprop} All the elementary specializations in Table \ref{table1} in fact do occur.
           \end{proposition}
   
           We show this by enumerating the remaining elementary specializations that we have not already considered.
           \begin{enumerate}[1)]
               \item Specializations  $(A,B)$ where $\dim A=8, \dim B=7.$
               \begin{enumerate}[a)]
               \item \begin{tikzpicture}[baseline={([yshift=-.8ex]current bounding box.center)},scale=0.2]
\draw  [line width=0.75] [line join = round][line cap = round] (0,0) .. controls  (-0.5,0.4) and (-1.3,1.1) .. (-1.9,1.5) .. controls (-2.2,1.7) and (-2.8,1.5) .. (-2.85,1.1) .. controls (-2.9,0.9) and (-2.7,0.55) .. (-2.5,0.5) .. controls (-2,0.4) and (-1,0.8) .. (-0.6,0.9) .. controls (-0.5,0.94) and (0.103,1.0695) .. (0.1,1.15) ;
\end{tikzpicture} $\rightarrow$ \begin{tikzpicture}[baseline={([yshift=-.8ex]current bounding box.center)},scale=0.2]
\draw (0,0) circle (1);
\draw (-1.5,0) -- (1.5,0);
\end{tikzpicture}. One considers the \#8a family $(XY,ZY-\mu X^2,ZX-Y^2)$ whose element for $\mu\neq 0$ belongs to \begin{tikzpicture}[baseline={([yshift=-.8ex]current bounding box.center)},scale=0.2]
\draw  [line width=0.75] [line join = round][line cap = round] (0,0) .. controls  (-0.5,0.4) and (-1.3,1.1) .. (-1.9,1.5) .. controls (-2.2,1.7) and (-2.8,1.5) .. (-2.85,1.1) .. controls (-2.9,0.9) and (-2.7,0.55) .. (-2.5,0.5) .. controls (-2,0.4) and (-1,0.8) .. (-0.6,0.9) .. controls (-0.5,0.94) and (0.103,1.0695) .. (0.1,1.15) ;
\end{tikzpicture}  and for $\mu = 0$ belongs to \begin{tikzpicture}[baseline={([yshift=-.8ex]current bounding box.center)},scale=0.2]
\draw (0,0) circle (1);
\draw (-1.5,0) -- (1.5,0);
\end{tikzpicture}.
              \item By duality, one has \begin{tikzpicture}[baseline={([yshift=-.8ex]current bounding box.center)},scale=0.2]
\draw  [line width=0.75] [line join = round][line cap = round] (0,0) .. controls  (-0.5,0.4) and (-1.3,1.1) .. (-1.9,1.5) .. controls (-2.2,1.7) and (-2.8,1.5) .. (-2.85,1.1) .. controls (-2.9,0.9) and (-2.7,0.55) .. (-2.5,0.5) .. controls (-2,0.4) and (-1,0.8) .. (-0.6,0.9) .. controls (-0.5,0.94) and (0.103,1.0695) .. (0.1,1.15) ;
\draw (-0.9,0.8) circle (0.3);
\end{tikzpicture} $\rightarrow$ \begin{tikzpicture}[baseline={([yshift=-.8ex]current bounding box.center)},scale=0.2]
\draw (0,0) circle (1);
\draw (-1.5,0) -- (1.5,0);
\draw (-1,0) circle (0.3);
\draw (1,0) circle (0.3);
\end{tikzpicture};
\item \begin{tikzpicture}[baseline={([yshift=-.8ex]current bounding box.center)},scale=0.4]
\draw (0,0) circle (0.4);
\draw    (0.7,0.8) .. controls  (0,-0.3) and (1.2,0.2) .. (0.4,-0.8);
\draw (0.7,-0.4) node [anchor=west]{$\lambda$};
\end{tikzpicture} $\rightarrow$ \begin{tikzpicture}[baseline={([yshift=-.8ex]current bounding box.center)},scale=0.2]
\draw  [color={rgb, 255:red, 0; green, 0; blue, 0 }  ][line width=0.75] [line join = round][line cap = round] (0,1.3) .. controls (0.5,1.3) and (0.95,1) .. (1.2,0.75) .. controls (1.4,0.6) and (1.5,0.4) .. (1.7,0.15) ;
\draw  [color={rgb, 255:red, 0; green, 0; blue, 0 }  ][line width=0.75] [line join = round][line cap = round] (0,1.3) .. controls (0.5,1.3) and (0.95,1.6) .. (1.2,1.85) .. controls (1.4,2) and (1.5,2.2) .. (1.7,2.55) ;
\draw (0,1.3) circle (0.4);
\end{tikzpicture} has already been verified in \S 6.5 (2) page \pageref{keyspecial}.

\end{enumerate}

\item Specializations $(A,B)$ with $\dim A=7, \dim B=6$.

        \par All the specializations envisioned exist. Using the duality, it suffices to verify that the following specializations exist:
          \par
         \begin{equation*} \begin{tikzpicture}[baseline={([yshift=-.8ex]current bounding box.center)},scale=0.3]
\draw (0,0) circle (1);
\draw (-1.5,0) -- (1.5,0);
\end{tikzpicture} \rightarrow \begin{tikzpicture}[baseline={([yshift=-.8ex]current bounding box.center)},scale=0.2]
\draw (-0.4,0) -- (3.4,0);
\draw (-0.2,-0.4) -- (1.7,3.4);
\draw (3.2,-0.4) -- (1.3,3.4);
\end{tikzpicture} 
\hspace{1cm}  (XY,ZY, ZX-\mu Y^2)\in \begin{cases} \begin{tikzpicture}[baseline={([yshift=-.8ex]current bounding box.center)},scale=0.15]
\draw (0,0) circle (1);
\draw (-1.5,0) -- (1.5,0);
\end{tikzpicture} &\text { if }\mu\not=0\\
          \begin{tikzpicture}[baseline={([yshift=-.8ex]current bounding box.center)},scale=0.1]
\draw (-0.4,0) -- (3.4,0);
\draw (-0.2,-0.4) -- (1.7,3.4);
\draw (3.2,-0.4) -- (1.3,3.4);
\end{tikzpicture} &\text { if } \mu=0
          \end{cases}
     \end{equation*}
     
     \begin{equation*}
     \begin{tikzpicture}[baseline={([yshift=-.8ex]current bounding box.center)},scale=0.3]
\draw (0,0) circle (1);
\draw (-1.5,0) -- (1.5,0);
\end{tikzpicture} \rightarrow \begin{tikzpicture}[baseline={([yshift=-.8ex]current bounding box.center)},scale=0.3]
\draw (0,0) circle (1);
\draw (-1.5,-1) -- (1.5,-1);
\draw (0,-1) circle (0.3);
\end{tikzpicture} \hspace{1cm}
       (X(X-Z),Y(Y-\mu X), YZ)\in 
         \begin{cases} \begin{tikzpicture}[baseline={([yshift=-.8ex]current bounding box.center)},scale=0.15]
\draw (0,0) circle (1);
\draw (-1.5,0) -- (1.5,0);
\end{tikzpicture} &\text { if }\mu\not=0\\
          \begin{tikzpicture}[baseline={([yshift=-.8ex]current bounding box.center)},scale=0.15]
\draw (0,0) circle (1);
\draw (-1.5,-1) -- (1.5,-1);
\draw (0,-1) circle (0.3);
\end{tikzpicture}  &\text { if } \mu=0
          \end{cases}
     \end{equation*}
     
      \begin{equation*}
      \begin{tikzpicture}[baseline={([yshift=-.8ex]current bounding box.center)},scale=0.3]
\draw  [color={rgb, 255:red, 0; green, 0; blue, 0 }  ][line width=0.75] [line join = round][line cap = round] (0,1.3) .. controls (0.5,1.3) and (0.95,1) .. (1.2,0.75) .. controls (1.4,0.6) and (1.5,0.4) .. (1.7,0.15) ;
\draw  [color={rgb, 255:red, 0; green, 0; blue, 0 }  ][line width=0.75] [line join = round][line cap = round] (0,1.3) .. controls (0.5,1.3) and (0.95,1.6) .. (1.2,1.85) .. controls (1.4,2) and (1.5,2.2) .. (1.7,2.55) ;
\draw (0,1.3) circle (0.3);
\end{tikzpicture} \rightarrow \begin{tikzpicture}[baseline={([yshift=-.8ex]current bounding box.center)},scale=0.3]
\draw (0,0) circle (1);
\draw (-1.5,-1) -- (1.5,-1);
\draw (0,-1) circle (0.3);
\end{tikzpicture} \hspace{1cm}
 (XY,Z^2,(Y+X)Z-\mu X^2)\in 
        \begin{cases}  \begin{tikzpicture}[baseline={([yshift=-.8ex]current bounding box.center)},scale=0.2]
\draw  [color={rgb, 255:red, 0; green, 0; blue, 0 }  ][line width=0.75] [line join = round][line cap = round] (0,1.3) .. controls (0.5,1.3) and (0.95,1) .. (1.2,0.75) .. controls (1.4,0.6) and (1.5,0.4) .. (1.7,0.15) ;
\draw  [color={rgb, 255:red, 0; green, 0; blue, 0 }  ][line width=0.75] [line join = round][line cap = round] (0,1.3) .. controls (0.5,1.3) and (0.95,1.6) .. (1.2,1.85) .. controls (1.4,2) and (1.5,2.2) .. (1.7,2.55) ;
\draw (0,1.3) circle (0.3);
\end{tikzpicture}  &\text { if }\mu\not=0\\
     \begin{tikzpicture}[baseline={([yshift=-.8ex]current bounding box.center)},scale=0.15]
\draw (0,0) circle (1);
\draw (-1.5,-1) -- (1.5,-1);
\draw (0,-1) circle (0.3);
\end{tikzpicture} &\text { if } \mu=0
          \end{cases}
     \end{equation*}
     
     \begin{equation*}
     \begin{tikzpicture}[baseline={([yshift=-.8ex]current bounding box.center)},scale=0.3]
\draw (0,0) circle (1);
\draw (-1.5,0) -- (1.5,0);
\draw (-1,0) circle (0.3);
\draw (1,0) circle (0.3);
\end{tikzpicture} \rightarrow \begin{tikzpicture}[baseline={([yshift=-.8ex]current bounding box.center)},scale=0.3]
\draw (0,0) circle (1);
\draw (-1.5,-1) -- (1.5,-1);
\draw (0,-1) circle (0.3);
\end{tikzpicture} \hspace{1cm}
       (X(X-Z),(Y-\mu X)^2, YZ)\in 
         \begin{cases} \begin{tikzpicture}[baseline={([yshift=-.8ex]current bounding box.center)},scale=0.15]
\draw (0,0) circle (1);
\draw (-1.5,0) -- (1.5,0);
\draw (-1,0) circle (0.3);
\draw (1,0) circle (0.3);
\end{tikzpicture} &\text { if }\mu\not=0\\
     \begin{tikzpicture}[baseline={([yshift=-.8ex]current bounding box.center)},scale=0.15]
\draw (0,0) circle (1);
\draw (-1.5,-1) -- (1.5,-1);
\draw (0,-1) circle (0.3);
\end{tikzpicture}  &\text { if } \mu=0
          \end{cases}
     \end{equation*}
     
     \item Specializations $(A,B)$ with $\dim A=6, \dim B=5$.
     All the specializations envisioned exist, and the demonstration reduces, by duality, to proving the existence of the following specializations.
     
      \begin{equation*} \begin{tikzpicture}[baseline={([yshift=-.8ex]current bounding box.center)},scale=0.2]
\draw (-0.4,0) -- (3.4,0);
\draw (-0.2,-0.4) -- (1.7,3.4);
\draw (3.2,-0.4) -- (1.3,3.4);
\end{tikzpicture} \rightarrow \begin{tikzpicture}[baseline={([yshift=-.8ex]current bounding box.center)},scale=0.3]
\draw (1,2) -- (1,-2);
\draw (2,0) circle(0.5);
\draw (-1,0.2) -- (3,0.2);
\draw (-1,-0.2) -- (3,-0.2);
\end{tikzpicture} \hspace{1cm}(X^2-aXZ,XY,YZ)\in \begin{cases} \begin{tikzpicture}[baseline={([yshift=-.8ex]current bounding box.center)},scale=0.15]
\draw (-0.4,0) -- (3.4,0);
\draw (-0.2,-0.4) -- (1.7,3.4);
\draw (3.2,-0.4) -- (1.3,3.4);
\end{tikzpicture} &\text { if }a\not=0\\
          \begin{tikzpicture}[baseline={([yshift=-.8ex]current bounding box.center)},scale=0.2]
\draw (1,2) -- (1,-2);
\draw (2,0) circle(0.5);
\draw (-1,0.2) -- (3,0.2);
\draw (-1,-0.2) -- (3,-0.2);
\end{tikzpicture}&\text { if } a=0
          \end{cases}
     \end{equation*}
     
     \begin{equation*}
     \begin{tikzpicture}[baseline={([yshift=-.8ex]current bounding box.center)},scale=0.3]
\draw (0,0) circle (1);
\draw (-1.5,-1) -- (1.5,-1);
\draw (0,-1) circle(0.3);
\end{tikzpicture} \rightarrow \begin{tikzpicture}[baseline={([yshift=-.8ex]current bounding box.center)},scale=0.3]
\draw (1,2) -- (1,-2);
\draw (2,0) circle(0.5);
\draw (-1,0.2) -- (3,0.2);
\draw (-1,-0.2) -- (3,-0.2);
\end{tikzpicture} \hspace{1cm}
       (Y^2,(X-aZ))(aX-Z),YZ)\in 
         \begin{cases} \begin{tikzpicture}[baseline={([yshift=-.8ex]current bounding box.center)},scale=0.2]
\draw (0,0) circle (1);
\draw (-1.5,-1) -- (1.5,-1);
\draw (0,-1) circle(0.3);
\end{tikzpicture}&\text { if }a\not=0\\
         \begin{tikzpicture}[baseline={([yshift=-.8ex]current bounding box.center)},scale=0.2]
\draw (1,2) -- (1,-2);
\draw (2,0) circle(0.5);
\draw (-1,0.2) -- (3,0.2);
\draw (-1,-0.2) -- (3,-0.2);
\end{tikzpicture} &\text { if } a=0
          \end{cases}
     \end{equation*}
     
      \begin{equation*}
      \begin{tikzpicture}[baseline={([yshift=-.8ex]current bounding box.center)},scale=0.3]
\draw (0,0) circle (0.5);
\draw (1,2) -- (1,-2);
\draw (2,0) circle(0.5);
\draw (-1,0.2) -- (3,0.2);
\draw (-1,-0.2) -- (3,-0.2);
\end{tikzpicture} \rightarrow \begin{tikzpicture}[baseline={([yshift=-.8ex]current bounding box.center)},scale=0.3]
\draw (1,2) -- (1,-2);
\draw (2,0) circle(0.5);
\draw (-1,0.2) -- (3,0.2);
\draw (-1,-0.2) -- (3,-0.2);
\end{tikzpicture} \hspace{1cm}
 (X^2,Y(X-aY),YZ)\in 
        \begin{cases}       \begin{tikzpicture}[baseline={([yshift=-.8ex]current bounding box.center)},scale=0.2]
\draw (0,0) circle (0.5);
\draw (1,2) -- (1,-2);
\draw (2,0) circle(0.5);
\draw (-1,0.2) -- (3,0.2);
\draw (-1,-0.2) -- (3,-0.2);
\end{tikzpicture} &\text { if }a\not=0\\
     \begin{tikzpicture}[baseline={([yshift=-.8ex]current bounding box.center)},scale=0.2]
\draw (1,2) -- (1,-2);
\draw (2,0) circle(0.5);
\draw (-1,0.2) -- (3,0.2);
\draw (-1,-0.2) -- (3,-0.2);
\end{tikzpicture} &\text { if } a=0
          \end{cases}
     \end{equation*}
     \item Specializations $(A,B)$ with $\dim A=5,$ $\dim B=4.$\par
     All the specializations envisioned exist. Using the duality, it suffices to show that the following example of specialization.\par
     
      \begin{equation*}
      \begin{tikzpicture}[baseline={([yshift=-.8ex]current bounding box.center)},scale=0.3]
\draw (1,2) -- (1,-2);
\draw (2,0) circle(0.5);
\draw (-1,0.2) -- (3,0.2);
\draw (-1,-0.2) -- (3,-0.2);
\end{tikzpicture} \rightarrow \begin{tikzpicture}[baseline={([yshift=-.8ex]current bounding box.center)},scale=0.4]
\draw (0,0) -- (3,0);
\draw (0,-0.2) -- (3,-0.2);
\draw (0,0.2) -- (3,0.2);
\draw (1.5,0) circle (0.5);
\end{tikzpicture} \hspace{1cm}
 (YZ-X^2,(Y-aX)X,(Y-aX)Y)= 
        \begin{cases} \begin{tikzpicture}[baseline={([yshift=-.8ex]current bounding box.center)},scale=0.2]
\draw (1,2) -- (1,-2);
\draw (2,0) circle(0.5);
\draw (-1,0.2) -- (3,0.2);
\draw (-1,-0.2) -- (3,-0.2);
\end{tikzpicture} &\text { if }a\not=0\\
     \begin{tikzpicture}[baseline={([yshift=-.8ex]current bounding box.center)},scale=0.27]
\draw (0,0) -- (3,0);
\draw (0,-0.2) -- (3,-0.2);
\draw (0,0.2) -- (3,0.2);
\draw (1.5,0) circle (0.5);
\end{tikzpicture} &\text { if } a=0
          \end{cases}
     \end{equation*}
     
 \item    Specializations $(A,B)$ with $\dim A=4,$ $\dim B=2.$\par
     All the specializations envisioned exist. Using the duality, it suffices to show that the following specialization exists.\par
     
      \begin{equation*}
       \begin{tikzpicture}[baseline={([yshift=-.8ex]current bounding box.center)},scale=0.4]
\draw (0,0) -- (3,0);
\draw (0,-0.2) -- (3,-0.2);
\draw (0,0.2) -- (3,0.2);
\draw (1.5,0) circle (0.5);
\end{tikzpicture} \rightarrow \begin{tikzpicture}[baseline={([yshift=-.8ex]current bounding box.center)},scale=0.4]
\draw (0,0) circle (0.5);
\draw[color=gray] (-0.5,-1) -- (0.5,1);
\draw[color=gray] (-0.3,-1.1) -- (0.7,0.9);
\draw[color=gray] (-0.1,-1.2) -- (0.9,0.8);
\draw[color=gray] (0.1,-1.3) -- (1.1,0.7);
\draw[color=gray] (-0.7,-0.9) -- (0.3,1.1);
\draw[color=gray] (-0.9,-0.8) -- (0.1,1.2);
\draw[color=gray] (-1.1,-0.7) -- (-0.1,1.3);
\end{tikzpicture} \hspace{1cm}
 (Y^2,YZ+aX^2,XY)= 
        \begin{cases} \begin{tikzpicture}[baseline={([yshift=-.8ex]current bounding box.center)},scale=0.27]
\draw (0,0) -- (3,0);
\draw (0,-0.2) -- (3,-0.2);
\draw (0,0.2) -- (3,0.2);
\draw (1.5,0) circle (0.5);
\end{tikzpicture}&\text { if }a\not=0\\
     \begin{tikzpicture}[baseline={([yshift=-.8ex]current bounding box.center)},scale=0.3]
\draw (0,0) circle (0.5);
\draw[color=gray] (-0.5,-1) -- (0.5,1);
\draw[color=gray] (-0.3,-1.1) -- (0.7,0.9);
\draw[color=gray] (-0.1,-1.2) -- (0.9,0.8);
\draw[color=gray] (0.1,-1.3) -- (1.1,0.7);
\draw[color=gray] (-0.7,-0.9) -- (0.3,1.1);
\draw[color=gray] (-0.9,-0.8) -- (0.1,1.2);
\draw[color=gray] (-1.1,-0.7) -- (-0.1,1.3);
\end{tikzpicture}  &\text { if } a=0
          \end{cases}
     \end{equation*}
     \end{enumerate}
\vskip 0.2cm\par
All the other specializations that may be envisioned are the products of the elementary specializations whose existence has been proven. Thus, we have the list of all the elementary specializations that occur.

\section{Deformations}\label{deformsec}
    
\underline{Introduction} The graded Artinian local algebras of Hilbert function $H(A)=(1,3,3)$ have been implicitly classified in Theorem \ref{1thm} (see the Introduction~\S\ref{introsec}). We here show
\begin{theorem}\label{smooththm} \begin{enumerate}[(1)]
\item (\S 7.1)
Artinian local algebras $A$ of Hilbert function $H(A)= (1,3,3)$ are smoothable. That is to say, $A$ can be deformed to a smooth algebra of dimension 7 (isomorphic to the product algebra ${\sf k}^7$).
\item (\S 7.2, Proposition \ref{1r2smoothprop})
Each Artinian algebra of Hilbert function $H(A)=(1,r,2)$ is smoothable.\footnote{Our convention is that a deformation of an algebra $A(0)$ to $B$ will be a flat family $\mathcal A\to T$ over a parameter variety $T$, such that the special fiber $\mathcal A(t_0)$ over the point $t_0$ of $T$ is $A(0)$ and the general fiber $\mathcal A(t), t\neq t_0$ is isomorphic to $B$.} 
\end{enumerate}
\end{theorem}
The second length 7 continuous family of algebras is those of Hilbert function 
$H(A)=(1,4,2)$.  To show smoothability, we study the more general case of algebras of type $(1, r, 2)$, for which the proof is not more difficult. To show Theorem \ref{smooththm} (2) we consider a general family with $r-3$ parameters of non-isomorphic algebras of type $(1, r, 2)$; we give generators and the relations between generators; and we show that for $r\geq 3$, such an algebra deforms into an algebra $A'={\sf k}\times A''$ where $A''$ is an algebra of type $(1, r-1,2).$ From which it follows, by induction on $r$, that $A$ may be deformed to the product algebra of $r+3$ copies of ${\sf k}$. 

\subsection{Deformation of algebras of Hilbert function $(1,3,3)$ into a smooth algebra}

An algebra of Hilbert function $(1,3,3)$ has general form ${\sf k} [X,Y,Z]/(P,Q,R,\m^3)$ where $P,Q,R$ are three linearly independent elements in $R_2$. The above and Table~1 proves that, generically, we can find a presentation of such an algebra in the form $${\sf k} [X,Y,Z] / (Y^2 + \lambda XZ, Z^2 +\lambda XY, X^2 + \lambda YZ) + \m^3$$ 
In addition, we have the equality $$(Y^2 + \lambda XZ, Z^2 +\lambda XY, X^2 + \lambda YZ) + \m^3= (Y^2 + \lambda XZ, Z^2 +\lambda XY, X^2 + \lambda YZ,XYZ),$$
 hence the presentation:
$${\sf k} [X,Y,Z] / (X^2 + \lambda YZ, Y^2 +\lambda XZ, Z^2 + \lambda XY,XYZ).$$
Consider then the family of algebras: 
$$A_t={\sf k} [X,Y,Z]/(Y^2+\lambda XZ,Z^2+\lambda XY,X^2+\lambda YZ+tX,XYZ+\lambda^2t(1+\lambda^3)^{-1}X^2).$$
For $t\neq 0$ it's the product of a local algebra of dimension 4 and three copies of ${\sf k}$. Indeed, the zeros of the ideal $$(Y^2+\lambda XZ,Z^2+\lambda XY,X^2+\lambda YZ+tX,XYZ+\lambda^2t(1+\lambda^3)^{-1}X^2)$$ are $$(0,0,0),\quad (-\frac{t}{1+\lambda^3},\frac{\lambda t}{1+\lambda^3},\frac{\lambda t}{1+\lambda^3}),$$  $$(-\frac{t}{1+\lambda^3},\frac{\lambda jt}{1+\lambda^3},\frac{\lambda j^2t}{1+\lambda^3}),\quad (-\frac{t}{1+\lambda^3},\frac{\lambda j^2t}{1+\lambda^3},\frac{\lambda jt}{1+\lambda^3}),$$
where $j$ is the cube root of unity; the localization of this algebra at $(0,0)$, which is isomorphic to ${\sf k} [X,Y]/(X^2,Y^2)$, has length 4; and the other localizations have length 1. 

We therefore have a flat deformation of $A_0$ into a family of algebras isomorphic to ${\sf k} [X,Y] / (X^2,Y^2)\times {\sf k}^3$. Since all Artinian algebras with two generators are smoothable (see \cite{Fog}), ${\sf k} [X,Y]/(X^2,Y^2)$ may be deformed into ${\sf k}^4$, and $A_0$ is finally deformable into ${\sf k}^7$. 
\begin{transl} That an algebra $A_0$ is smoothable, does not answer the question of whether it is \emph{alignable}: has a deformation to a curvilinear algebra:  one isomorphic to ${\sf k}[t]/(t^n)$. It is open whether these algebras of Hilbert function $H=(1,3,3,0)$ are alignable. Furthermore, there is at present no effective deformation theory giving obstructions to being alignable - aside of the dimension of the family. The alignable algebras of length $n$ in $r$ variables are a family of dimension $(r-1)(n-1)$, so for $r=3, n=7$ their dimension is 12, while the family of $H=(1,3,3,0)$ algebras are parametrized by the nets of conics, so have dimension 9: so here dimension is not an obstruction.\footnote{Codimension three complete intersection [CI] algebras are smoothable  but may or may not be alignable. It is open whether general graded $H=(1,3,3,1)$ algebras defined by a CI net of conics are alignable. For some discussion and still open problems about alignability and peelability (see below) of CI  local algebras in codimension three see \cite{I-Arcata}.}\par
 An algebra of length $n$ is \emph{peelable} if it can be smoothed in exactly 
$n-1$ steps by splitting off a regular point. Since all codimension two algebras are alignable 
by \cite{Bri} (characteristic zero, other authors have shown this for $\sf k$ algebraically closed in characteristic $p$) and alignable algebras are peelable, the above proof shows that all Artinian algebras of Hilbert function
$(1,r,2)$ are peelable over an algebraically closed field of characteristic zero, or characteristic $p$ not $2$.
\end{transl}
\subsection{Deformation of algebras of Hilbert function $(1,r,2)$ into a smooth algebra}\label{def1r2sec}

Such an algebra of Hilbert function $(1,r,2)$ is a quotient of ${\sf k}[X_1,\dots, X_r]$ determined by a subspace of $R_2$ having codimension 2. We will first study a generic family with $(r-3)$ parameters of non-isomorphic vector spaces having dimension 2, where $r\geq 2$.\footnote{These are pencils of quadrics, that correspond to Artinian algebras of Hilbert function $(1,r,\binom{r+1}{2}-2,0)$. The defining ideals of these algebras are generated in degrees two and three. 
The defining ideals of their dual algebras of Hilbert function $(1,r,2)$ are generated in degree two alone if the corresponding pencil is general enough.} By duality, we immediately deduce an analogous family of non-isomorphic quotients of ${\sf k} [[X_1, \dots X_r]]$ of Hilbert function $(1, r, 2)$. We give a presentation of the algebras of this family by specifying generators and the relations between these generators. Then we describe a deformation for which all these relations extend, so that this deformation is flat; the deformed algebra splits into a product ${\sf k}\times A"$ where $A"$ is an algebra of Hilbert function $(1,r-1,2).$ Since all the quotients of ${\sf k} [[X_1,X_2]]$ are smoothable, a reasoning by induction proves that the algebras $(1, r, 2)$ are smoothable.

We show that a general enough 2-dimensional subspace of $R_2$ (i.e. belonging to an open dense set of $\Grass(2, R_2)$) is isomorphic to a pencil\par $V[\lambda_4, \lambda_5,\dots,\lambda_r] = \langle f,g\rangle$ defined as follows $$f=X_1^2+X_3^2+\cdots+X_r^2; \quad g=X_2^2+X_3^2+\lambda_4X_4^2+\cdots \lambda_rX_r^2.$$
Furthermore, if the $\lambda_i$ are all distinct, $V[\lambda_4,\lambda_5,\dots,\lambda_r]\simeq V[\lambda'_4,\lambda'_5,\dots,\lambda'_r]$ is equivalent to $\{\lambda_4,\dots, \lambda_r\}=\{\lambda'_4,\dots, \lambda'_r\}.$

First, notice that we expect a space of orbits, of dimension $$\dim \Grass(2,R_2)-\dim \Pgl(r-1)=2\left[\binom{r+1}{2}-2\right]-[r^2-1]= r-3.$$
Assume therefore that: $\langle f,g\rangle \in \Grass (2, R_2)$. Let us look at the singular quadrics of the pencil defined by $\langle f,g\rangle$. Here $tf + (1-t) g$ corresponds to a singular quadric if its partial derivatives are simultaneously zero. That is, a determinant, which is a polynomial of degree $r$ in $t$, must be zero. This polynomial has in general $r$ distinct roots so that there are $r$ singular quadrics in the pencil. Let us make a change of coordinates so that singular points of two of them have for homogeneous coordinates $(1,0,\dots,0)$ and $(0,1,0,\dots,0)$, respectively. So the first quadric has for equation in general $$X_1^2+X_1h(X_3,\dots,X_r)+h'(X_3,\dots, X_r)$$
and after a change of coordinates only on $X_1$ $$X_1^2+f''(X_3,\dots,X_r).$$ 

In a similar way, a change of variable affecting only $X_2$ provides for the second quadric an equation $X_2^2+g''(X_3,\dots,X_r)$. 

If $r=3$, we get a normal form $\langle f,g\rangle=\langle X_1^2+X_3^2,X_2^2+X_3^2\rangle.$

If $r>3$, we use induction: there is a new basis of the space of the forms generated by $x_3,\dots, x_r$, such that $f''$ and $g''$ have simultaneous diagonalizations. We finally have the announced form where none of the $\lambda_i$'s is zero.

It is clear then, when the $\lambda_i$'s are distinct, that the only isomorphisms between the $V(\lambda_4,\dots,\lambda_r)$ consist in permuting the variables $X_4,\dots, X_r$. Thus we have obtained a family of pencils parametrized by the affine space $A_{r-3}$ such that:

\begin{enumerate}[a)]
    \item there exists an open in the space of pencils of quadrics such that every pencil of this open is isomorphic to a pencil of this family.
    \item The pencils of the family, isomorphic to a given pencil, form a set of cardinality less than or equal to $(r-3)!$. 
\end{enumerate}

The vector space dual to $V[\lambda_4,\dots, \lambda_r]$, namely $V'[\lambda_4,\dots,\lambda_r]$ (following the definition of Section 6.1) contains the vector space $W$ generated by all the monomials $X_iX_j$ where $i\neq j$ and a supplement of dimension $(r-2)$ of $W$ generated by $h,h_4,\dots, h_r$, where: $$h=X_3^2-X_1^2-X_2^2,\quad h_1=X_4^2-X_1^2-\lambda_4X^2_2,\quad \dots,$$
$$h_i=X_i^2-X_1^2-\lambda_i X_2^2,\dots h_r=X_r^2-X_1^2-\lambda_rX_2^2$$

Calling $(V')$ the ideal generated by $V'$, we consider the algebra $$A={\sf k} [[X_1,\dots,X_r]]/(V'),$$
of Hilbert function $H(A)=(1,r,2,0)$.
We will show that there exists a flat deformation of such an algebra to the family of algebras $$B_t={\sf k} [[X_1,\dots,X_r]]/(h_r+tX_r,h_{r-1},\dots, h_4,h).$$ 

In order to verify that this deformation is flat, we must list all the non-trivial relations between $h, h_4,\dots, h_r$ and show that these relations extend to the generators of the deformed algebra.\footnote{{\it Note of translators}: the extension of relations is a well-known criterion for flatness of a family, see \cite[Chap. I.2.11]{Bour}, \cite[Cor. A.11]{Ser}.}

First, we list the linear relations between the generators whose coefficients are in $R_1$. Those are:\vskip 0.2cm
\begin{tabular}{lr}
    $ e_{1i}: X_1h_i-X_1h-X_i(X_1X_i)+X_3(X_1X_3)+(1-\lambda_i)X_2(X_1X_2) $&  $4\leq i\leq r$\\
    $e_{2i}: X_2h_i-\lambda_i X_2h-X_i(X_2X_i)+\lambda_iX_3(X_2X_3)+(1-\lambda_i)X_1(X_1X_2)$ & $4\leq i\leq r$\\
   $e_{3i}: X_ih-X_3(X_3X_1)+X_1(X_1X_i)+X_2(X_2X_i)$ & $4\leq i\leq r$\\
   $e_{si}: X_sh_i+X_1(X_1X_s)+X_2(X_2X_1)-X_i(X_iX_s)$ \ $s\neq i, 4\leq i, 2\leq s$.&
\end{tabular}
\vskip 0.2cm
The linear relations between the monomial $X_iX_j,i\neq j$, are of the form: 
$$e_{kij};\quad X_k(X_iX_j)-X_j(X_kX_i),\quad i,j,k \text{ distinct}.$$
\normalsize

These relations correspond to the sets of two adjacent edges of a simplex (whose vertices are the $X_i$, the edges $X_iX_j$, $\dots$).

There is a relation between the $e_{kij}$ for each sub-simplex of dimension 2. Then, there are in total $\binom{r-1}{2}-\binom{r}{3}$ independent linear relations among the generators of $W$. There are $3(r-3)+(r-3)^2$ linear relations concerning $h$ and the $h_i$'s. Therefore, we have completed a list of $r\binom{r-1}{2}-\binom{r}{3}+r(r-3)=\frac{1}{3}(r^3-7r)$ linear relations that are clearly linearly independent. It is easy to verify that $R_3\subset I=(V')$ and to conclude that there are then $r$ (minimal number of generators of $I$)-$\dim R_3$ linear relations with coefficients in $R_1$, which is $r(\binom{r}{2}+r-2-\binom{r+2}{3})=(r^3-7r)/3$.  So we have found a basis for these relations. For more complete calculations the reader will refer to \cite{Em-I}).

We affirm that these relations and trivial relations span all relations. Suppose indeed that:
$$e = ah + a_4h_4+ \cdots + a_r h_r + u$$
is a homogeneous relation where $u$ is formed by a linear combination of monomials of $W$ having coefficients of degree $j$ in ${\sf k}[X_1,\ldots,X_r]$. According to the previous discussion $j\geq 2$. Using multiples of the first relations $e_{i3},e_{1i},e_{2i},$ we can eliminate all the terms in $h,h_1,\ldots,h_4$ except for $X_1^jh$, $X_2^jh$, $X_3^jh$, $X_4^jh_4$, $X_5^jh_5,\ldots,X_r^jh_r$. But these terms should be zeroes, since they cannot be compensated by any other term neither of $u$, nor a multiple of $h$, nor any of the $h_i$. Thus, we have reduced the relation to a relation between monomials in $W$. But it is trivial to verify that the $e_{kij}$ span all the relations among the monomials of $W$. 

Now we consider the following deformation: 
$$I(t)=(W,h,h_4,h_5,\ldots, h_{r-1},h_r+tX_r).$$ The relationships we have listed extend as follows:

These new linear relations are the same as the old ones up to the following modifications:

\begin{enumerate}[a)]
    \item In $e_{1r}$, we replace the coefficient $-X_r$ of $X_1X_r$ by $-(X_r+t)$. 
    \item In $e_{2r}$, we replace the coefficient $-X_r$ of $X_2X_r$ by $-(X_r+t)$.
    \item In the other $e_{sr}$, we replace the coefficient $-X_r$ of $X_sX_r$ by $-(X_r+t)$. 
    \end{enumerate}
    Since all the relations between the generators of $I$ extend to relations between the generators of $I(t)$, the given deformation is flat. Its effect at the origin is to add $X_r$ to the ideal I, thus providing, as localized at the origin, an algebra of Hilbert function $(1, r-1,2)$ with the same values $(\lambda_4, \dots, \lambda_r)$ as before. As the deformation is flat, and the dimension of the localization at the origin only decreases by one, a point must be separated from the origin with a localization of length 1. Thus the deformed algebra is isomorphic to a product $A(t)\simeq {\sf k}\times A''$. 
    
    We can continue the same procedure with $A''$, until $r=3$. Then, $A''$ is of Hilbert function $(1,2,2)$ and by an old result of R. Hartshorne and J. Fogarty (irreducibility of the Hilbert scheme of points in $\mathbb P^2$, see \cite{Fog}), $A''$ is smoothable. Therefore, all the algebras that we have considered are smoothable, and since they form a dense set in the set of algebras of Hilbert function $(1,r,2)$, we conclude:\footnote{{\it Note of translators}: this smoothability result has been significantly generalized in \cite[Theorem C]{CJN}.}
    
    \begin{proposition}\label{1r2smoothprop}
    All the algebras of Hilbert function $(1,r,2)$ are smoothable. 
    \end{proposition}
    
\section{Regular maps of degree $4$ from the projective plane into itself}\label{mapsec}
A regular map (morphism) $\lambda: \mathbb P^2\to \mathbb P^2$ can be written in homogeneous coordinates by a formula:
\begin{equation}\label{mapequation}
\lambda[(\overline{X,Y,Z})]=[\overline{P_n[X,Y,Z],Q_n[X,Y,Z],R_n[X,Y,Z]}],
\end{equation}
where $P_n,Q_n,R_n$ are three homogeneous polynomials of the same degree
$n$ without a common zero, and $(\overline{X,Y,Z})$ is the point $(X:Y:Z)$ in projective plane associated to the vector $(X,Y,Z)$.\par
Using Bezout's Theorem, one sees that $\lambda$ is of degree $n^2$. As a consequence, when $n=2$ we obtain a map of degree $4$.  We will say that two regular maps $\lambda_1: \mathbb P^2\to \mathbb P^2$ and $\lambda_2: \mathbb P^2\to \mathbb P^2$ are of the same type if there are two biregular maps
\begin{equation*}\alpha: \mathbb P^2\to\mathbb P^2 \text { and } \beta: \mathbb P^2\to\mathbb P^2,
\end{equation*}
such that $\beta\circ \lambda_2=\lambda_1\circ \alpha$. As the biregular maps from 
$\mathbb P^2$ to $\mathbb P^2$ are none other than the elements of $Pgl(3)$, it becomes evident that the classification of the regular applications of degree $4$ from $\mathbb P^2\to \mathbb P^2$ is equivalent to the classification of nets of planar conics that do not have a base point.\par
Finally, it is clear that these maps specialize as the corresponding nets. Using the Table \ref{lengthfig} of \S 6.3 and the specializations of Table \ref{table1}, we have then the classification of regular maps from $\mathbb P^2$ to $\mathbb P^2$ (Table \ref{mapfig}).\footnote{Note of Translators:  Correction from original: \#8a from our Table \ref{table1} has the base point $(0,0,1)$ so by Equation \eqref{mapequation}ff it does not lead to a regular map. See also \cite{Rong}.}

\begin{table}
\caption{Regular maps from $\mathbb P^2$ to $\mathbb P^2$.}\label{mapfig}
\vskip -0.8cm
\begin{align*}
&\\
&\qquad\qquad
\begin{tikzpicture}[baseline={([yshift=-.5ex]current bounding box.center)},scale=0.6]
\draw (0,0) circle (0.3);
\draw    (0.7,0.8) .. controls  (0,-0.3) and (1.2,0.2) .. (0.4,-0.8);
\draw (0.7,-0.4) node [anchor=west]{$\lambda$};
\end{tikzpicture}
\\
\ \quad (\overline{X,Y,Z})&\mapsto
 (\overline{X^2+\lambda YZ,Y^2+\lambda XZ,Z^2+\lambda XY})\\
 &\\
&\qquad\qquad\begin{tikzpicture}[scale=0.3]
\draw  [line width=0.75] [line join = round][line cap = round] (0,0) .. controls  (-0.5,0.4) and (-1.3,1.1) .. (-1.9,1.5) .. controls (-2.2,1.7) and (-2.8,1.5) .. (-2.85,1.1) .. controls (-2.9,0.9) and (-2.7,0.55) .. (-2.5,0.5) .. controls (-2,0.4) and (-1,0.8) .. (-0.6,0.9) .. controls (-0.5,0.94) and (0.103,1.0695) .. (0.1,1.15) ;
\draw (-0.9,0.8) circle (0.3);
\end{tikzpicture}\\
 &(\overline{X,Y,Z})\mapsto  (\overline{Z^2,X^2-YZ,Y^2-XZ})\\
 &\qquad\qquad \downarrow\\
 &\\
 &\qquad\qquad \begin{tikzpicture}[scale=0.23]
\draw (0,0) circle (1);
\draw (-1.5,0) -- (1.5,0);
\draw (-1,0) circle (0.35);
\draw (1,0) circle (0.35);
\end{tikzpicture}\\
 &(\overline{X,Y,Z})\mapsto  (\overline{X^2+YZ,Y^2,Z^2})\\
  &\qquad\qquad \downarrow\\
  &\\
  &\qquad\qquad \begin{tikzpicture}[baseline={([yshift=-.8ex]current bounding box.center)},scale=0.17]
\draw (-0.4,0) -- (3.4,0);
\draw (-0.2,-0.4) -- (1.7,3.4);
\draw (3.2,-0.4) -- (1.3,3.4);
\draw (0,0) circle (0.4);
\draw (3,0) circle (0.4);
\draw (1.5,3) circle (0.4);
\end{tikzpicture}\\
 &(\overline{X,Y,Z})\mapsto  (\overline{X^2,Y^2,Z^2})
\end{align*}
\end{table}
In what follows, we give canonical procedures for passing from a regular map of degree $4$, $\lambda:\mathbb P^2\to \mathbb P^2$ to a net of conics, and vice versa.\par
If $\lambda: \mathbb P^2_a\to \mathbb P^2_b$ is such a map, the net $R$ corresponding to it is obtained in the $\mathbb P^2_a$ source of $\lambda$ as inverse image of the net of all lines of $\mathbb P^2_b$.\par
The lines of $\mathbb P^2_b$ correspond to the points of $R$ and, by duality, the points of $\mathbb P^2_b$ to the lines of $R$.\par
We therefore have the reciprocal canonical process: \par
Let $P$ be a projective plane together with a net $\pi$ of conics without base point (we change notation to return to the notations of \S 4). We may define a map:
$$ \lambda: P\to \pi^\ast$$
in the following manner. If $x\in P$ the set of conics of $\pi$ that pass through $x$ form a pencil $\delta$. We set $\lambda(x)=\delta$. One verifies immediately that $\lambda$ is of degree $4$, and that the inverse image by $\lambda$ of the net of lines of $\pi^\ast$ (which identifies with $\pi$), is indeed the net of conics of $P$ belonging to $\pi$.\par
Further, in the case where $\pi$ has the generic form studied in \S 4.1.2 we have the following interpretation in terms of $\lambda$ of the cubics $\Gamma$ (defined in Section 4.0), and $H$ (defined in \S 4.1.3).
\begin{proposition}\label{mapprop}
\begin{enumerate}[1)]
\item $H$ is the singular locus of $\lambda$;
\item $\Gamma$ is the curve dual to the discriminant locus of $\lambda$.
\end{enumerate}
\end{proposition}
\begin{proof}
To show 1), it suffices to use the formula \eqref{mapequation} for $\lambda$ where
$$P_n(X,Y,Z)=X^2+\lambda YZ, Q_n(X,Y,Z)=Y^2+\lambda ZY, R_n(X,Y,Z)=Z^2+\lambda XY.$$

Let's show 2). A point $y$ of $\pi^\ast$ belongs to the discriminant locus $\Delta$ of $\lambda$ if and only if its inverse image by $\lambda$ has set-wise $3$ points instead of $4$. The points are the points of the base of the pencil $\delta$ that $y$ represents. Thus, to say that $\delta$ has only three base points, is to say also that
$\delta$ has setwise only $2$ singular conics instead of $3$. By the Theorem of Bezout, that is to say that $\delta$ is tangent to $\Gamma$. Thus $\Delta$ is the dual curve of $\Gamma$ and, dually, $\Gamma$ is the dual curve of the discriminant of $\lambda$.
\end{proof}\par
\addcontentsline{toc}{section}{Note on the Hessian form of a smooth cubic}
\section*{Note on the Hessian form of a smooth cubic}\label{Hessiansec}
Here we review some well known facts developing our earlier discussion of Hesse pencils in Propositions~\ref{4.4prop}, and \ref{Hesse2prop} in Section \ref{involutionsec}.\footnote{{\it Note of translators}: see the 2009 \cite{ArtDol}, and the 2022 \cite[\S 4]{CilOt}
for further exposition, development and references.}\par 
Let us consider the pencil $P$ generated by the degenerate cubic $XYZ$ (three lines) and the smooth cubic $X^3+Y^3+Z^3$.\par
The base points of the pencil are the points with homogeneous coordinates
(here $\omega$ is a primitive cube root of $1$)\vskip 0.2cm
\qquad $\begin{array}{ccc}
(0,1,-1)&(0,1,-\omega)&(0,1,-\omega^2)\\
(1,0,-1)&(1,0,-\omega)&(1,0,-\omega^2)\\
(1,-1,0)&(1,-\omega,0)&(1,-\omega^2,0).
\end{array}$
In order for a cubic to belong to the pencil $P$, it is necessary and sufficient that it pass through these 9 points.\par
These 9 points are the 9 inflection points of each smooth cubic of the pencil $P$. In effect, given such a smooth cubic, it and its Hessian generate $P$.\par
Further, this set $E$ of 9 points has the following property:\par
(A) Given two of them, there is a third point of $E$ that is collinear with them.\par
That is to say that $E$ enriched by the set of subsets of three collinear points has the structure of an affine plane over $\mathbb Z/3\mathbb Z$.\par
Conversely, given a set of nine points of $\mathbb P^2$ satisfying the property (A), there is a homogeneous system of coordinates of $\mathbb P^2$ such that the set of points is identified with $E$.\par
Thus, given a smooth cubic of $\mathbb P^2$, it is easy to see that its nine points of inflection satisfy (A). Consequently, the pencil generated by a smooth cubic and its Hessian is identified with the pencil $P$.\par
This is to say that for every smooth cubic there is a system of homogeneous coordinates for which it has the Hesse form:
\begin{equation*}
 (X^3+Y^3+Z^3+\lambda XYZ).
\end{equation*}
In $(\mathbb Z/3\mathbb Z)^2$ there are twelve lines that may be partitioned into 4 directions. The following table gives the lines with this partition:\vskip 0.2cm \par
\begin{align*}&\Bigg\{\begin{array}{cccc}
X=0&(0,1,-1)&(0,1,-\omega)&(0,1,-\omega^2)\\
Y=0&(1,0,-1)&(1,0,-\omega)&(1,0,-\omega^2)\\
Z=0&(1,-1,0)&(1,-\omega,0)&(1,-\omega^2,0)
\end{array}\\
&\Bigg\{\begin{array}{cccc}
X+Y+Z=0&(0,1,-1)&(1,0,-1)&(1,-1,0)\\
X+\omega Y+\omega^2Z=0&(0,1,-\omega^2)&(1,0,-\omega)&(1,\omega^2,0)\\
X+\omega^2Y+\omega Z=0&(0,1,-\omega)&(1,0,-\omega^2)&(1,-\omega,0)
\end{array}\\
&\Bigg\{\begin{array}{cccc}
\omega X+Y+Z=0&(0,1,-1)&(1,0,-\omega)&(1,-\omega,0)\\
X+\omega Y+Z=0&(0,1,-\omega)&(1,0,-1)&(1,-\omega^2,0)\\
X+Y+\omega Z=0&(0,1,-\omega^2)&(1,0,-\omega^2)&(1,-1,0)
\end{array}\\
&\Bigg\{\begin{array}{cccc}
\omega^2X+Y+Z=0&(0,1,-1)&(1,0,-\omega^2)&(1,-\omega^2,0)\\
X+\omega^2Y+Z=0&(0,1,-\omega^2)&(1,0,-1)&(1,-\omega,0)\\
X+Y+\omega^2Z=0&(0,1,-\omega)&(1,0,-\omega)&(1,-1,0)
\end{array}\\
\end{align*}
The cubics of the pencil decomposed into three lines correspond to the sets of three lines passing through the nine base points.  They correspond to the four directions of $(\mathbb Z/3\mathbb Z)^2$:\par
\vskip 0.2cm
\footnotesize
$\begin{array}{lrc}
&XYZ&=0\\
(X+Y+Z)(X+\omega Y+\omega^2Z)(X+\omega^2Y+\omega Z)&=X^3+Y^3+Z^3-3XYZ&=0\\
{\frac{1}{\omega}}(\omega X+Y+Z)(X+\omega Y+Z)(X+Y+\omega Z)&=X^3+Y^3+Z^3-3\omega XYZ&=0\\
{\frac{1}{\omega^2}}(\omega^2X+Y+Z)(X+\omega^2Y+Z)(X+Y+\omega^2Z)&=X^3+Y^3+Z^3-3\omega^2XYZ&=0\\
\end{array}$\vskip 0.2cm
\normalsize
These are the only singular cubics of the pencil $P$. One verifies them by finding the $\lambda$ such that $X^3+Y^3+Z^3+3\lambda XYZ=0$ has a singular point, that is, that the conics defined by the three partial derivatives have a common zero.\footnote{Note of translators: it is straightfoward to verify that each  of these singular nets have three common zeroes, so by Table \ref{lengthfig} correspond to \#6a in Figure \ref{table1}.}\par \label{singularHessepencil}
Now we find the condition on the cubic of the pencil $X^3+Y^3+Z^3+3\lambda XYZ$  for it to be isomorphic to the cubic $X^3+Y^3+Z^3$. (See Proposition \ref{4.2prop}). In this case the vector space spanned by the quadratic forms $X^2+\lambda YZ,Y^2+\lambda ZX,Z^2+\lambda XY$ is spanned by the partial derivatives of another cubic form $G$. We should have\vskip 0.2cm\par
$\begin{array}{rc}
G'_X&= a[X^2+\lambda YZ]+b[Y^2+\lambda ZX]+c[Z^2+\lambda XY]\\
G'_Y&= a'[X^2+\lambda YZ]+b'[Y^2+\lambda ZX]+c'[Z^2+\lambda XY]\\
G'_Z&= a''[X^2+\lambda YZ]+b''[Y^2+\lambda ZX]+c''[Z^2+\lambda XY]\\
\end{array}$\par\vskip 0.2cm
with:\par\vskip 0.2cm
$\begin{array}{rrrc}
G''_{XY}&=\lambda cX+2bY+\lambda aZ&=2a'X+\lambda c'Y+\lambda b'Z&=G''_{YX}\\
G''_{XZ}&=\lambda bX+\lambda aY+2cZ&=2a''X+\lambda c''Y+\lambda b''Z&=G''_{ZX}\\
G''_{YZ}& =\lambda b'X+\lambda a'Y+2c'Z&=\lambda c''X+2 b''Y+\lambda a''Z&=G''_{ZY}\\
\end{array}$\vskip 0.2cm\par
We may suppose $\lambda\neq 0$. One should have $b' =a=c'' $,
\begin{equation*}
c' =\frac{2b}{\lambda};\, a' =\frac{\lambda c}{2}; b''=\frac{2c}{\lambda};
a'' = \frac{\lambda b}{2}; b''= \frac{\lambda a'}{2}=\frac{\lambda^2 c}{4};\, a''= \frac{2 c'}{\lambda}=\frac{4b}{\lambda^2}.
\end{equation*}
From which $\frac{2c}{a}=\frac{\lambda^2 c}{4}$ and $\frac{\lambda b}{2}=\frac{4b}{\lambda^2}$, that is to say
\begin{equation*}
(\lambda^3-8)b=(\lambda^3-8)c=0.
\end{equation*}
If $\lambda^3\neq 8$, one has $b=c=a'' =b'' =a' =c' = 0, b' =a=c'' $, that is to say that $G$ is proportional to the original cubic.\par
The cubic $(X^3+Y^3+Z^3+3\lambda XYZ)$ is thus isomorphic to $(X^3+Y^3+Z^3)$ if and only if $\lambda^3=0$ or $8$, hence $\lambda=0,2, 2\omega, 2\omega^2$ (excluded values in \#8b of Table \ref{table1}).\par \label{excludedvalues}
The four cubics of the pencil isomorphic to $(X^3+Y^3+Z^3)$ correspond also to the four directions of the plane $(\mathbb Z/3\mathbb Z)^2$.\par
For each direction one finds these cubics by taking the sum of the cubes of the
linear forms defining each of the lines of the direction.\par
Here $X^3+Y^3+Z^3$ corresponds to $XYZ$, $X^3+Y^3+Z^3+6XYZ$ corresponds to $X^3+Y^3+Z^3-3XYZ$; $X^3+Y^3+Z^3+6\omega^2XYZ$ corresponds to $X^3+Y^3+Z^3-3\omega XYZ$, $X^3+Y^3+Z^3+6\omega XYZ$ corresponds to $X^3+Y^3+Z^3-3\omega^2XYZ$.\par
 One may describe this correspondence in another way, by saying that the second term of each correspondence is the Hessian of the first term. Finally, there is a third way to describe it.
\par
Associate to each cubic of the pencil, the vector space $V$ of its partial derivatives. Let's consider also the orthogonal $V^\perp$ to the space in the space of quadratic forms with scalar product defined in \S 6.1. This space $V^\perp$ may be obtained as the space of partial derivatives of a cubic in the pencil. If $X^3+Y^3+Z^3+3\lambda XYZ$ is the first cubic, this correspondence associates to it the cubic $X^3+Y^3+Z^3-\frac{6}{\lambda}XYZ$. In particular,
by this correspondence the cubics isomorphic to $X^3+Y^3+Z^3$ are associated in one-to-one manner to the cubics isomorphic to $XYZ$.
\par\newpage
\addcontentsline{toc}{section}{List of tables}
\listoftables
\addcontentsline{toc}{section}{Original References}
\renewcommand{\refname}{Original References}

\appendix
\addcontentsline{toc}{section}{Appendix A. Historical Note, pre 1977}

\subsection*{Appendix A. Historical note, pre 1977.}\vskip 0.2cm\noindent
{\it Classification of Pencils of Conics, pre-1977.} This is often attributed to 
Corrado Segre in his 1883 article \cite{Seg}: he classified the pencils of conics over the complexes in any number of variables.  The 1884 note \cite{Ros-Seg} on simultaneous invariants of two ternary conics is a letter from Corrado Segre to M.J. Rosanes concerning this classification in response to \cite{Ros}. However there were earlier works: the book of I. Dolgachev \cite[p. 468]{Dol} mentions the 1868 article of K. Weierstrass \cite{Wei} as giving the classification of pencils of conics up to isomorphism, in work partially dependent on a note by J. Sylvester of 1851. 
The article \cite[Thm. 3.1]{Bli} attributes determining the four simultaneous invariants of 2 ternary conics to the 1882 article of P. Gordon \cite{Gor}.
In any case by 1900 there was a substantial literature about pencils of conics in any $\mathbb P^r$ (see, for example, P. Newstead's MathSciNet review of \cite{Dim} and also the 1898-1904 Encyclopedia, which has many detailed historical remarks \cite{Mo}); J. Emsalem notes in the Outline \S \ref{outlinesec} above that he had learned the classification of pencils in a course of E. Riche.\vskip 0.2cm\noindent
{\it Nets of Planar conics before 1977.}\footnote{We report on these results, but we have not checked the proofs.}
 Unbeknowst to J. Emsalem and A.~Iarrobino in 1977, the classification up to isomorphism of nets of conics over the complexes had been attempted by Camille Jordan in 1906 (\cite{J}).\footnote{Until doing the present literature search, with convenient modern tools of MathSciNet, we were still unaware of Jordan's work. Although MSN does not list C.~Jordan's article, it does list the 1914 article of A. Wilson about the modular case \cite{Wil}, and later articles of A. Campbell, which cite it. I. Dolgachev and S. Kondo in \cite[Example 7.2.9]{DolKo} give the classification over characteristic ${\sf k}\neq 2$, attributing it to C. Jordan.}  C.~Jordan's article in principle considers an $m$-dimensional vector space of $n$-ary forms: after determining the eight ternary pencils, he goes on in \S12-24 to handle nets of ternary conics, arriving at a table of 13 classes in \cite[\S 25 p. 427]{J}, omitting one class (our \#6c), and apparently combining the continuous class \#8b with \#8c.\footnote{We note our classes in Table \ref{table1} \#8a,\#8b,\#8c, \ldots \#2b where \# denotes the dimension and a,b,$\ldots$ the position from left. Then the correspondence of C Jordan's classes I-XIII with ours is \#8a III; \#8b,c, I; \#7a XII; \#7b II; \#7c IV:\#6a XIII; \#6b VIII; \#6c missing; \# 6d V; \# 5a IX, \#5b VI; \#4 X; \#2a XI, \#2b, VII.}
 He continues to treat dimension four and five spaces of ternary conics, without using duality. His methods combine geometry and algebra, use the related cubic we called $\Gamma$, but he does not formally use the notion of type of a triple as fundamental invariant, as we have in Theorem \ref{mainthm}B, and in our diagrams. The article of Jordan does not discuss specialization.\par
 In the modular case, the 1914 article of the L.E. Dickson student, A.~Wilson, classifies nets of conics over $GF(p^n)$ when $p$ is odd, but he leaves the ``rank one'' case where the net contains a double line incomplete \cite{Wil}.\footnote{A. Wilson does not explicitly state that  $p$ is odd - apparently considered obvious, as a key method is to complete the square! A careful analysis, correction, and completion of A.~Wilson's work in this $p$ odd case is made in \cite{LPS}, see Appendix B.} The 1928 article of A. Campbell \cite{Cam2} classifies these nets over the field $GF(2^n)$, giving 34 classes; however, C. Zanella in 2012 exhibits classes of nets over any $GF(q)$ containing only nonsingular conics, contradicting a statement in \cite{Cam2} that these only exist when $p$ is odd \cite[\S 6]{Zan}. A. Campbell treated the real case in the 1928 article \cite{Cam}.\par
 A 1943 article of C. Ciamberlini \cite{Cia} discusses the simultaneous invariants of three ternary forms; a recent note of B. Blind \cite[\S 5]{Bli} presents these and other results concerning simultaneous invariants of $k$ ternary forms, giving new proofs using the Jordan algebra structure on ternary forms.\par
 The 1898-1904 Encyclopedia \cite[\S 32ff, p.224ff]{Mo} reviews nets of planar conics, beginning with the Hesse form \S 32; however, the results reported are not without errors.\par
 A 1957 article of V\'{a}clav Alda gives a lucid presentation of the 6 equivalence classes of nets that arise as Hessians, as well as 3 degenerate classes of dimension 2 or 1 \cite{Alda}. He corrects an error in \cite[\S 34, line 6 p. 228]{Mo}, which incorrectly states [our translation] ``Inversely, given three conics  $f=g=h=0$ not part of the same pencil, employing new variables allows us to represent them as partial derivatives of a cubic''; V. Alda gives the exact conditions on the net for which this
representation occurs \cite[p.50]{Alda}. 
\par\vskip 0.2cm\noindent
{\it Classification and specialization of planar cubics.} This classification over the reals had been begun by I. Newton, and continued by many. Experts assured us that this classification up to algebraic isomorphism, including specializations was known in the 19th century, although we have not found the diagram in our Figure 6  (see also \cite[Fig 1.13]{EH}). The book of G.~Salmon \cite[Chapter V, Section V, \S 197ff]{Sal} gives a classification, and \cite[p. 548-549]{Kl} has some discussion of specializations.  A later section of the same Encyclopedia \cite{Mo}, Chapter III.4 \S 29 p.316  by H.G. Zeuthen and M.~Pieri, includes reference to work of  S. N. Maillard and H. G. Zeuthen specifying the numbers of singular and degenerate cubics in the one
    parameter families defined by the conditions to pass through n general
    points and to be tangent to $8-n$ general lines for $0 \le n \le 8$.\footnote{We thank Steve Kleiman for this reference.}

\par\vskip 0.2cm\noindent
{\it Classification of Artinian algebras of small lengths before 1977}. This classification for $n\le 6$ was known to J. Emsalem and A. Iarrobino at the time of writing R\'{e}seaux, as indicated in footnote 3, p. \pageref{smalllength}). See the history of early work from 1880 to 1920 by D.~Happel \cite{Hap}. Tables of isomorphism class and specialization for $n\le 6$ for embedding dimension one and two are in J.~Brian\c{c}on's \cite{Bri} (see especially the figures in the actual dissertation of 1972). D.A.~Suprunenko had shown in 1956 that the length seven algebras (of Hilbert function $(1,4,2)$) have continuous isomorphism families \cite{Sup} (not difficult), and Z. M. Dyment had shown in 1966 that there are a finite number of classes for length six, listing 34 classes.  G.~Mazzola's 1980 article \cite{Maz} gave a complete classification of isomorphism classes of unitary algebras (not necessarily commutative) for $n\le 6$, over an algebraically closed field of characteristic not $2$ or $3$, using a variety of methods; Mazzola uses this and Hochschild cocyles to show the irreducibility of the family of local length $n$ commutative algebras for $n\le 7$, over any such field:  that is, he showed that each such algebra may be deformed to ${\sf k}[z]/(z^n).$\footnote{This no longer holds for $n=8$, as the algebras of Hilbert function $H(A)=(1,4,3)$ defined by a vector space of quadrics having dimension 7, form a generic point of a component  of the Hilbert scheme $\Hilb^8(\mathbb P^4)$: these correspond by duality to nets of quadrics in $\mathbb P^3$, having generically $\dim\Grass(10,3)-\dim \PGL(4)= 6$ parameters \cite[\S 2.2]{Em-I}. } He accomplished this without classifying up to isomorphism the local algebras of length 7. G. Mazzola's result implies that the punctual Hilbert scheme $\Hilb^n(\mathbb P^r)$ is irreducible for $n\le 7$ (see J. Emsalem's MSN review of \cite{Maz} and A. Iarrobino's MSN review of \cite{CN}).\par\vskip 0.2cm\noindent
{\it Structure, Smoothability}:
J. Emsalem's {\it G\'{e}om\'{e}trie des points \'{e}pais} \cite{Em} of 1978 concerning the topology of the punctual schemes, and the Macaulay duality, forms part of the background of this paper. Also, an interest in smoothability questions as in Section 7 arose from A. Iarrobino's \cite{I1} where it had been shown by a dimension argument that there are non-smoothable Artinian algebras. The number of irreducible components of the Hilbert scheme parametrizing
length-$n$ commutative unitary algebras had been studied in \cite{I2}. It took until \cite{Em-I} for a specific non-smoothable scheme to be given, corresponding to graded Artinian algebras of Hilbert function $H=(1,4,3,0)$, defined by seven general enough quadratic forms (see also \cite[\S 1.5.8]{Hart},\cite{CEVV}, and the discussion of smoothability below in Appendix~B).
\addcontentsline{toc}{section}{Appendix B. Update on related topics}

\subsection*{Appendix B. Update on related topics}\label{Appendix B}
In this section we give a survey of more recent work, since the original R\'{e}seaux of 1977, relating to pencils and nets of planar conics, and their associated algebras. \vskip 0.2cm\par\noindent
{\it Classification of pencils.}  This classification was already well known in 1977, as noted in Appendix A, even for $n$ variables. The 1983 article of A. Dimca \cite{Dim} gives for pencils $\pi$ of quadric hypersurfaces in $n$ variables, a geometric approach to their classification, showing that the class of a pencil is uniquely determined by that of its associated line, and the position of the line with respect to the rank subvarieties of quadrics, and the singularity set of the zeroes of $\pi$.  There are connections with Fano varieties: the thesis of M. Reid \cite{Reid}, articles by X. Wang \cite{Wang} and others, now include classification of pencils of conics in $\mathbb P^2$ in characteristic $p$. Finally we have the very recent survey for pencils of quadric hypersurfaces in $n$ variables by C. Fevola, Y.~Mandelshtam, and B. Sturmfels \cite{FMS}: they review what is known about these pencils in $n$ variables, in particular the classification using Segre pencils \cite{Seg}, \cite[\S XIII.10]{HP}. They also extend the classification over the reals to parametrizing Jordan type loci, and  other strata in the appropriate Grassmannian, with examples over 3 and 4 variables. They give applications to the maximum likelihood estimation for Gaussian distributions in algebraic statistics. \vskip 0.3cm\noindent
{\it Classification of nets of planar conics.}
 \par Appearing in 1977 was C.T.C Wall's complete and succinct classification \cite{Wall} of these nets over the complexes and over the reals.\footnote{J. Emsalem and A. Iarrobino saw the C.T.C. Wall preprint only after writing most of R\'{e}seaux de Coniques; the appearance of \cite{Wall} was a factor in leaving the present work in preprint status for so long.}\par
Since C.T.C.Wall's goal was the study of map germs of finite type up to source-image isomorphism, there is a different emphasis, and an efficient algebraic approach to classification of nets. By contrast, the R\'{e}seaux preprint \cite{Em-I1} is more geometric in style.\par
We also note that the classification in Table \ref{table1} of isomorphism classes of nets of planar conics over an algebraically closed field is valid in characteristic $ {\sf k}>3$; although the original R\'{e}seaux was written assuming $\sf k$ algebraically closed of characteristic zero. Indeed, this classification over algebraically closed fields of arbitrary characteristic except $2$ or $3$ is nicely explained in a recent posting by Naoto Onda \cite{Onda}, which uses an efficient analysis of the normal forms of the net having a given isomorphism type of discriminant, to classify the ten orbits for which the discriminant is singular. The context of \cite{Onda} is of classifying Artinian algebras of Hilbert function $(1,3,3)$ up to isomorphism - which, of course, is equivalent to classifying the equivalence classes of the nets of planar conics, a fact not noted by the author.  N. Onda refers to the following result of M. Bhargava and Wei Ho for the nonsingular case. They set $K$ to be an arbitrary field of characteristic not $2$ or $3$: $K$ need not be algebraically closed.\vskip 0.2cm\par\noindent
{\bf Theorem A1}\label{A1thm} (\cite[Theorem 5.8]{BhHo}).
Let $V_1$ and $V_2$ be 3-dimensional vector spaces over $K$.  Then the nondegenerate $\Gl(V_1)\times \Gl(V_2)$ orbits of $V_1\otimes \mathrm{Sym}^2(V_2)$ are in bijection with the isomorphism classes of triples $(C,L,P)$ where $C$ is a genus one curve over $K$, $L$ is a degree 3 line bundle on $C$ and $P$ is a nonzero 2-torsion point of the Jacobian 
$\Jac(C)(K)$.\vskip 0.2cm
The paper of  M. Domokos, L.M. Feh\'{e}r,  E. Rim\'{a}nyi \cite{DFR} extends the discussion of nets of conics over the complexes $\mathbb C$ to invariants and equivariants, in essence determining cohomology classes of the closures of orbits. The classes of nets of conics of codimension at least two (all but the top row of our Table \ref{table1}) are the $\Sigma^0$ portion of their Table \ref{table1} and 3, which include further invariants;  the corresponding portion of their Figure 2 has the specialization diagram of our Table \ref{table1}. They apply this to determine Thom polynomials of contact singularities of type $(3,3)$. Thus, in the course of their work with equivariant cohomology, they give an independent derivation of the specializations in our Table \ref{table1} or of \cite{Wall} over the complexes - not including the discussion in our Section 6.5 concerning specializations with $\lambda$ constant from the smooth orbits \#8b in our Table \ref{table1}.\footnote{We for short label the orbits in our Table \ref{table1} by the dimension and position from the left in each dimension, thus \#6b is the second orbit shown in the dimension 6 row.}\vskip 0.2cm\noindent
{\it Concordance of tables of nets of conics.}
A. Conca gave in the 2009 \cite[p.1568]{Co} a table of nets of conics with considerable further information, including whether the algebra $A=R/(V)$ is Koszul, whether the net is of gradient type (i.e. polar, or Jacobian), and he gives the Hilbert series for $R/(V)$. He notes that the point $(a:b:c)$ belongs to the locus defined by $V$ if and only if $(ax+by+cz)^2$ is in $V^\ast$. The duality is indicated in C.T.C Wall's notation by an asterisk, as $E$ vs. $E^\ast$. For the classification he used \cite{Wall} and our preprint.\par
\begin{table}[h]
\begin{center}
\caption{Concordance of tables for nets of conics from our Table 1 (EI), C.T.C. Wall \cite{Wall}, and A. Conca \cite[p.1567]{Co}.}
\label{concordtab}\vskip 0.2cm
$\begin{array}{|c||ccc|ccc|cccc|cc|c|cc|}\hline
EI&8a&8b&8c&7a&7b&7c&6a&6b&6c&6d&5a&5b&4&2a&2b\\\hline
Wall&B^\ast&A&B&D^\ast&C&D&E^\ast&F^\ast&F&E&G^\ast&G&H&I^\ast&I\\\hline
Conca&14&15&13&6&12&5&4&10&9&3&8&7&11&2&1\\\hline
\end{array}$\vskip 0.2cm
\vskip 0.2cm
\end{center}
\end{table}
Note that A. Conca in his column $p$ \cite[p.\,1568]{Co} gives the number of distinct points in the support of the punctual scheme (in all but case 2 where the support is a line), while our Table~\ref{lengthfig}, page \pageref{lengthfig} gives the total scheme length, including multiplicities: thus the entries are different. \par
Koszul properties of nets are discussed in A. Conca's \cite{Co} and in A.~D'Al\`{i}'s \cite{Dal} the latter handling characteristic $p$ cases.
\vskip 0.2cm\par
\noindent

\par\noindent
{\it Nets of conics, and Hessian of smooth cubics.}  The theme of our Note on the Hessian form of a smooth cubic is greatly developed by
M. Artebani and I.~Dolgachev in \cite{ArtDol}, who discuss a related elliptic fibration, two sextic curves and a $K$-$3$ surface associated. 
The article \cite{ATT} treats the Hessian pencil over all characteristics but three.  See also \cite{Fr} and \cite[\S 4]{CilOt}.
\vskip 0.2cm\par\noindent
{\it Classification of nets of conics over algebraically closed fields of small characteristic, or over finite fields $\sf k$}.
This classification in characteristic 2 for $\sf k$ closed is discussed in \cite{Bal}, see also \cite[Example 7.3.11]{DolKo}. We have mentioned in Appendix A as an update of Wilson's \cite{Wil} the very recent article by M. Lavrauw, T. Popiel, and J. Sheekey \cite{LPS} for a finite field of odd order $q$, finding 15 orbits of rank one, that is, which contain a conic that is a double line. C. Zanella studies nets of conics over finite fields
\cite{Zan}.  A. S. Sivatsi for an arbitrary field $k$ of charateristic $p$ not 2, shows that three planar quadrics $f_1,f_2,f_3$ have a common zero if and only if the quadratic form $u_1f_1+u_2f_2+u_3f_3$ over the field $K={\sf k}(u_1,u_2,u_3)$ is isotropic (has a $K$-rational point), and the cubic form $t_1f_1+t_2f_2+t_3f_3$ is isotropic over $\sf k$. 
\vskip 0.2cm\noindent
{\it Applications of the classification of nets of planar conics.}
Notable applications besides the study of Artinian algebras with Hilbert functions $H(A)=(1,3,3,\ldots)$, include the study of map germs from $(\mathbb C^3,0)\to (\mathbb C^3,0)$ (see \cite{DG,Wa,DFR,Rong}). The article \cite{DFR} applies the study of nets to determining the Thom  classes of related maps from $(\mathbb C^3,0)$ to $(\mathbb C^p,0), p\ge 3$  using their calculation of the $\Gl(3)\times $ Aut$ (\mathbb P^2)$ equivariant cohomology classes of the orbit closures of the nets. The study of the intersections of three conics by R. Feng and L. Shen, \cite{FS} includes a careful development of the scheme length we studied briefly in \S \ref{lengthsec}. There is a connection with Fano threefolds (\cite{AC-V}). A. Verra has used the classification of these nets as a step in characterizing certain Enriques surfaces that correspond to a fibered product of two planes $\mathbb P^2$ over $\mathbb P^2$ with the projection maps being two rational maps of degree four, each corresponding to a net of conics as in Section \ref{mapsec} \cite[Theorem 1,3]{Ver}: the result uses Proposition \ref{mapprop} to characterize the discriminant locus of a regular map, $\mathbb P^2$ to $\mathbb P^2$. F. Ronga \cite{Rong} studies topologically over $\mathbb C$ the four isomorphism classes of maps $(f_0,f_1,f_2): \mathbb C^3\backslash 0\to \mathbb C^3\backslash 0$. For reference on Enriques surfaces see \cite{CDL,DolKo}.\vskip 0.2cm\par\noindent
{\it Classification and specialization of planar cubics.}
Tables of singular cubics and their specializations in \cite[p. 52, Theorem 5.17]{Dim2} and \cite[Figure~1.13]{EH} agree with our Table \ref{cubicfig}, except for an extra arrow from node to triangle in the former, which is proscribed by our Lemma \ref{nospecialXYZ}.\footnote{The author communicated that this is an error.} The latter reference goes on to determine the degrees of the loci \cite[\S 1.2.5]{EH}.
Several authors have discussed the Waring problem for cubics in $\mathbb P^2$ - decomposing the cubic into the sums of powers, in particular over the reals, some making connections to nets of conics \cite{DolK,MMSV,Ott,Ott2}. In particular the Aronhold invariant $S=Ar(f)$ of a cubic $f$ vanishes if $f$ is equivalent to a sum of cubes (a Fermat cubic); it is written as a certain Pfaffian in \cite[\S 4.7]{Ott2}. I.~Dolgachev and V. Kanev show that $Ar(f)=0$ only if $f$ is a Fermat cubic, a cone, or a conic union a non-transversal line \cite[Thm 5.13.2]{DolK}: this is equivalent to being in the closure of the family of Fermat cubics.  As noted earlier in Section \ref{specjsec} the invariants $S,T$ of a ternary cubic were given by G. Salmon \cite{Sal}. For a modern derivation of these explicit formulas see \cite[Theorem 4.4.6, Proposition 4.4.7, Example~4.6.3]{Stu}. For further study of cubics and their invariants see  Ch. Okonek and A. Van de Ven \cite[Examples 6,8, \S 5.3]{OkV}, \cite[\S 10.3]{Dol1} (classifying plane cubics over algebraically closed fields of arbitrary characteristic),  H. R. Frium \cite{Fr} and \cite[Chap. 3]{Dol}.
\vskip 0.2cm\par\noindent
{\it Determinantal nets of conic and reciprocal nets}. See also \cite{Tj} for an exposition obtaining invariants for a compactification of the determinantal nets of conics.
Analogous to the article \cite{FMS} for pencils, S. Dye, K. Kohn, F.~Rydell, R. Sinn have applied the classification of nets of conics over $\mathbb C$ to obtain maximum likelihood estimates for a Gaussian distribution \cite{DKRS}. They use the \emph{reciprocal surface} of an invertible net (one with nonsingular symmetric matrices),  a rational variety closure of the inverse images of nonsingular matrices in the net, which they show is an image of the Veronese. Their work uses the intersection of the reciprocal surface with the polar net. 
\vskip 0.2cm\par
\noindent
{\it Classification of commutative local algebras for $n\le 8$}.
A. Iarrobino wrote in 1985, ``There is a virtually complete picture of isomorphism classes of Artinian algebras and their deformations for $n \le 8$, but that picture has never been written down in one place \cite{I3}.''  To explain this comment, for lengths $n\le 6$ we had G. Mazzola \cite{Maz}, and for earlier authors see \cite{Hap}. For codimension two and small $n\le 5$ see the actual thesis of J. Brian\c{c}on. For the case codimension three and length $n= 7$, for $H=(1,3,3)$, see this article in \S 7.1; for $H=(1,3,4)$ consider the Macaulay duals of pencils of conics (this article in \S 7.2). Various classifications of spaces of quadrics in $r$ variables were given by D.A. Suprunenko and R. I. Ty\v{s}kevi\v{c} in the 1966  book \cite{SupTy}. \par A succinct derivation of the isomorphism classes of commutative local algebras over an algebraically closed field in lengths up to six is given by B.~Poonen in \cite{Po1} [the web-posted update has a correction in the characteristic 2 case].  The case of $n=6$ over $\mathbb C$, corresponding to length 5 nilpotent algebras is also in \cite{KRS}, there obtained by quite different methods involving central extensions of algebras, and given in a different language.\footnote{See \cite[Remark 1.1]{Po1} for some further comparison with the classifications of nilpotent algebras, and historical references. I. Kaygorodov in a note to the authors pointed out that the algebra ${\sf k}^6$ of \cite{Po1}, corresponding to a nilpotent algebra $0$ has no analogue in the classification of length-5 nilpotent algebras \cite{KRS}, thus a difference in the number of classes listed in \cite{Po1,KRS}.} See also the 1969 \cite{Pav} for some other cases specifying socle degree (class) and the number of variables $n$. \par
Graded algebras with Hilbert function $H=(1,3,1,1,1)$ or $(1,4,1,1)$ are special and easier to classify because of the tail including $(1,1)$ which requires that $I_{j-1}=\ell R_{j-2}$ for some linear $\ell$ (by G. Gotzmann's theorems). The classification of those having Hilbert function $(1,3,2,1)$ will follow from the classification of 
4-spaces $V\subset R_2$ of conics, and the very restrictive condition that $\dim R_1V=1$ - allowing an approach using the cubic $F$ in an inverse system to $A$. Classifying Gorenstein algebras $R/I$ of  Hilbert function $H=(1,r,1)$ is essentially the classification of the conic perpendicular in the Macaulay duality to $I_2$.  \par Thus, the lack of specific writing down of the classifications mentioned in the 1985 \cite{I3}, has since been somewhat remedied, with relevant articles by B.~Poonen again taking up the classification for $n\le 6$ \cite{Po1}, and with a whole new direction of research showing that Gorenstein Artinian algebras of lengths up to 13 are smoothable (see \cite{CN,CJN}). The article  \cite{Po2} gives an asymptotic formula for the dimension of the family of length-$n$ Artinian algebras with a  fixed basis. \par
A subtle issue when the socle degree is at least 3 is whether the classification of local algebras of a given Hilbert function is the same as the classification of graded algebras of that Hilbert function: is each $A$ canonically graded? Articles by J. Elias and M. Rossi \cite{ER},  J. Jelisiejew \cite{Je}, and others \cite{CENR} go deeply into the connections between local and graded Artinian algebras, for example showing that $H=(1,r,r,1)$ algebras are canonically graded \cite[Theorem 3.3]{ER}  cha $ \sf k=0$), \cite[Proposition 1.1]{Je}. \vskip 0.2cm\par\noindent
{\it Smoothability.} We define elementary components of the punctual Hilbert scheme $\Hilb^n(\mathbb P^r)$, to be those irreducible components that parametrize finite length punctual schemes concentrated at a single point of $\mathbb P^r$: every irreducible component is made up of a concatenation of elementary components, concentrated at different points of $\mathbb P^r$ \cite{I2,Je2}.  The deformation space of a non-smoothable length-$n$ punctual scheme must contain a scheme of possibly smaller length having an elementary component. As mentioned in Appendix~A, the first specific example noticed of an elementary component was the graded algebras of Hilbert function $H(A)=(1,4,3,0)$ defined by seven general enough quadratic forms \cite{Em-I}, shown using a ``small tangent space'' method, where the available tangent space is matched to elements of the family.\footnote{We showed that $A=R/I$ having only trivial negative one tangents in $\Hom(I,R/I)$ implies that all higher order negative tangents are zero, which suffices to show that $A$ is elementary. We relied on what we had learned from M. Schlessinger. J. Jelisiejew reproves this in an interesting way \cite[Theorem 4.5]{Je2}; his article does survey - and advance - much of the current state of the art of finding elementary components.  The simpler example  $A={\sf k}[a, b, c, d]/(a^2, ab, b^2, c^2, cd, d^2, ad- bc)$, also of Hilbert function $H=(1,4,3,0)$ and having no nontrivial negative tangents is reported in \cite[Prop. 9.6]{Po2}.}  I. Shafarevich using new methods greatly generalized this result: he showed that for algebras $A={\sf k}[x_1,\ldots,x_r]/I$ of Hilbert function $(1,r,n,0), r\ge 4$, that for $2<n\le\frac{(r-1)(r-2)}{6}+2$ (a third of the available values) the generic algebra of Hilbert function $(1,r,n,0)$ determines an elementary component of $\Hilb^{n+r+1}\mathbb P^r$. For large $n$, about a third of the available values, all the algebras in $\Hilb^{n+r+1}\mathbb P^r$ are smoothable - that is, they are in the closure of the separable algebras, those isomorphic to ${\sf k}^{1+n+r}$; for the remaining third of the values the question of smoothability remains open \cite{Shaf1}.\par Further study of the ``principal component'' of the punctual Hilbert scheme, parametrizing smoothable algebras is in \cite{RySk}; and the usually reducible scheme $\mathcal B$ parametrizing non-smooth but smoothable algebras is studied in \cite{Shaf2}.  Distinct approaches to smoothability are given by D. Erman and M. Velasco, who introduce a finer, syzygetic ``k-vector'' obstruction to smoothability in \cite{EV}, by G. Casnati and R.~Notari in \cite{CN}, and by J.~Jelisiejew \cite{Je2}. M.~Szachniewicz shows that $\Hilb^{13}(\mathbb A^6)$ is non-reduced at a point corresponding to a $H(A)=(1,6,6)$ algebra \cite{Sz}.  A striking and surprising result (counterexamples to Conjecture 1 of \cite{Shaf2} and to an erroneous \cite[Proposition A.5]{Em-I}) was that the length-$n$ singular schemes have in general codimension only one in the family of smooth schemes in $\mathbb P^r$ when both $r\ge 11$, and $n\ge 15$ \cite[Proposition 1.6]{EV}: they construct these by adding smooth points and lifting, from a single basic example of the smoothable local algebras concentrated at a single point  $0\in \mathbb A^{11}$ and satisfying $H(A)=(1,11,3)$ - a family they show has dimension 153 (so dimension 164 as a family fibered over $\mathbb A^{11}$). \par Substantial work of  G.~Casnati and R. Notari and others concerning Gorenstein algebras of lengths no greater than thirteen have resulted in showing that they are smoothable \cite{CN, CN2,CJN} while the $H=(1,6,6,1)$ generic Gorensteins form an elementary component \cite{Em-I}. Surprisingly the $H=(1,7,7,1)$ Gorensteins have been shown to be smoothable using new methods of marked bases (C. Bertoni, F. Cioffi, M. Roggero \cite{BCR}), while a dimension calculation shows that  generic Gorensteins of Hilbert function $H=(1,r,r,1)$ are not smoothable for $r\ge 8$ (\cite[p.152]{Em-I}). R.~Szafarczyk has recently shown that these generic Gorensteins give elementary components of the Hilbert scheme
$\Hilb^n(\mathbb P^{r-1})$ \cite{Sza}. The small-tangent space method \cite{Em-I} suggests there are similar elementary components in codimension $r$ four and larger, determined by suitable-dimension vector spaces of degree-$j$ forms, $j\ge 3$: however, the nub is, given a suitable triple $(r,j,d )$ to verify that the tangent space in at least one example has the conjectured small dimension.  \par Other elementary components of the punctual Hilbert scheme had been found by M.~Huibregtse \cite{Hui}: these, strikingly, are comprised of essentially non-homogeneous algebras.  M. Satriano and A.~P.~Staal have recently shown that there are an infinite set of choices $d$ for which $\Hilb^d(\mathbb A^4)$ has an elementary component of very small length, less that $3(d-1)$ \cite{SatSt}, answering a question asked in \cite[p.186]{Em-I}. \par\vskip 0.2cm
\noindent
{\it Classification of isomorphism types of Artinian Gorenstein (AG) algebras having Hilbert function $H=(1,3^k,1), k\ge 2$.}
This is closely related to the study of length 3 punctual subschemes of $\mathbb P^2$. The article of A. Bernardi et al \cite[\S 4.2, Theorem 37]{BGI} classifies the isomorphism types of length 3 subschemes of $\mathbb P^n$ in characteristic zero; when $n=2$ this is closely related to the classification of nets of ternary conics: the scheme-length 3 cases of Table~\ref{lengthfig} and classifying also those length-3 schemes that are on a line will give this classification. G. Casnati, J. Jelisiejew, and R. Notari determine for arbitrary characteristic the three isomorphism classes of forms $F\in {\sf k}[X,Y,Z]_j$ whose apolar algebra $A=R/\mathrm{Ann} F$ has Hilbert function $H(A)=(1,3,\ldots, 3^{j-1},1_j) $ when $j\ge 4$ \cite[Proposition 4.9]{CJN}. This can be derived also from the list of scheme-length-3 nets of conics in Table \ref{lengthfig}.
J. Jelisiejew \cite{Je} determines the isomorphism types of AG algebras $A$ of Hilbert function $H(A)=(1,3,3,3,1)$ including those for non-graded $A$. In a sequel using the current ``Nets of conics'' paper, the authors with J. Yam\'{e}ogo determine the isomorphism types of graded Artinian Gorenstein algebras $A$ of Hilbert function $H(A)=(1,3^k,1)$ for all $k\ge 2$; and also show that the closure of this family is all the graded algebras $G_T$ of this Hilbert function \cite{AEIY}. This is unlike the case for other Gorenstein sequences, such as $T^\prime=(1,3,5,3,1)$ where $G_T$ has several irreducible components, one the closure of $\Gor(T^\prime)$.
\par\vskip 0.2cm\noindent
{\it Classification of Jordan types for pairs $(A,\ell), A\in \Gor (H) $.} The Jordan type $P_{A,\ell}$ for a pair $(A,\ell)$ is the partition of length $|A|$ giving the Jordan block decomposition of multiplication $m_\ell$ on $A$. For a generic $\ell$ the graded AG Gorenstein algebras of Hilbert function $H=(1,3^k,1)$ are strong Lefschetz \cite{AA-Y}; for an arbitrary linear $\ell$ the partitions $P_{A,\ell}$  are unknown. \par\vskip 0.2cm\noindent
We remark that the book ``Classical Algebraic Geometry''  \cite{Dol} by I. Dolgachev, and the books 
 ``Enriques Surfaces, I'' \cite{CDL} with F. Cossec and C. Liedtke,  ``Enriques Surfaces, II'' \cite{DolKo} with S. Kondo give many further results pertaining to the geometry of nets of conics, nets of quadrics, and have extensive bibliographies including many classical as well as modern results in related areas.

\addcontentsline{toc}{section}{Bibliography, 2022}
\renewcommand{\refname}{Bibliography, 2022}
\small

\vskip 0.4cm
{\sf e-mail}: \par
 nancy.abdallah@hb.se\par
a.iarrobino@northeastern.edu
\end{document}